\documentclass[11pt]{article}
\usepackage{latexsym,color,amsmath,amsthm,amssymb,amscd,amsfonts}
\usepackage[labelsep=period]{caption}
\usepackage{tikz}
\usetikzlibrary{calc}
\usetikzlibrary{shapes.geometric}
\usepackage{pgfplots}
\usepackage{cite}
\usetikzlibrary{arrows,calc}
\tikzset{
>=stealth',
help lines/.style={dashed, thick}, axis/.style={<->}, important
line/.style={thick}, connection/.style={thick, dotted}, }

\def\Z{\mathbb{Z}}
\def\N{\mathbb{N}}
\def\R{\mathbb{R}}
\def\RR{\mathbb{R}}

\def\eps{\varepsilon}
\def\cE{\mathcal E}
\def\conv{{\rm Conv}}
\newcommand{\iprod}[2]{\langle #1,#2 \rangle} % Inner product %
\newcommand{\del}{\partial}

\newtheorem{lemma}{Lemma}[section]
\newtheorem{proposition}[lemma]{Proposition}
\newtheorem{remark}[lemma]{Remark}
\newtheorem{theorem}[lemma]{Theorem}
\newtheorem{definition}[lemma]{Definition}
\newtheorem{corollary}[lemma]{Corollary}

\newtheorem{cor}[lemma]{Corollary}

\newtheorem*{remark*}{Remark}

\newcommand{\cL}{{\cal L}}
\newcommand{\cP}{{\cal P}}
\DeclareMathOperator{\length}{Length}
\DeclareMathOperator{\Per}{Per}

\begin{document}
\title {Duality of Caustics in Minkowski Billiards }
\author{S. Artstein-Avidan, D. Florentin, Y. Ostrover, D. Rosen}
\date{}
\maketitle
\begin{abstract}
In this paper we study convex caustics in Minkowski billiards. We show that for the Euclidean billiard dynamics in a planar smooth, centrally symmetric, strictly convex body $K$, for every convex caustic which $K$ possesses, the ``dual" billiard dynamics in which the table is the Euclidean unit ball and the geometry that governs the motion is induced by the body $K$, possesses a dual convex caustic. Such a pair of caustics are dual in a strong sense, and in particular they have the same perimeter, Lazutkin parameter (both measured with respect to the corresponding geometries), and rotation number. We show moreover that for general Minkowski billiards this phenomenon fails, and one can construct a smooth caustic in a Minkowski billiard table which possesses no dual convex caustic.
\end{abstract}

\section{Introduction and Main Results } \label{sec:Int}

Mathematical billiard models have been studied in various contexts, including dynamical systems, ergodic theory, statistical mechanics, and geometry. Euclidean billiard dynamics concern the propagation of a point particle with no mass in some domain in ${\mathbb R}^n$, called the ``billiard table".
The particle moves in straight lines until it encounters the boundary. There, it reflects specularly, and bounces off according to the classical reflection law: the angle of incidence equals the angle of reflection.
This law is the consequence of the following variational principle: for two points $a,b$ in the interior of the billiard table, the reflection point $x$ on the boundary which is part of the billiard orbit between $a$ and $b$, is a critical point for the Euclidean length of the broken line $\overline {axb}$ (see e.g.,~\cite{KT} and~\cite{T}).

The notion of {\it ``Minkowski billiard"} was introduced by Gutkin and Tabachnikov in~\cite{GT}, as an important special case of the natural extension of Euclidean billiard dynamics to the Finsler setting.
In the Minkowski case, the geometry that governs the billiard dynamics is determined by a norm on ${\mathbb R}^n$.
% which is given by the so called support function $h_T$ associated with a convex body $T$.
The reflection law for the associated Minkowski billiard follows from a variational principle analogous to the one  mentioned above, where the Euclidean length is replaced by this norm. % and further discussions).
It is sometimes useful to think of the Minkowski billiard dynamics as
a dynamical system associated with a pair of convex bodies $K \subset {\mathbb R}^n_q$, and $T \subset {\mathbb R}^n_p$, where the ambient space ${\mathbb R}_q^n \times {\mathbb R}_p^n$ is the classical phase space. 
%$(K,T) \subset {\mathbb R}_q^n \times {\mathbb R}_p^n$.
 Here, one of the bodies, say $K$, plays the role of the billiard table, and the other body $T$ determines a norm, given by the so-called support function $h_T$, which controls the billiard dynamics in $K$ (see Section~\ref{sec-Minkowski-billiards} below for the precise formulation).
The $(K,T)$-billiard dynamics takes place on the boundary $\partial (K \times T)$, and the projections of
the associated orbits to $ {\mathbb R}^n_q$ (respectively  to $ {\mathbb R}^n_p$) are called $T$-billiard trajectories in $K$ (respectively $K$-billiard trajectories in $T$). When $T=B$ is the Euclidean unit ball, the $B$-billiard trajectories in $K$ are the classical billiard orbits in $K$. With this description, it follows immediately  that  there is a one-to-one correspondence between $T$-billiard trajectories in $K$, and $K$-billiard trajectories in $T$.
%(Remark~\ref{rmk:Role-of-K-T-interchangable} below).
%Furthermore, one can easily check that for closed billiard trajectories, the length with respect to $h_T$ of a $T$-billiard trajectory in $K$ equals the length with respect to $h_K$ of the corresponding $K$-billiard trajectory in $T$.
In what follows we call such a pair of trajectories ``dual billiard trajectories" (see Section~\ref{Subsection-background-Minkowski-Billiards}).

Motivated by this orbit-to-orbit duality, in this paper we study a
certain caustic-to-caustic duality for planar Minkowski billiards.
Recall that for Euclidean billiards, a convex {\it caustic}
is a convex body with the property that once a billiard trajectory is tangent to it, it remains tangent after every reflection at the boundary. Caustics play an important role in the study of Euclidean billiard dynamics, and their existence has an essential impact on the dynamics (see Section~\ref{sec-historical-background} below for more details).
 The definition of a convex caustic in the Minkowski billiard setting, at least when the body $T$ which induces the geometry is smooth, centrally symmetric, and strictly convex, is the same as in the Euclidean case, i.e., a $T$-caustic in $K$ is a convex body in ${\mathbb R}^n_q$ such that every $T$-billiard trajectory once tangent to it, stays tangent after every reflection at $\partial K$. Similarly, a $K$-caustic in $T$ is a convex body 
 in ${\mathbb R}^n_p$  with the property that every $T$-billiard trajectory once tangent to it, remains tangent to it after every reflection with $\partial T$.
In the non-centrally symmetric case, one must distinguish between `left' and `right' caustics, which for simplicity we wish to avoid in this text. Thus, we introduce the following terminology for planar Minkowski billiards.

 \begin{definition} A ``symmetric billiard configuration" is a pair $(K,T)$ % \subset {\mathbb R}^2_q \times {\mathbb R}^2_p$ 
 of $C^1$-smooth centrally symmetric strictly convex bodies, where $K \subset {\mathbb R}^2_q$, and $T \subset {\mathbb R}^2_p$.
\end{definition}
% In the planar case, for a pair
% $(K,T) \subset {\mathbb R}^2_q \times {\mathbb R}^2_q$, where $K$ is smooth and convex, and $T$ is smooth, centrally symmetric, and strictly convex,
For a symmetric billiard configuration $(K,T)$
 we denote by
${\mathfrak C}(K,T)$ the set of all the convex $T$-caustics in the billiard table $K$. The set ${\mathfrak C}(T,K)$ is defined in a similar manner.
%{\bf A (planar) ``Billiard table" is a  $C^1$-smooth strictly convex body in ${\mathbb R}^2$.}

\begin{definition}[{\bf Caustic Duality}]   \label{def:caustics-duality}
Let $(K,T)$ be  a symmetric billiard configuration.
%two
%smooth, centrally symmetric, strictly convex bodies.
		Two convex caustics $C \in {\mathfrak C}(K,T)$ and $C' \in {\mathfrak C}(T,K)$
	are called ``dual caustics" if
	for every $T$-billiard trajectory in $K$ which is tangent to $C$, %(and thus remains tangent to it after every reflection),
	the dual $K$-billiard trajectory in $T$ is tangent to $C'$, and vice versa.
\end{definition}

\begin{remark} \label{rmk-dual-caustics-uniquness-and-parameters}
{\rm
We emphasize that the dual convex caustic to $C \in {\mathfrak C}(K,T)$, if exists, is unique (see Section~\ref{sec:caustics-invariant-circles}).
Moreover, we shall see below that several parameters naturally  associated  with a convex caustic $C \in {\mathfrak C}(K,T)$, such as  its rotation number, perimeter, and the so-called Lazutkin parameter (where the last two quantities are defined via the corresponding support function $h_T$), are the same for any pair of smooth strictly convex dual caustics (see %Proposition~\ref{prop:dual-invariants} in 
Section~\ref{Sec:dual-caustics-classical-invariants} below for the definitions and more details).}
\end{remark}

Our main result, concerning dual caustics for Euclidean billiards, is the following.

\begin{theorem} \label{thm:main}
Let $K \subset {\mathbb R}^{2}_q$ be a $C^1$-smooth centrally symmetric strictly convex body, and $B \subset {\mathbb R}^{2}_p$ be the Euclidean unit disk.
Then for every convex $B$-caustic in $K$ there exists a dual convex $K$-caustic in $B$.  Moreover, 
the two caustics have the same perimeter, Lazutkin parameter (both measured with respect to the corresponding geometries), and rotation number. 
\end{theorem}

\begin{remark}\label{rmk:dual-caustic-formula} {\rm
The dual caustic  mentioned in Theorem~\ref{thm:main} has a simple geometric description. More precisely, let $C $ be a $B$-caustic in $K$. For $q \in \partial K$, denote by $e(q)$ and $b(q)$ the positive and negative tangency points to $C$ from $q$, respectively (see Figure~\ref{fig-dual-caustics} below). Moreover, let $L(q) = |q-e(q)| + |q-b(q)|$. Then, the dual caustic $C' \in {\mathfrak C}(B,K) $ is given by
$$
C' = \conv\left\{ \frac{e(q)-b(q)}{L(q)} \,:\, q \in \partial K\right\}.
$$
In particular, if $C$ is a polygon, then so is $C'$. On the other hand, if $C$ is sufficiently smooth, then $C'$ is smooth, and its boundary is parametrized by $q \mapsto  \frac{e(q)-b(q)}{L(q)}$, where $q \in \partial K$.
}
\end{remark}

The notion of caustics for planar convex billiard tables is closely related with the notion of an `invariant circle' of the associated monotone twist map (see Section~\ref{sec:caustics-invariant-circles} below). In particular, any convex caustic gives rise to such an invariant circle.
It follows immediately from the definition that for any symmetric billiard configuration $(K,T)$,
there is a natural bijection between the invariant circles of the twist map associated with the $T$-billiard dynamics in $K$, and the invariant circles of the twist map associated with the $K$-billiard dynamics in $T$. However, in general, not every invariant circle corresponds to a \emph{convex} caustic  (see e.g.,~\cite{GK} and~\cite{Kni}).
Thus, the existence of a convex dual caustic in Theorem~\ref{thm:main} above is not obvious. Moreover, as Theorem \ref{thm-counter} below demonstrates,
this result fails in general if one replaces the Euclidean ball $B$ with an arbitrary convex body.

\begin{theorem}\label{thm-counter}
	There exists a  symmetric billiard configuration  $(K,T) $ and a smooth strictly convex  $T$-caustic in $K$ which possesses no dual convex caustic, that is, there exists no $K$-caustic in $T$ which is dual to it in the sense of Definition~\ref{def:caustics-duality} above.
\end{theorem}

%In light of the results in Theorems~\ref{thm:main} and~\ref{thm-counter} above, it is tempting to speculate that one can characterize the Euclidean case among Minkowski billiards via the property of having dual covnex caustics. More preciesly, 
%\begin{ques} If for a symmetric billiard configuration $(B,K)$ where $B$ is a Euclidean ball and $K$ is 
%... if every $C \in {\mathfrak C}(B,K)$ has a dual (convex) $C' \in {\mathfrak C}(K,B)$ then $K$ must be an ellipse. 
%\end{ques}

\noindent{\bf Notations:} Throughout the paper the word ``smooth" means $C^1$-smooth, unless explicitly stated otherwise. For a smooth function $F : {\mathbb R}^n \to {\mathbb R}$ we write $\nabla F$  for its gradient.
A convex body $K$ in ${\mathbb R}^n$  is a compact convex set with non-empty interior. It is said to be `strictly convex' if its boundary contains no line segment, and `centrally symmetric' if it is symmetric with respect to the origin, i.e., $K = -K$. For a centrally symmetric convex body $K \subset {\mathbb R}^n$, the gauge function $g_K : {\mathbb R}^n \to {\mathbb R}$ given by $g_K(x) = \inf \{ r \geq 0 \, : \, x \in r K \}$ defines a norm in ${\mathbb R}^n$, denoted by $\| \cdot \|_K$.
The support function $h_K : {\mathbb R}^n \to {\mathbb R}$ of a convex body $K$ is defined by $h_K(u) = \sup \{ \langle x,u \rangle  :  x \in K \}$. If $K$ is centrally symmetric, then one has $h_K(u)=\|u\|_{K^{\circ}}$, where $K^{\circ} = \{ y \in {\mathbb R}^n  :  h_K(y) \leq 1 \}$ is the dual (polar) body of $K$. The Euclidean unit ball is denoted by $B^n$ (when there is no risk of confusion, we sometimes drop the superscript $n$ to simplify notation).  We denote by $d_{H}(\cdot,\cdot)$ the Hausdorff distance between compact sets in ${\mathbb R}^n$. The convex hull of the sets $\{A_i\}_{i=1}^m$ in ${\mathbb R}^n$ is denoted by $\conv(A_1,A_2,\dots,A_m)$. An oriented line $\ell$ in ${\mathbb R}^n$ is denoted by $(x,v)$, where $x \in \ell$ and $v$ is the orientation of $\ell$. We denote by $\bar \ell := (x,-v)$ the same line with the reversed orientation. Given a planar convex body $C \subset {\mathbb R}^{2}$, an oriented line $\ell = (x,v)$ is said to be positively tangent to $C$ if $\ell$ and $C$ intersect, and $C$ is contained in the closed left half-plane determined by $\ell$, i.e.,
$C \subset \{ y  \in {\mathbb R}^2  :  \det(v,y-x) \geq 0 \}$. The oriented line $\ell = (x,v)$ is said to be negatively tangent to $C$ if the oriented line $\bar \ell$ is positively tangent to $C$. For a  centrally symmetric convex body $K \subset {\mathbb R}^2$ and a (rectifiable) curve $\gamma \subset {\mathbb R}^2$, we write ${\rm Length}_{h_K}(\gamma)$ for the length of $\gamma$ measured with respect to $h_K$ (see e.g.,~\cite{Tho}), and ${\rm Length}(\gamma)$ for the Euclidean length. Finally, we denote the perimeter of a planar convex set $C$ with respect to $h_K$ by $\Per_{h_K}(C)$, and the Euclidean perimeter by $\Per(C)$.

\noindent{\bf Organization of the paper:}
In Section \ref{sec:Prelim+Backgroupnd} we provide the necessary background on Minkowski billiards, and present two models for the dynamics: a continuous one via the characteristic foliation in the classical phase space, and a discrete-time model via a Poincar\'e surface of section.
Then, confining ourselves to the planar case, we discuss Minkowski caustics and their relations with invariant circles of the corresponding monotone twist map, and introduce some natural invariants associated with them. In Section~\ref{sec:existence-of-dual-caustic} we prove Theorem~\ref{thm:main} on the existence of dual caustics for Euclidean billiard dynamics. This is done first for sufficiently regular caustics, and then, using an approximation argument, for the general case. In Section~\ref{sec:non-existence-of-dual-caustic} we show that the situation is radically different for arbitrary Minkowski billiards, and prove Theorem~\ref{thm-counter}. Finally, in the Appendix we sketch an alternative proof of Theorem~\ref{thm:main} in the case where the caustic $C$ 
%\in {\mathfrak C}(K,B)$ 
is a convex polygon. 

%In Section~\ref{sec:proof-monotonicity-lemma} we prove one of the key lemmas from the proof of Theorem~\ref{thm:main}, and in the Appendix we provide proofs for several technical lemmas used throughout the text.

\noindent{\bf Acknowledgment:} The first named author is partially supported by ISF grant number 665/15.
%D.F is partially supported by 
The third named author is partially supported by the European
Research Council (ERC) under the European Union Horizon 2020 research and innovation
programme, starting grant No. 637386, and by the ISF grant No. 1274/14.
The fourth named author is partially supported by the European Research Council advanced grant 338809.

\section{Preliminaries and Background} \label{sec:Prelim+Backgroupnd}

\subsection{Some Historical Background} \label{sec-historical-background}

As mentioned in the introduction, for a Euclidean billiard table, a convex caustic\footnote[2]{The term ``caustic" (which means burning in Greek) is borrowed from optics, where it means a curve on which light rays focus after being reflected off a mirror.}  is a convex body with the property that every billiard trajectory once tangent to it, remains tangent after every reflection at the boundary. For example, it follows from elementary geometrical considerations that elliptical planar billiard tables have a 1-parameter family of convex caustics given by confocal ellipses. This phenomenon is
% in some sense a geometric manifestation of
closely related to the classical fact that the billiard dynamics inside an ellipse is an integrable dynamical system (see e.g.,~\cite{KT} and~\cite{T}).
%The existence of convex caustics is non-typical. % (see e.g.,~\cite{Gru}).
%Gruber~\cite{Gru} showed that generically, in the Hausdorff topology on the set of planar convex billiard tables, there are no convex caustics.
%Moreover, 
For convex billiard tables of dimension at least three,
Berger~\cite{Ber} and Gruber~\cite{Gru1}  proved that for Euclidean billiards only ellipsoids have convex caustics.
On the other hand, convex caustics do play an important role in the study of planar Euclidean billiards.
For example, the presence of a convex caustic implies that the dynamics cannot be ergodic, as the corresponding invariant curve (family of rays tangent to a caustic -- see Section~\ref{sec:caustics-invariant-circles} below) separates the phase space into invariant components. Caustics are also closely related with the integrability of the billiard dynamics (see e.g.,~\cite{Bi} and~\cite{Gu}). 
%On the other hand, the absence of convex caustics implies that there are orbits which exhibit a very different past and future behaviour. More precisely, recall that a billiard trajectory is said to be positively (respectively negatively) $\varepsilon$-glancing if for some bounce the angle of reflection with the positive (respectively negative) tangent vector is smaller than $\varepsilon$.
%Mather proved~\cite{Mat} that if there are no convex caustics inside a smooth planar convex billiard table then there exist infinitely many billiard trajectories that are both positively and negatively $\varepsilon$-glancing for any $\varepsilon >0$.
Finally,  in the theory of semiclassical approximations, convex caustics can be used to estimate eigenvalues, and to construct quasimodes for the corresponding Dirichlet problem (see~\cite{Laz1}).

For 2-dimensional Euclidean billiards, the existence of (uncountably many) caustics near the boundary of a sufficiently smooth convex billiard table with positive curvature  was
first proved by Lazutkin~\cite{Laz} using KAM theory. Lazutkin's original proof assumes the existence of 553 continuous derivatives of the billiard table.
Later, %using results of Russman and Herman, 
Douady~\cite{Dou}
reduced the degree of differentiability to six. In contrast, Mather~\cite{Mat} showed that if the curvature
of the boundary vanishes at one point, then the billiard possesses no caustics at all.
A quantitative version of this result was given by Gutkin and Katok in~\cite{GK}.  Hubacher~\cite{Hub} proved that no caustics exist
near a boundary whose curvature is discontinuous at some point.
For further information on caustics of Euclidean planar billiard tables we refer the reader to~\cite{GK, Gu,Kni, Laz1,T,T1}.

In~\cite{GT}, Gutkin and Tabachnikov made a first step in investigating caustics of Minkowski billiards. In particular, they extended the above mentioned result of Mather to this setting. More precisely, they showed that planar convex billiard tables whose Minkowski curvature is not strictly positive do not have caustics. To the best of our knowledge, analogous results in the Minkowski setting to the works of Lazutkin~\cite{Laz} on the existence of caustics near the boundary, of Gutkin-Katok~\cite{GK} on caustic-free regions, and of Berger~\cite{Ber} and Gruber~\cite{Gru1} regarding the non-existence of caustics in higher dimensions, are yet to be explored.

\subsection{The Minkowski $(K,T)$-Billiard Dynamics} \label{sec-Minkowski-billiards}

Minkowski billiards
%were introduced and studied in
% They
model, for example, the propagation of light in an anisotropic homogeneous medium, i.e., where the velocity depends on the direction, but is independent of the position (see~\cite{GT}, and also~\cite{JR} for a motivation coming from the study of light patterns observed in liquid crystal layers).
Below we present two models for Minkowski billiards. The first is a continuous description given by %(the projection of) a Hamiltonian flow using
characteristic foliations on non-smooth convex hypersurfaces in the classical  phase space ${\mathbb R}^{2n} = {\mathbb R}^n_q \times {\mathbb R}^n_p$ as described in~\cite{AAO}.
The second model, which essentially goes back to Birkhoff~\cite{Bir} (and is somewhat more common in the study of billiard dynamics) is a time-discrete dynamical system which can be derived from the first one using Poincar\'e sections. Before we turn to the precise formulation, we give first a non-formal geometric description of the $(K,T)$-billiard dynamics.

%\subsubsection{Geometric Description} \label{sec:geometric-description}

As mentioned in the introduction, the Minkowski $(K,T)$-billiard is a dynamical system associated with a pair of convex bodies, $K \subset \R^n_q$ and $T \subset \R^n_p$. %$(K,T) \subset {\mathbb R}_q^n \times {\mathbb R}_p^n$.
%Usually, one can consider one of the bodies, say $K$, as playing the role of the ``billiard table", while the other body $T$ induces the geometry that governs the billiard dynamics in $K$.
%However, as we shall see below, the roles of $K$ and $T$ are interchangeable, and one can relate the billiard dynamics in $K$ with respect to the geometry induced by $T$ with the billiard dynamics in $T$ with respect to the geometry induced by $K$.
We shall assume for simplicity that %$(K,T)$ form a symmetric billiard configuration. 
$K$ and $T$ are $C^1$-smooth convex bodies. 
The $(K,T)$-billiard dynamics takes place on the boundary $\partial (K \times T)$, and can geometrically be described as follows (see Section~\ref{Subsection-background-Minkowski-Billiards} below):
suppose we start at some point $(q_0,p_0) \in K \times \partial T$.
%When we follow the vector field of the dynamics,
The dynamics is defined so that we move in $K \times \partial T$ from $(q_0, p_0)$
to a point $(q_1, p_0) \in \partial K \times \partial T$, following the inner normal to the boundary $\partial T$ at the point $p_0$. When we hit the boundary $\partial K$ at the point $q_1$
%, the vector field changes, and
we start to move in
$\partial K \times T$ from $(q_1, p_0)$ to $(q_1, p_1) \in \partial K \times \partial T$ following the outer normal to $\partial K$ at the point $q_1$ (see Figure~\ref{fig-Minkowksi-billiard}). Next, we move from
$(q_1, p_1)$ to $(q_2, p_1)$ following the inner normal to $\partial T$ at $p_1$, and so on and so forth.
%We refer to the projection of this dynamics to $K$ (respectively $T$) as $T$-billiard trajectories in $K$ (respectively $K$-billiard trajectories in $T$).
Similarly to the classical Euclidean case, $T$-billiard trajectories in $K$ correspond to critical points of a length functional defined on the (infinite) cross product of the boundary $\partial K$, where
the distances between two consecutive points are measured with respect to the support function $h_T$. %(u) = \sup \{ \langle x, u \rangle \, | \, x \in T \}$.
In the case where $T = B^n$ is the Euclidean unit ball, these trajectories are the classical billiard orbits in the billiard table $K$. Hence, the above reflection law is a natural variation of the classical one when the Euclidean structure is replaced by a Minkowski norm. 
We proceed with a formal description
\begin{figure} %[h1]
\begin{center}
\begin{tikzpicture}[scale=0.7]

 \draw[important line][rotate=30] (0,0) ellipse (75pt and 40pt);

 \path coordinate (w1) at (2.1,4*0.75) coordinate (q0) at
 (-4.5*0.5,-2*0.2) coordinate (q1) at (1.6,4*0.44) coordinate (q2) at
 (1.74,+0.46) coordinate (w2) at (2.6,-1.1) coordinate (w3) at (-3.2,-0.11) coordinate (K) at
 (0,0.15);

\draw[blue] [important line] (q0) -- (q1); \draw[->] (q1) -- (w1);
\draw[blue] [important line] (q1) -- (q2); %\draw[->] (q2) -- (w2);
\draw[->] (q0) -- (w3);

\filldraw [black]
  (w3) circle (0pt) %node[above left=-2.5pt] {{\footnotesize $\nabla \|q_2\|_{K}$}}
    (w1) circle (0pt) %node[right] {{\footnotesize $\nabla \|q_1\|_{ K}$}}
    (w2) circle (0pt) % node[below ] {{\footnotesize $w_2=\nabla \|q_2\|_K$}}
     (q0) circle (2pt) node[below right] {{\footnotesize $q_2$}}
      (q1) circle (2pt) node[above right=0.5pt] {{\footnotesize $q_1$}}
       (q2) circle (2pt) node[below=0.5pt] {{\footnotesize $q_0$}}
        (K) circle (0pt) node[right=0.5pt] {${ K}$};

 %      % We start the second graph
       \begin{scope}[xshift=7cm]

 \draw[important line][rounded corners=10pt][rotate=10] (1.8,0) --
 (0.8,1.8)-- (-0.8,1.8)--  (-1.8,0)--  (-0.8,-1.8) -- (0.8,-1.8) --
 cycle;

 \path coordinate (p1) at (0.5,-4*0.433) coordinate (np0) at
 (2.3,2*0.9) coordinate (p0) at (1.2,2*0.53) coordinate (np1) at
 (0.7,-3) coordinate (p2) at (-1.2,2*0.66) coordinate (D) at
 (0.3,0.22);

 \draw[<-] (np0) -- (p0); 
 %node[right] {{\footnotesize $\nabla \|p_1  \|_{T}$}} -- (p0);
 \draw[<-] (np1)  -- (p1);
%node[right] {{\footnotesize $\nabla \|p_0  \|_{ T}$}}
 \draw[blue][important line] (p0) -- (p1);
 \draw[blue][important  line] (p0) -- (p2);

  \filldraw [black]
       (p1) circle (2pt) node[below right] {{\footnotesize $p_0$}}
         (p2) circle (2pt) node[left] {{\footnotesize $p_2$}}
         (p0) circle (2pt) node[above=2pt] {{\footnotesize $p_1$}}
          (D) circle (0pt) node[left] {${ T}$};
 \end{scope}
 \end{tikzpicture}

 \caption{A $({K},{T})$-billiard trajectory.}  \label{fig1-KT-billiard} \label{fig-Minkowksi-billiard}
 \end{center}
 \end{figure}
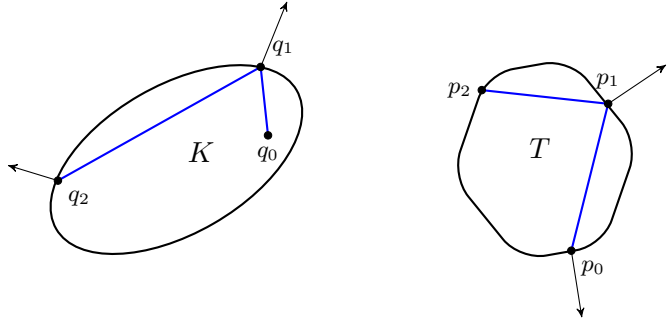

\subsubsection{A Continuous Model} \label{Subsection-background-Minkowski-Billiards}

In this section we provide a continuous-time model for the Minkowski billiard dynamics, %The definition below uses
%characteristic foliation on non-smooth convex hypersurfaces in the classical
%phase space, and is
which is motivated by the natural correspondence between
geodesics in a Riemannian manifold  and characteristics of its unit cotangent bundle (see~\cite{AAO}).

Consider ${\mathbb R}^{2n}$ with position-momentum coordinates $(q,p)$, and equipped with the standard symplectic form $\omega = dq \wedge dp$. %, and with the standard inner product $\langle \cdot, \cdot \rangle$.
Recall that the restriction of $\omega $ %the standard symplectic form $\omega=dq \wedge dp$
to a smooth closed connected
hypersurface ${\mathcal S}
\subset {\mathbb R}^{2n}$ defines the canonical line bundle %a 1-dimensional subbundle
${\mathfrak S}_{{\mathcal S}}  := {\rm ker}(\omega | {\mathcal S})$, whose integral curves
are called the characteristics of ${\mathcal S}$.
%comprise the
%characteristic foliation of ${\mathcal S}$.
% A {\it closed
%characteristic}  of  ${\mathcal S}$ is thus an embedded circle
%in  ${\mathcal S}$ tangent to the canonical %characteristic
%line-bundle
%\begin{equation*} {\mathfrak S}_{{\mathcal S}} = \{(x,\xi) \in T
%{\mathcal S} \ | \ \omega(\xi,\eta) = 0 \ {\rm for \ all} \ \eta \in T_x
%{\mathcal S} \}. \end{equation*}
%The classical geometric problem of finding a closed characteristic %on $\Sigma$
%has a well-known dynamical interpretation: if  the boundary
%$\partial \Sigma$ is represented as a regular energy surface $\{x
%\in {\mathbb  R}^{2n} \ | \ H(x) = {\rm const} \}$ of a smooth
%Hamiltonian function $H : {\mathbb  R}^{2n} \rightarrow {\mathbb
%R}$, then the restriction to $\partial \Sigma$ of the Hamiltonian
%vector field $X_H$, defined by $i_{X_H} \omega_{\rm st} = -dH$, is a section
%of ${\mathfrak S}_{\Sigma}$. Thus, the image of the periodic
%solutions of the classical Hamiltonian equation $\dot x= X_H(x) = J\nabla H(x)$ on
%$\partial \Sigma$ are precisely the closed characteristics of
%$\partial \Sigma$.
%Recall that
Moreover, recall that the symplectic action %$A(\gamma)$
of a closed curve $\gamma$ % which is the enclosed symplectic area,
is defined by $A(\gamma) = \int_{\gamma} \lambda$,
where  $\lambda =pdq$ is the Liouville 1-form. %whose differential  $d \lambda = \omega$.
%The action spectrum of ${\mathcal S}$ is %defined as
%\begin{equation*}  {\cal L}({\mathcal S}) = \left \{ \, | \, {A}({\gamma}) \,  | \, ;
%\, \gamma \ {\rm closed \ characteristic \ on} \  {\mathcal S}
%\right \}.\end{equation*}

Next, let $K \subset {\mathbb R}^n_q$, and $T \subset {\mathbb R}^n_p$ be two smooth convex bodies, and consider the
$g_T$-unit cotangent bundle
\begin{equation*}
U_T^*K :=  \{ (q,p) : q \in K, \ {\rm and} \
g_T(p)  \leq 1 \}  \simeq K \times T , \end{equation*}
where $g_T$ is the gauge function of $T$, and the standard inner product identifies $ T^* {\mathbb R}^n_q$ with ${\mathbb R}^{n}_q
\times  {\mathbb R}^{n}_p$.
%\begin{equation*}
%U_T^*K := K \times T = \{ (q,p) \, | \, q \in K, \ {\rm and} \
%g_T(p)  \leq 1 \} \subset T^* {\mathbb R}^n_q  = {\mathbb R}^{n}_q
%\times  {\mathbb R}^{n}_p. \end{equation*} Here $g_T$ is the gauge function of $T$. Note that
%% i.e., $g_T(x) = \inf \{r \, | \, x \in rT \}$, and in particular when $T$ is centrally symmetric i.e., $T=-T$, one has $g_T(x) = \|x\|_T$.
%for $p \in \partial T$, the gradient vector $\nabla g_T(p)$ is the outer normal to $\partial T$ at the point $p$,  and is naturally considered to be in $\mathbb R^n_q$.
% where we have used the standard identification between $T_p{\mathbb R}^n$ and $T^*_p{\mathbb R}^n$ via the usual scalar product.
Note that the boundary of the product $K \times T \subset {\mathbb R}^n_q \times {\mathbb R}^n_p$
is a non-smooth hypersurface. However, away from the singular strata  $\partial K \times \partial T$
the canonical line bundle ${\mathfrak S}_{{\partial (K \times T)}}$  is well defined, and is generated by the (Hamiltonian) vector field
\begin{equation*}
{\mathfrak X}(q,p) = \left\{
\begin{array}{ll}
(-\nabla g_T(p) ,0), &  (q,p) \in {\rm int}(K) \times \partial T,\\
(0,\nabla g_K(q)), & (q,p) \in \partial K \times {\rm int}(T).
\end{array} \right.
\end{equation*}
The convexity of $K \times T$ allows one to naturally extend the notion of characteristic directions to $\partial K \times \partial T$,
which leads to the following definition of billiard dynamics (see~\cite{AAO}).
\begin{definition} \label{def-of-periodic-traj}
A $(K,T)$-billiard trajectory is the image of a piecewise smooth
map\footnote{Here by piecewise smooth we mean a continuous map which is $C^1$ except at a discrete subset of ${\mathbb R}$.} $\gamma \colon {\mathbb R} \rightarrow \partial (K \times T) $, $\gamma(t)=(q(t),p(t))$, 
%such that % the following:
such that for every  
$t \notin {\mathcal B}_{\gamma}:= \{ t
\in {\mathbb R} : \gamma(t) \in \partial K \times \partial T \}$ one has
\begin{equation*} %\label{eq:Ham-equations-of-billiards}
\dot \gamma(t) = a \, {\mathfrak X}(\gamma(t)), \ {\rm for \ some  \ constant } \ a \in (0,\infty). \end{equation*}
%for some %positive constant $d >0$.
 %and the vector field ${\mathfrak X}$ is given by
%\begin{equation*}
%{\mathfrak X}(q,p) = \left\{
%\begin{array}{ll}
%(-\nabla g_T(p) ,0), &  (q,p) \in {\rm int}(K) \times \partial T,\\
%(0,\nabla g_K(q)), & (q,p) \in \partial K \times {\rm int}(T).
%\end{array} \right.
%\end{equation*}
Moreover, for any $t \in {\mathcal B}_{\gamma}$, the left and right
derivatives $\dot \gamma^{\pm}(t)$ of $\gamma(t)$ exist, and 
\begin{equation*} \label{eq-the-cone}
\dot \gamma^{\pm}(t) \in \{   \alpha (-\nabla g_T(p(t)) ,0) + \beta
(0,\nabla g_K(q(t)))    :  \alpha,\beta \geq 0,  \ (\alpha, \beta) \neq (0,0) \}.
\end{equation*}
%A general $(K,T)$-billiard trajectory billiard (not necessarily closed) is a piecewise smooth
%$\gamma \colon {\mathbb R} \rightarrow \partial (K \times T) $ which satisfies the same condition.
\end{definition}

\begin{remark} \label{rmk:Role-of-K-T-interchangable}
{\rm
For a $(K,T)$-billiard trajectory $\gamma$, the curve $\pi_q(\gamma)$,  where $\pi_q \colon {\mathbb R}^{2n} \rightarrow {\mathbb R}^n_q$ is the projection onto $ {\mathbb R}^n_q$,
is called a $T$-billiard trajectory in $K$, and the curve $\pi_p(\gamma)$ (which is the corresponding projection to ${\mathbb R}^n_p$), is called a $K$-billiard trajectory in $T$.
Thus, there is a natural one-to-one correspondence between $T$-billiard trajectories in $K$ and $K$-billiard trajectories in $T$. We refer to such a pair of trajectories as ``dual billiard trajectories". Furthermore, for a pair of periodic dual trajectories, the length with respect to the support function $h_T$ of the $T$-billiard trajectory in $K$ equals the length with respect to $h_K$ of the corresponding $K$-billiard trajectory in $T$, and also equals the symplectic action of the corresponding closed characteristic in ${\mathbb R}^{2n}$ traversed once (see  Section 7 of~\cite{GT}, and  Section 2.4 of~\cite{AAO}).
} \end{remark}

\begin{remark} \label{rmk:proper-gliding-trajectories} {\rm
The trajectories described geometrically in the beginning of this section (see Figure~\ref{fig1-KT-billiard}) are $(K,T)$-billiard trajectories in the sense of Definition~\ref{def-of-periodic-traj}. Their projections to ${\mathbb R}^n_q$ are Minkowski $T$-billiard trajectories in $K$ in the sense of~\cite{GT}
(which include the case of classical Euclidean billiards where $T=B$).
Note also that Definition~\ref{def-of-periodic-traj} allows also the so-called ``gliding trajectories" which move solely on $\partial K \times \partial T$ (see~\cite{AAO} for more details). However, these trajectories will play no role in what follows.
Finally, we remark that our notion of $T$-billiard dynamics in $K$  includes ``strange billiard orbits" in the sense of Halpern~\cite{Halp}. Such an orbit, when it exists, can be realized, say,  as the image of a piecewise smooth map which in finite time converges to a point, and continues as a gliding trajectory. These trajectories demonstrate the following phenomenon: on the singular strata $\partial K \times \partial T$, the differential relation in Definition~\ref{def-of-periodic-traj}  might have non-unique solutions. As we shall see, this will not influence our discussion below, as away from the singular strata the differential relation has a unique solution, which, moreover, can be extended uniquely up to a time at which the solution intersects the singular strata.
%
%e.g.,  when $K=T=B^2$, one can check that
%that $\gamma(t) = (e^{it+\pi/2},e^{it})$ is a valid billiard trajectory (see~\cite{AAO} for more details).
}
\end{remark}

\subsubsection{The Discrete Model} \label{Subsection-background-Minkowski-Billiards-2nd}

In the previous subsection we described a continuous-time model of Minkowski billiard dynamics via the characteristic foliation in the phase space ${\mathbb R}^{2n}$. There exists another natural description of the billiard motion as a discrete-time dynamical system, which for Euclidean billiards goes back to Birkhoff~\cite{Bir} (cf.~\cite{GT} for a similar approach in the Minkowski case). Consider the $T$-billiard dynamics in $K$ as a mapping of pairs: $(q,v) \mapsto (q',v')$, where $q \in \partial K$, the vector $v $ is the direction of propagation,
$q'$ is the following impact point, that is $q'$
is the point where the trajectory starting
at $q$ with velocity $v$ hits the boundary $\partial K$ next, and $v'$
is the reflected velocity, i.e., the direction of propagation after the reflection at the point $q'$. More precisely, let $K \subset {\mathbb R}^n_q$ and $T \subset {\mathbb R}^n_p$ be two smooth centrally symmetric strictly convex bodies. Let $n_K(q)$ be the unit outer normal to $K$ at the point $q \in \partial K$, and similarly for $n_T(p)$. Note that $n_K(q) = \nabla g_K(q)/ | \nabla g_K(q) |$. Denote 
\begin{eqnarray*}
 (\partial K \times \partial T )_{+} & = &  \{ (q,p) \in \partial K \times \partial T   :  \iprod{n_K(q)}{n_T(p)} > 0\},\\
 (\partial K \times \partial T )_{-} & = &  \{ (q,p) \in \partial K \times \partial T   :  \iprod{n_K(q)}{n_T(p)} < 0\}, \\
  (\partial K \times \partial T )_{\,0} & = &  \{ (q,p) \in \partial K \times \partial T   :  \iprod{n_K(q)}{n_T(p)} = 0\},
\end{eqnarray*}
and consider the decomposition
$$ \partial K \times \partial T = (\partial K \times \partial T )_{+} \cup \partial (K \times \partial T)_{0} \cup  (\partial K \times \partial T)_{-}.$$
% The non-normalized gradient in fact satisfies $h_K (\nabla g_K (x) ) = 1$.
%\begin{remark} {\rm
It follows from Definition~\ref{def-of-periodic-traj} (see also~\cite{AAO}) that a $(K,T)$-billiard trajectory $(q(t),p(t))$ for which $(q(t_0),p(t_0)) \in (\partial K \times \partial T )_{-} $ will
continue to a point in $\partial K \times {\rm int}(T)$, i.e., one has $(q(t_0+\eps),p(t_0+\eps)) \in \partial K \times {\rm int}(T)$, for small enough $\eps >0$, and was in ${\rm int}(K) \times \partial T$  before the collision with the singular set $\partial K \times \partial T$, i.e., $(q(t_0-\eps),p(t_0-\eps)) \in {\rm int}(K) \times \partial T$, for small enough $\eps >0$.
In a similar manner, a $(K,T)$-billiard trajectory $(q(t),p(t))$ for which $(q(t_0),p(t_0)) \in (\partial K \times \partial T )_{+}$ arrives to this point from $\partial K \times {\rm int}(T)$, and continues to move in ${\rm int}(K) \times \partial T$. On the other hand, trajectories for which  $(q(t_0),p(t_0)) \in (\partial K \times \partial T )_{0}$ may exhibit a more complicated behavior. They can move in the singular strata $\partial K \times \partial T$ as gliding trajectories, or, as the example of ``strange (Euclidean) billiards" (see~\cite{Halp}) shows, the point $(q(t_0),p(t_0))$ might be
the accumulation point of the collision points of
a billiard trajectory which has an infinite number of collisions in finite time. We remark that under sufficient smoothness conditions on the boundaries $\partial K$ and $ \partial T$, strange billiards cannot occur, as shown in Proposition 2.12 in~\cite{AAO}.
%} \end{remark}

Next, consider the Poincar\'e return map associated with the subset  $$\Sigma := (\partial K \times \partial T)_{+} \cup (\partial K \times \partial T)_{-},$$ which the $(K,T)$-billiard trajectories cross transversally. By the above discussion, the return map is indeed well defined. We shall denote this map by $\Psi = \Psi_{K,T}$. One can check directly that $\Psi$ interchanges the two sets $(\partial K \times \partial T)_{+}$ and $(\partial K \times \partial T)_{-}$, i.e.,
\[ \Psi : (\partial K \times \partial T )_{+}\to (\partial K \times \partial T )_{-}   \ {\rm and} \
\Psi : (\partial K \times \partial T )_{-}\to (\partial K \times \partial T )_{+}.\]
%Considering the dynamics of $(K,T)$-billiard trajectories given in Definition~\ref{def-of-periodic-traj} above one sees that for  $(x,y) \in (\partial K \times \partial T)_-$ one has $\Psi(x,y) = (x,h(y))$, where $h$ depends on $x$ and $h(y)$ belongs to the oriented line $(y,n_K(x))$. Note that $(x,h(y)) \in (\partial K \times \partial T)_+$ since
% the segment $[y,h(y)] \subset T$, $h(y) \in \partial T$,  and $h(y)-y = \lambda n_K(x)$ for some $\lambda >0$, and thus  $\iprod{n_K(x)}{n_T(h(y))} >0$. Similarly,
%for $(x,y) \in (\partial K \times \partial T)_+$ one has $\Psi(x,y) = (g(x),y)$, where $g(x)$ is on the oriented line $(x, -n_T(y))$ (equivalently, $x$ belongs to the line passing through $g(x)$ in the direction $n_T(y)$). As above, one checks that $(g(x),y) \in  (\partial K \times \partial T )_{-}$.
We observe that $\Sigma$ is a symplectic submanifold of $\R^{2n}$, and since the vector field $\mathfrak{X}$ generates the null direction of $\omega$ on $\Sigma$, the Poincar\'{e} map $\Psi$ is symplectic, i.e.,
\begin{equation}
\Psi^*(dq \wedge dp) = dq \wedge dp.
\end{equation}
Considering the dynamics of the $(K,T)$-billiard trajectories given in Definition~\ref{def-of-periodic-traj} above, one sees that for  $(q,p) \in (\partial K \times \partial T)_+$ one has $\Psi(q,p) = (q',p)$, where $q'$ is the next impact point of the ray in $K$ emanating from $q$ in direction $-n_T(p)$ with the boundary $\partial K$. Note that this point exists, as $n_T(p)$ points outside of $K$ at $q$ by the definition of the set $(\partial K \times \partial T)_+$. Similarly, for $(q,p)
\in (\partial K \times \partial T)_-$, one has $\Psi(q,p)=(q,p')$, where $p'$ is the next impact point of the ray in $T$ emanating from $p$ in the direction $n_K(q)$.

\begin{definition}\label{def-Psi-square} In the following we refer to $\Psi^2$ as the  ``discrete $(K,T)$-billiard map".
\end{definition}
To motivate this, consider the restriction of $\Psi^2$ to, say, $(\partial K \times \partial T )_{+}$. Note that the set  $(\partial K \times \partial T)_+$
can be viewed as the phase space of the $T$-billiard dynamics in $K$.  Indeed,
a natural description of the phase space associated with Minkowski billiards (see~\cite{GT}) is the set of inward pointing $h_T$-unit tangent vectors with foot at $\partial K$, i.e.,
$$
\cP(K):=\{(q,v) \,:\, q \in \partial K, v \in T_qK, \, h_T(v)=1, \, \iprod{v}{n_K(q)} < 0 \}.
$$
There is a natural bijection $\cP(K) \to (\partial K \times \partial T)_+$ via $(q,v) \mapsto (q,p)$, where $v=-n_T(p)$ (see Figure \ref{fig-P(K)}). Under this identification, the map $\Psi^2$ takes the form of the usual billiard ball map $(q,v) \mapsto (q',v')$, as described above.
%where $q'$ is the points of impact of the line emanating from $q$ in direction $v$ with $q'$, and $v'$ is the unit vector in the reflected direction according to the Minkowski reflection law ({\bf TODO} add picture?).
A similar identification as above holds for the set $(\partial K \times \partial T)_-$ as the phase space $\cP(T)$ of the $K$-billiard dynamics in $T$.

\begin{figure} %[h1]
	\begin{center}
		\begin{tikzpicture}[scale=0.7]
		
		\draw[important line][rotate=30] (0,0) ellipse (75pt and 40pt);
		
		\path coordinate (w1) at (2.1,4*0.75) coordinate (q0) at
		(-4.5*0.5,-2*0.2)  coordinate (q0new) at
		(-4.5*0.5-0.4*3.85,-2*0.2-0.4*2.16)  coordinate (q1) at (1.6,4*0.44) coordinate (q1new) at (1.6+0.5*3.85,4*0.44+0.5*2.16)  coordinate (q2) at
		(1.74,+0.46) coordinate (w2) at (2.6,-1.1) coordinate (w3) at (-3.2,-0.11) coordinate (K) at
		(0,0.15) %;
		coordinate (v) at (-4.5*0.5+2.3-1.2,-2*0.2+2*0.9-2*0.53);  %(q0)+(np0)-(p0)
		\draw[<-] (v)node[right] {{\footnotesize $v$}}  -- (q0);
		%\draw[red] [important line] (q0) -- (q1);
		%\draw[->] (q1) -- (w1);
		%\draw[blue][->] [important line] (q0new) -- (q1new);
		%\draw[red] [important line] (q1) -- (q2); %\draw[->] (q2) -- (w2);
		
		\filldraw [black]
		%(w1) circle (0pt) node[above] {{\footnotesize $$}}
		(w2) circle (0pt) % node[below ] {{\footnotesize $w_2=\nabla \|q_2\|_K$}}
		(q0) circle (2pt) node[left] {{\footnotesize $q$}}
		%      (q1) circle (2pt) node[below] {{\footnotesize $x$}}
		%(q1new) circle (0pt) node[below right] {{\footnotesize ${\ell}$}}
		%    (q2) circle (2pt) node[below=0.5pt] {{\footnotesize $q_0$}}
		(K) circle (0pt) node[right=0.5pt] {${ K}$};
		
		%      % We start the second graph
		\begin{scope}[xshift=8cm]
		
		\draw[important line][rounded corners=10pt][rotate=10] (1.8,0) --
		(0.8,1.8)-- (-0.8,1.8)--  (-1.8,0)--  (-0.8,-1.8) -- (0.8,-1.8) --
		cycle;
		
		\path coordinate (p0new1) at (1.2+.8*2.4 ,2*0.53-.8*0.84) coordinate (p2new1) at (-1.2,1.9) coordinate (p1) at (0.5,-4*0.433) coordinate (p1new) at (0.5-0.3*0.7,-4*0.433-0.3*2.792)  coordinate (np0) at
		(2.3,2*0.9) coordinate (p0) at (1.2,2*0.53) coordinate (p0new) at (1.2+0.6*0.7,2*0.53+0.6*2.792)  coordinate (np1) at
		(0.7,-3) coordinate (p2) at (-1.2,2*0.66) coordinate (D) at
		(0.3,0.22) 
		coordinate (p3) at (-1.2,-2*0.53) coordinate (np3) at
		 (-2.3,-2*0.9);
		%\draw[<-] (np0) node[right] {{\footnotesize $n_T(p)=v$}} -- (p0);
		\draw[<-] (np3) node[below] {{\footnotesize $n_T(p)=-v$}} -- (p3);
		%\draw[<-] (np1) node[right] {{\footnotesize $\nabla \|p_0  \|_{ T}$}} -- (p1);

		\filldraw [black]
		%    (p1) circle (2pt) node[below right] {{\footnotesize $p_0$}}
		%     (p2) circle (2pt) node[left] {{\footnotesize $p_2$}}
		%(p0new1) circle (0pt) node[below] {{\footnotesize $\alpha({\ell})=(p,w)$}}
		%(p0) circle (2pt) node[below] {{\footnotesize $p$}}
		(p3) circle (2pt) node[above right] {{\footnotesize $p$}}
		(D) circle (0pt) node[left] {${ T}$};
		\end{scope}
		\end{tikzpicture}
		
		\caption{\rm The isomorphism $\cP(K) \simeq (\partial K \times \partial T)_+$, $(q,v) \mapsto (q,p)$.} \label{fig-P(K)}
	\end{center}
\end{figure}
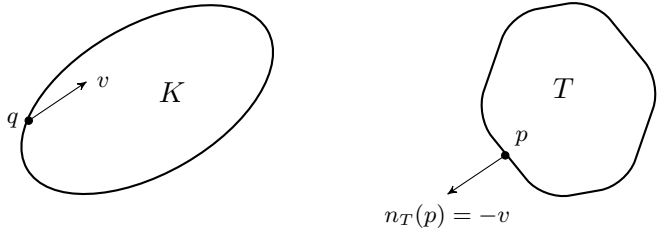

%Similarly, the set $(\partial K \times \partial T)_-$ is identified with the classical phase space of $K$-billiards in $T$,
%$$
%\cP(T) := \{(p,w)\,:\, p\in \partial K, w \in T_pT,\, h_K(w) = 1,\, \iprod{w}{n_T(p)} < 0 \},
%$$
%via $(q,p) \mapsto (p,w)$ where $w = n_K(q)$, and $\Psi^2$ is identified with the usual billiard ball map.

%
%
%This map takes a pair $(x,y)$, corresponding to a point on $\partial K$ and a direction of propagation $v$ (which is $-n_T(y)$, for our $y\in \partial T$) to another pair $(x',y')$, where again $x'\in \partial K$ and the new direction of propagation is $v' = -n_T(y')$.  In other words, our ``velocity vectors'' are naturally parametrized by points $y\in \partial T$ on the boundary of $T$ such that the velocity is given by $-n_T(y)$.
%In a similar manner, the restriction of $\Psi^2$ to $(\partial K \times \partial T )_{-}$ is naturally identify  with $K$-billiards on $T$.

%\begin{remark} \rm
%
%The map $\Psi$ is closely related with another map on $\partial K \times \partial T$ which we will describe in Section .....
%%
%The relation between the map $\Psi$ and the $(K,T)$-duality transform $\alpha$ given in Definition~\ref{def-KT-duality-map} above is as follows:
%Assume that $(q_{-1},p_0) \in  (\partial K \times \partial T)_+$, and let $(q_{0},p_0) = \Psi (q_{-1},p_0)$, and $  (q_{0},p_1) = \Psi (q_{0},p_0)$.
%Then, one has $\alpha( (q_0, -n_T(p_0))) = (p_0,n_K(q_0))$. That is, the oriented line $\overrightarrow {q_0,q_{-1}}$ is mapped by $\alpha$ to the  oriented line  $\overrightarrow {p_0,p_{1}}$.
%\end{remark}

\subsection{Planar Minkowski Billiards}
In this section we restrict attention to the case of convex bodies in the plane.
In this case, a convex set determines in a simple geometric way a 1-parameter family of convex billiard tables for which it is a Minkowski caustic. This geometric construction will play a central role in what follows, and will be described in Section~\ref{sec:string-const}.
In Section~\ref{subsect:twist-map} we recall the fact that the Minkowski billiard map is a monotone twist map, and in Section~\ref{sec:caustics-invariant-circles} we discuss the relation between invariant circles of this monotone twist map and convex caustics. In particular, we rephrase the notion of caustic duality via duality of the corresponding invariant circles. Finally, in Section~\ref{Sec:dual-caustics-classical-invariants} we discuss some classical invariants of caustics and their behaviour under caustic duality.

\subsubsection{The String Construction}\label{sec:string-const}

It is well known that for Euclidean billiards, for every convex set $C$ in the plane one can associate a 1-parametric family of billiard tables $(K_L)_{L > {\rm Per(C)}}$  such that each table has $C$ as a caustic.
Roughly speaking, $\partial K_L$ is obtained by the following procedure: wrap a loop of inelastic string of length $L$ around $C$. Then, pull the string tight away from $C$ to produce a point $p$ on the boundary of the billiard table $K_L$. Finally, move the point $p$ around $C$, keeping the string tight, to obtain the rest of $\partial K_L$. This technique is known as the ``gardener's string construction" (see e.g.,~\cite{CM,T1}). Note that when this construction is applied to a closed line segment one obtains an ellipse. For an illustration of a Euclidean string construction over a triangle see Figure~\ref{fig:ellipsogon} in the Appendix. 
%More formally, fixing %$L > {\rm length}(\partial C)$(here ``length'' means Euclidean length)
%$L>\Per(C)$, let
%\[ K_L = \{ x\in \R^2: \Per ({\rm conv} (x, C) ) \le L\}.\]
%%\[ K = \{ x\in \R^2: {\rm length} ({\rm conv} (x, C) ) \le L\}.\]
%The convexity of $K_L$ is a classical fact and will follow, e.g., from Lemma~\ref{lem:conveixty-of-Mink-string-construction} below.

This string construction can be naturally generalized to planar Minkowski billiards (see Section 3 of~\cite{GT}), where lengths are measured with respect to some Minkowski norm on ${\mathbb R}^2$. More precisely, let $C\subset \RR^2_q$ be a convex set, and $T\subset \RR^2_p$ a centrally symmetric convex body. The result of the $h_T$-string construction over $C$ with string length $L > \Per_{h_T}(C)$ is, by definition,
 %\[ K = \{ x\in \R^2: {\rm length}_{T^{\circ}} \partial ({\rm conv} (x, C) ) \le L\}.\]
 \[ K = \{ q\in \R^2: \Per_{h_T} (\conv (q, C) ) \le L\}.\]
%In fact, every convex $T$-caustic in $K$ can be obtained via some $h_T$-string construction. To state this precisely, we fix two smooth convex bodies $K \subset {\mathbb R}^2_q$, and $T \subset {\mathbb R}^2_p$, where $T$ is assumed to be centrally symmetric, and equip ${\mathbb R}^2_q$ with the norm given by the support function $h_T$. All measurements from now on are with respect to this norm.

%First, let us show that a table $K$ obtained via a (Minkowski) string construction around a convex body $C$ with respect to the support function $h_T$ of a centrally symmetric convex body $T$,  is indeed convex.
%We first observe that the body thus obtained is convex.
\begin{remark} \label{rmk:Lazutkin-parameter} {\rm The quantity $L - \Per_{h_T}(C)$ is called the ``Lazutkin parameter" associated with the above string construction over $C$ (see e.g., Chapter 3 in~\cite{Sib} or~\cite{GT}). }
\end{remark}

Note that  the boundary $\partial K$ is a level set of the function $f(q) = \Per_{h_T}( \conv(q, C) )$.
The following lemma (cf. the proof of Lemma 3.6 in~\cite{GT}) shows that the function $f$ is smooth. 
For the proof we shall need some basic facts from the differential geometry of curves in normed planes. We refer the reader to~\cite{BMS,Bib,Pet}  for the relevant definitions and properties. 
\begin{lemma} \label{lem:derivatives-of-lazutkin} %Assume $C \subset {\mathbb R}^2$ is a $C^2$-smooth convex body with non-vanishing Minkowski curvature. 
Let $C \subset {\mathbb R}^2_q$ be a compact convex set, and $T  \subset {\mathbb R}^2_p$ a centrally symmetric strictly convex body. Then, the function $f(q) = \Per_{h_T}( \conv(q, C) )$ is $C^1$-smooth 
on ${\mathbb R}^2_q \setminus C$. Moreover, one has
$$ \nabla f(q) = n_{T^{\circ}}(u)+n_{T^{\circ}}(v),$$
where $u$ and $v$ are $h_{T}$-unit vectors parallel to the tangent lines from $q$ to $C$, pointing towards the point $q$ (see Figure~\ref{string_construction_fig}). 
%$$u = {\frac {x-b} {h_T(x-b)} } \ {\rm and} \ v={\frac {x-e} {h_T(x-e)} },$$ and the points $b$ and $e$  are the tangency points from $x$ to $C$ (see Figure ???).  
\end{lemma}

\begin{figure} %[h1]
\begin{center}
\begin{tikzpicture}[scale=0.7]

\node[inner sep=0pt] (tangents) at (-1,0)
    {\includegraphics[width=.3\textwidth]{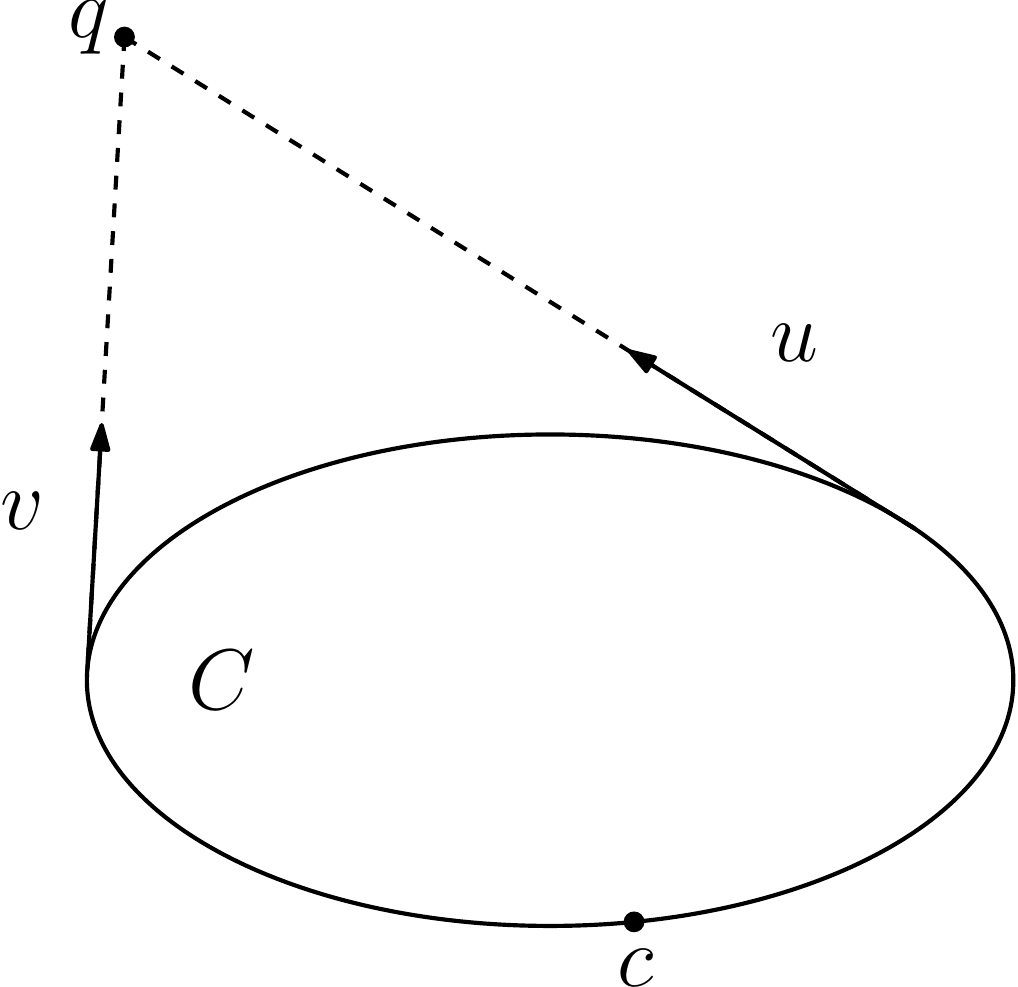}};

% \draw[important line][rotate=30] (0,0) ellipse (75pt and 40pt);
%
% \path coordinate (w1) at (2.1,4*0.75) coordinate (q0) at
% (-4.5*0.5,-2*0.2) coordinate (q1) at (1.6,4*0.44) coordinate (q2) at
% (1.74,+0.46) coordinate (w2) at (2.6,-1.1) coordinate (w3) at (-3.2,-0.11) coordinate (K) at
% (0,0.15);
%
%\draw[blue] [important line] (q0) -- (q1); \draw[->] (q1) -- (w1);
%\draw[blue] [important line] (q1) -- (q2); %\draw[->] (q2) -- (w2);
%\draw[->] (q0) -- (w3);
%
%\filldraw [black]
%  (w3) circle (0pt) %node[above left=-2.5pt] {{\footnotesize $\nabla \|q_2\|_{K}$}}
%    (w1) circle (0pt) %node[right] {{\footnotesize $\nabla \|q_1\|_{ K}$}}
%    (w2) circle (0pt) % node[below ] {{\footnotesize $w_2=\nabla \|q_2\|_K$}}
%     (q0) circle (2pt) node[below right] {{\footnotesize $q_2$}}
%      (q1) circle (2pt) node[above right=0.5pt] {{\footnotesize $q_1$}}
%       (q2) circle (2pt) node[below=0.5pt] {{\footnotesize $q_0$}}
%        (K) circle (0pt) node[right=0.5pt] {${ K}$};

 %      % We start the second graph
       \begin{scope}[xshift=9cm]
       \node[inner sep=0pt] (tangents-1) at (0,0)
 {\includegraphics[width=.3\textwidth]{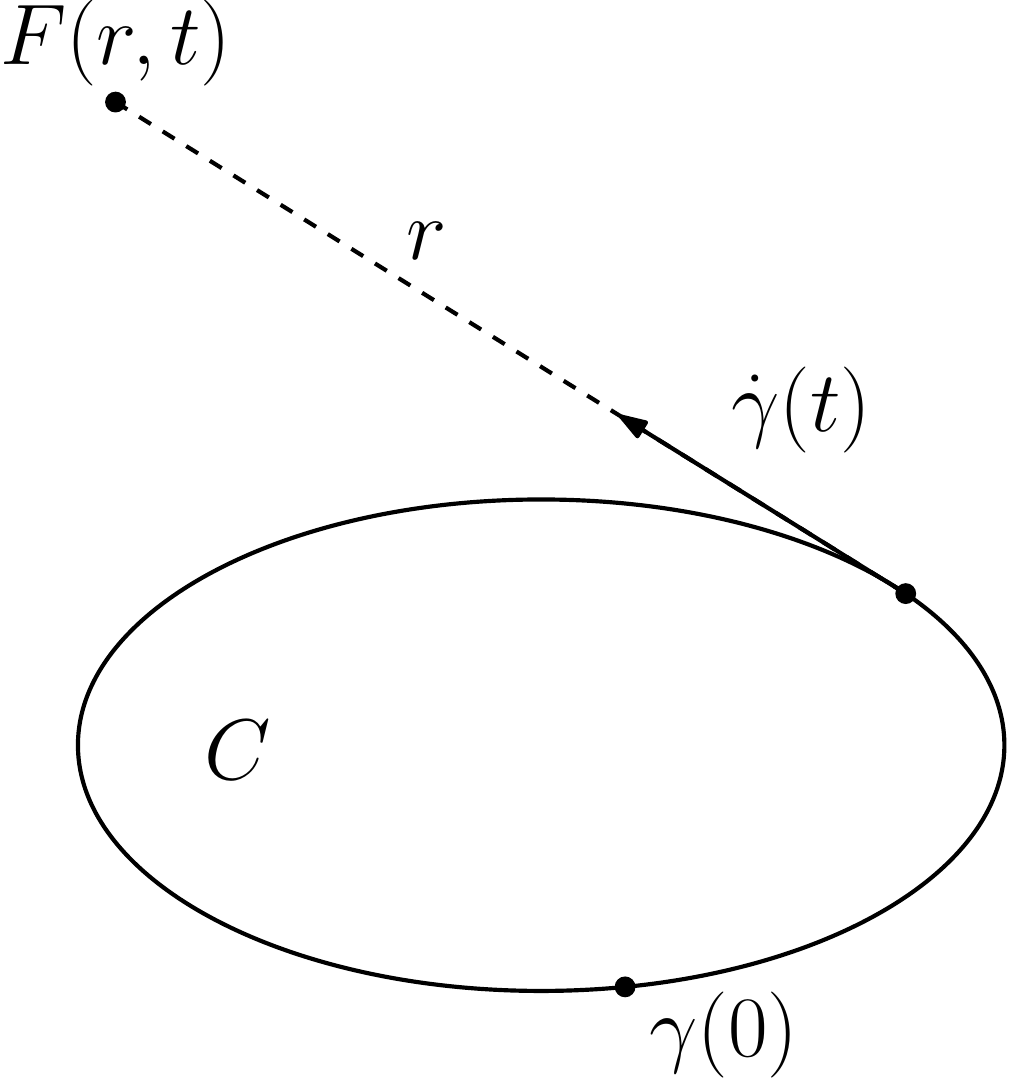}};
% \draw[important line][rounded corners=10pt][rotate=10] (1.8,0) --
% (0.8,1.8)-- (-0.8,1.8)--  (-1.8,0)--  (-0.8,-1.8) -- (0.8,-1.8) --
% cycle;
%
% \path coordinate (p1) at (0.5,-4*0.433) coordinate (np0) at
% (2.3,2*0.9) coordinate (p0) at (1.2,2*0.53) coordinate (np1) at
% (0.7,-3) coordinate (p2) at (-1.2,2*0.66) coordinate (D) at
% (0.3,0.22);
%
% \draw[<-] (np0) -- (p0); 
% %node[right] {{\footnotesize $\nabla \|p_1  \|_{T}$}} -- (p0);
% \draw[<-] (np1)  -- (p1);
%%node[right] {{\footnotesize $\nabla \|p_0  \|_{ T}$}}
% \draw[blue][important line] (p0) -- (p1);
% \draw[blue][important  line] (p0) -- (p2);
%
%  \filldraw [black]
%       (p1) circle (2pt) node[below right] {{\footnotesize $p_0$}}
%         (p2) circle (2pt) node[left] {{\footnotesize $p_2$}}
%         (p0) circle (2pt) node[above=2pt] {{\footnotesize $p_1$}}
%          (D) circle (0pt) node[left] {${ T}$};
 \end{scope}
 \end{tikzpicture}

 \caption{} %On the left the tangent vectors $u$ and $v$. On the right the coordinate system $(r,t)$.} 
 \label{string_construction_fig} %\label{fig-Minkowksi-billiard}
 \end{center}
 \end{figure}

\begin{proof}[{\bf Proof of Lemma~\ref{lem:derivatives-of-lazutkin}}]
We first assume that $T^{\circ}$ is $C^2$-smooth, and that the convex set  $C$ is $C^2$-smooth, with non-vanishing Minkowski curvature with respect to the norm $h_T$ (see~\cite{BMS,Bib,Pet}). 
%In what follows we use the notion of Minkowski curvature with respect to the norm $h_T$.  
We remark that the condition of ``non-vanishing Minkowski curvature" is independent of the choice of the Minkowski norm. 
Let $q \in \R^2 \setminus C$. We pick a point $c \in \partial C$ on the ``other side" of $\partial C$. We can write, near $q$, $f(q) = l_1(q) + l_2(q)$, where $l_1(q)$ and $l_2(q)$ are the lengths of the shortest paths from $c$ to $q$ in ${\mathbb R}^2 \setminus C$, in the counter-clockwise and clockwise directions, respectively (see Figure~\ref{string_construction_fig}). We will compute the derivative of $l_1$. To this end, let $\gamma \colon [0,A] \to \partial C$ be a counter-clockwise, $h_T$-unit parameterization of $\partial C$ such that $\gamma(0)=c$. We consider the local diffeomorphism
$$
F(r,t) = \gamma(t) + r \gamma'(t), \quad t \in [0,A], \, r > 0.
$$
In these coordinates, $l_1$ is particularly simple, and one has that $F^*l_1(r,t) := l_1 \circ F (r,t)= r+t$. Evidently, $F^*l_1$ is $C^1$-smooth, and $F^*dl_1 = dF^*l_1 = dr + dt$. Writing this out in Cartesian coordinates, we get
\begin{equation} \label{eq:=11}
DF^T (r,t) \cdot \nabla l_1 (F(r,t))= \begin{pmatrix}
1 \\ 1
\end{pmatrix}.
\end{equation}
To differentiate $F$, we recall the Frenet equations in a normed plane (see e.g.~\cite{BMS}). By definition, $\gamma'(t)$ lies on the boundary of the $h_T$-unit ball, which as noted above, is the polar body $T^\circ$. The right normal vector $n_\gamma(t)$ to $\gamma$ at $t$ is by definition the counter-clockwise tangent to $T^\circ$ at the point $\gamma'(t)$, unit normalized in the so-called anti-norm $h_{JT}$, where $J$ is a counter-clockwise rotation of the plane by $90^\circ$. Equivalently, $-Jn_\gamma(t) = n_{T^\circ}(\gamma'(t))$. %is the outer $h_T$-unit normal to $T^\circ$ at $\gamma'(t)$
The (first) Frenet equation for $\gamma$ reads $\gamma''(t) = k_{\rm m}(t) n_\gamma(t)$, where $k_{\rm m}(t)$ is the Minkowski curvature.
Using this we compute
$$
DF(r,t) = \begin{pmatrix}
\gamma'(t) + r k_{\rm m} (t) n_\gamma(t) \, ,  \gamma'(t)
\end{pmatrix}.
$$
Inverting the matrix $DF(r,t)^T$ and using \eqref{eq:=11}, we have
$$
\nabla l_1(F(r,t)) = -J n_\gamma(t)=n_{T^\circ }(\gamma'(t)). %,\quad \text{where} \quad J = \begin{pmatrix}0 & -1 \\ 1 & 0\end{pmatrix} 
$$
%One then easily verifies that $\nabla l_1$ satisfies
%$$
%h_t(\nabla l_1) = 1, \quad \iprod{\nabla l_1}{\gamma'(t)} = 1,
%$$
%which implies that $\nabla l_1 = n_{T^\circ }(\gamma'(t))$.
Observe that in the notations of the lemma, $\gamma'(t) = u$, thus we have computed $\nabla l_1(q) = n_{T^\circ }(u).$ An identical computation gives $\nabla l_2(q) = n_{T^\circ}(v)$ which, since $f(q)=l_1(q) + l_2(q)$, proves the claim in this case.

Now we turn to the case of a general compact convex set $C$. Let $C_j$ be a sequence of $C^2$-smooth convex bodies with positive Minkowski curvature which converge to $C$ in the Hausdorff topology. Consider the associated string length functions $f_j(q) = \Per_{h_T}(\conv(q, C_j))$. One easily verifies that the sequence $(f_j)$ tends to $f$ uniformly outside a neighbourhood of $C$, and by the previous step, 
$$
\nabla f_j (q) = n_{T^\circ}(u_j) + n_{T^\circ}(v_j),
$$
where $u_j$ and $v_j$ are $h_T$-unit vectors parallel to the two tangents from $q$ to $C_j$. Clearly $u_j \to u$ and $v_j \to v$ point-wise. The function $f$ is convex, and hence has a well-defined sub-differential $\del f(q)$ at every point (see e.g, Chapter 1 in~\cite{Sch}). %This set is a basis of the normal cone for the level set of $f$ at $q$.  
Moreover, at a point where the sub-differential is a singleton the function $f$ is differentiable. Since $f_j \to f$, one has that for every $q$, 
$$
n_{T^\circ}(u) + n_{T^\circ}(v) = \lim_{j\to \infty}\nabla f_j(q) \in \del f(q).
$$
Consider a level set $S:=\{q : f(q)=L\}$. At each $q\in S$, the vector $n(q) := n_{T^\circ}(u) + n_{T^\circ}(v)$ belongs to the sub-differential $\del f(q)$, where the latter is a base of the normal cone for the level set $S$ (see again~\cite{Sch}). The function $n(q)$ is a continuous section of this normal cone. However, normal cones at distinct points of $S$ have disjoint relative interiors (Chapter 2 in~\cite{Sch}). This clearly implies that the sub-differentials are all singletons, and so $f$ is differentiable. Moreover, is follows that $\nabla f(q) = n(q)$, whence $f$ is $C^1$, as claimed. This completes the proof when $T^{\circ}$ is $C^2$-smooth. To remove the requirement on the existence of the second derivative, we approximate $T^{\circ}$ by a sequence of $C^2$-smooth convex bodies, and 
observe that all the needed properties are preserved for the above argument.
% Since the points of non-differentiability of $f$ along $S$ are at most countable, it easily follows that all sub-differentials are singletons (contain a single ray?), and so $f$ is differentiable. Moreover, is follows that $\nabla f(q) = n(q)$ which is obviously continuous, whence $f$ is $C^1$. This completes the proof.
\end{proof}

\begin{lemma} \label{lem:conveixty-of-Mink-string-construction}
Let $K$ be the body formed by an $h_T$-string construction over a convex body $C$.  Then $K$ is convex.
%
%Let $C\subset \RR^2_q$ be a convex set and $T\subset \RR^2_p$ a centrally symmetric convex body.
%Then for any %$L>{\rm length}_{T^{\circ}} \partial  C$
%$L>\Per_{h_T} (C)$, the following set is convex
% %\[ K = \{ x\in \R^2: {\rm length}_{T^{\circ}} \partial ({\rm conv} (x, C) ) \le L\}.\]
% \[ K = \{ x\in \R^2: \Per_{h_T} ({\rm conv} (x, C) ) \le L\}.\]
Moreover, if $T$ is smooth and strictly convex, then so is $K$.
\end{lemma}

In the proof (and in the sequel) we shall use the following well-known fact from convex geometry: for a centrally symmetric convex body $T \subset {\mathbb R}^2$, the $h_T$-perimeter of a planar convex body $D$ is the mixed area of the body and $T$ rotated by $90^\circ$ (see, e.g., Section 2.3 in~\cite{Tho}).
More preciesly, denote by $V(\cdot,\cdot)$ the mixed area function, % given by
%$$ {\rm Area}(K_1 + t K_2) = {\rm Area}(K_1) + 2tV(K_1,K_2) + t^2 {\rm Area}(K_2), \ {\rm for} \ t \geq 0.$$
 and let $JT$ be the body obtained by rotation of $T$ by $90^\circ$. Then,
%\[ {\rm Length}_{h_T} \partial ({\rm conv} (x, C))  = V({\rm conv} (x, C), I_T),\]
\begin{equation}\label{eq:per_moon}
\Per_{h_T} (D)  = 2 V(D, JT).
\end{equation}
\begin{proof}[{\bf Proof of Lemma~\ref{lem:conveixty-of-Mink-string-construction}}]
%
%\begin{equation}\label{eq:per_moon}
%\Per_{h_T} ({\rm conv} (x, C))  = 2 V({\rm conv} (x, C), I_T).
%\end{equation}
%\begin{proof}[{\bf Proof of Lemma~\ref{lem:conveixty-of-Mink-string-construction}}]
%Just as the surface area of a convex body is a quermassintegral (a certain mixed volume of the body and the Euclidean ball), it is well known (see~\cite{Tho}, p. 59) that the $h_T$-perimieter of a planar convex body is given by a mixed area of the body and $T$ rotated by $90^\circ$, namely
%%\[ {\rm Length}_{h_T} \partial ({\rm conv} (x, C))  = V({\rm conv} (x, C), I_T),\]
%\begin{equation}\label{eq:per_moon}
%\Per_{h_T} ({\rm conv} (x, C))  = 2 V({\rm conv} (x, C), I_T),
%\end{equation}
%where $V(\cdot,\cdot)$ stands for the mixed area (see, e.g., Section 2.3 in~\cite{Tho}), and $I_T$ is the body obtained by rotation $T^\circ$ by $90^\circ$ (counter-clockwise???).
Mixed areas are known to be linear (with respect to Minkowski sum with positive coefficients) and monotone in each argument (see~\cite{Tho} Section 2.3). Hence, relation \eqref{eq:per_moon} implies that $\Per_{h_T}(D)$ is linear and monotone with respect to $D$.  Thus, for $x,y \in K$, since for $\lambda \in (0,1)$
\begin{equation} \label{inclusion-in-Mink-string-construction-convexity-lemma}  {\rm Conv} ((1-\lambda)x+\lambda y, C) \subseteq (1-\lambda){\rm Conv} (x, C)+ \lambda{\rm Conv} (y, C),
\end{equation} one has that 
 %also $(1-\lambda)x+\lambda y$ is, as
%\begin{eqnarray} \label{inequality-in-Mink-string-construction-convexity-lemma} V({\rm conv} ((1-\lambda)x+\lambda y, C), I_T) &\le& V((1-\lambda){\rm conv} (x, C)+ \lambda {\rm conv} (y, C), I_T) \notag \\&=&  (1-\lambda)V({\rm conv} (x, C), I_T)+ \lambda V({\rm conv} (y, C), I_T) \\& \le & L/2.\notag\end{eqnarray}
\begin{align} \label{inequality-in-Mink-string-construction-convexity-lemma} 
\Per_{h_T}({\rm Conv} ((1-\lambda)x+\lambda y, C))
 &\le  \Per_{h_T}((1-\lambda){\rm Conv} (x, C)+ \lambda {\rm Conv} (y, C)) \notag \\&=  (1-\lambda)\Per_{h_T}({\rm Conv} (x, C))+ \lambda \Per_{h_T}({\rm Conv} (y, C)) \\& \le  L.\notag\end{align}
Thus, $(1-\lambda)x+\lambda y$ is also in $K$, which proves the convexity of $K$. 
Moreover, note that the inclusion in~$(\ref{inclusion-in-Mink-string-construction-convexity-lemma})$ is in fact  strict when $x\neq y$. Indeed, the point $(1-\lambda)x+\lambda z$ is not in the left-hand side of \eqref{inclusion-in-Mink-string-construction-convexity-lemma}, for one of the tangency points $z$ from $(1-\lambda)x+\lambda y$ to $C$. If $T$ is smooth, the surface area measure $dS_{JT}$ (see Definition 2.3.12 in~\cite{Tho}) of $JT$ has full support, and from the fact that $2V(D,JT) = \int_{S^1} h_D dS_{JT}$ (see ibid.), it follows that the first inequality in~\eqref{inequality-in-Mink-string-construction-convexity-lemma} is strict, and hence $K$ is strictly convex.
Finally,  if $T$ is strictly convex, then by Lemma~\ref{lem:derivatives-of-lazutkin}, the boundary $\partial K$ is a regular level set of the $C^1$-function $ f(q) = \Per_{h_T}( \conv(q, C) )$, which completes the proof of the lemma.
\end{proof}

Another important consequence of Lemma~\ref{lem:derivatives-of-lazutkin}, which appeared in~\cite{GT}, 
is that those billiard tables admitting a convex $T$-caustic $C$ are precisely all possible $h_T$-string constructions over $C$. 
%We shall only use this in very specific cases, either for $C$ which is a polytope

\begin{lemma} [{Lemma 3.6 in~\cite{GT}}]  \label{lem:Minkowski-string-construction} Let $(K,T)$ be a symmetric billiard configuration.
% pair of smooth centrally symmetric strictly convex bodies .
A convex set $C \subset {\mathbb R}^2_q$ is a $T$-caustic in $K$ if and only if the function $L_{h_T}(x)=\Per_{h_T}(\conv (x, C))$, defined for $x \in \partial K$, is constant.
\end{lemma}
In the sequel we shall also need the following simple facts regarding the Minkowski string construction.
% We shall only make use of the first item and therefor provide the proof for this case only.

\begin{lemma}\label{monotonicity-of-caustics-lemma} Let $C$ be a convex body in ${\mathbb R}^2_q$, and denote by $K_{T,L}(C) \subset {\mathbb R}^2_q$ the billiard table obtained by means of a $h_T$-string construction with string length $L > \Per_{h_T}(C)$.
\begin{itemize}
\item[(i)] If $C \subseteq C'$ are two convex bodies, then $K_{T,L}(C') \subseteq K_{T,L}(C)$,
\item[(ii)] If $L_1 \leq L_2$, then $K_{T,L_1}(C) \subseteq K_{T,L_2}(C)$,
\item[(iii)] If $T_1 \subseteq T_2$ are two convex bodies, then $K_{T_1,L}(C) \subseteq K_{T_2,L}(C)$.
\end{itemize}
\end{lemma}

\begin{proof}[{\bf Proof of Lemma~\ref{monotonicity-of-caustics-lemma}}]
%(Proof of item (i)): Take a point $x \in K_{T,L}(C')$, so that the $h_T$-length of $\partial \conv(x,C')$ is $\leq L$. Clearly, $\conv(x,C) \subset \conv(x,C')$ and so by the monotonicity of perimeter the $h_T$-length of $\partial \conv(x,C)$ is less than or equal to $L$. Thus, $x \in K_{T,L}(C)$.
(i) Take a point $x \in K_{T,L}(C')$, so that the $\Per_{h_T} (\conv(x,C')) \leq L$. Clearly, $\conv(x,C) \subseteq \conv(x,C')$, and so by the monotonicity of the $h_T$-perimeter (which follows, say, from \eqref{eq:per_moon} above), one has that $\Per_{h_T} (\conv(x,C)) \leq L$. Thus, $x \in K_{T,L}(C)$. Items (ii) and (iii) are immediate from the definition.
\end{proof}

%\begin{lemma} % Unclear if true, AND can do away with weaker lemma
%Suppose that $C_1, C_2 \subseteq T$ are all convex and let $K_1 \subseteq K_2$ be convex bodies such that $C_j$ is a $K_j$-caustic in $T$ for $j =1,2$, with the same string length (measured in the appropriate geometries). Then $C_1 \subseteq C_2$.
%\end{lemma}

%{\bf Maybe we should add the following two lemmas for completeness? Note that for the second we need the notion of invariant circles which is not yet defined.}
%
%\begin{lemma}[{Lemma 6.3 in~\cite{GK}}] Let $K$ be a $C^2$ convex billiard table in a Minkowski plane. Let $\gamma$ be a caustic corresponding to an invariant circle of the corresponding billiard map. Then $\gamma$ is contained
%in the interior of $K$.
%\end{lemma}
%
%\begin{theorem}[{Theorem 6.4 in~\cite{GK}}]
%Let $K$ be a $C^2$ convex billiard table in a Minkowski plane. If the Minkowski curvature of the boundary $\partial M$ has a zero, then the Minkowski billiard map of $K$ does not have invariant circles.
%\end{theorem}

\subsubsection{The Monotone Twist Condition}\label{subsect:twist-map}

In this section we verify that the Minkowski billiard map is a monotone twist map. This description will be useful later for relating convex caustics and invariant circles, and when we discuss various classical parameters associated with convex caustics (see Section~\ref{sec:caustics-invariant-circles} below). For Euclidean billiard dynamics this is a classical fact (see e.g.,~\cite{Sib} and~\cite{T}), and for Minkowski billiards the details can be found in~\cite{GT}. Below we use somewhat different terminology and notations from~\cite{GT}, and therefore we repeat the argument for the sake of completeness.

 Let $(K,T)$ be a symmetric billiard configuration.
 %$K \subset {\mathbb R}^2_q$ and $T \subset {\mathbb R}^2_p$ be two smooth convex bodies.
 We introduce the following notations: denote $S^1 = S^1_P = \mathbb{R} / P \mathbb{Z}$, where $P = \Per_{h_T}(K)$. Note that according to formula \eqref{eq:per_moon} one has $P= \Per_{h_T}(K) = \Per_{h_K}(T)$.  Let $\gamma_K \colon S^1 \to \partial K$ be a $h_T$-unit speed, counter-clockwise parametrization of $\partial K$. For $q \in \partial K$ let $\tau_K(q)$ denote the positive $h_T$-unit tangent to $\partial K$, i.e., $\dot{\gamma}(t) = \tau_K(\gamma(t))$. 
This gives the following natural identification of annuli:  set $A_K := S^1 \times (-1,1)$, then one has
%\begin{equation*} \label{identification-of-annuli} (\partial K \times \partial T)_+ \simeq \partial K \times (-1,1) \simeq S^1 \times (-1,1), \end{equation*}
%where the first identification is via the map $(x,y) \mapsto (x,(\iprod{n_K(x)}{{\rm exp}(i \pi/2)n_T(y)})$ (note that ${\rm exp}(i \pi/2)n_T(y) = \tau T(y)$, the tangent vector to $T$ at $y$ in the direction of propagation - positive orientation) , and the second is via normalized arc-length parametrization of the boundary $\partial K$.
\begin{equation}\label{eq:p_to_s}
(\partial K \times \partial T)_+ \simeq A_K , \,\,\, (q,p) \mapsto (t,s), \,\,\, \text{where $q=\gamma_K(t)$  and $s=\iprod{-p}{\tau_K(q)}$}.
\end{equation}
In a similar way $(\partial K \times \partial T)_{-}$ is identified with $A_T := S^1 \times (-1,1)$
via the map $$ (q,p) \mapsto  (t,s), \,\,\, \text{where $p=\gamma_T(t)$  and $s=\iprod{q}{\tau_T(p)}$}.$$ 
Let us first verify that the map in \eqref{eq:p_to_s} is a well-defined bijection. For $q \in \partial K$, denote  $$\partial T_{+,q}=\{p \in \partial T\,:\, \iprod{n_K(q)}{n_T(p)} > 0 \}.$$

\begin{lemma}\label{lem:mon-s}
With the above notations, for every $p \in \partial T_{+,q}$ one has $\iprod{p}{\tau_K(q)} \in (-1,1)$. Moreover, the map
\begin{equation*}
\partial T_{+,q} \to (-1,1), \ \  {\rm given \ by \ \ } p \mapsto s=  \iprod{-p}{\tau_K(q)},
\end{equation*}
is a (monotone) bijection.
\end{lemma}

\begin{proof}[{\bf Proof of Lemma~\ref{lem:mon-s}}]
First, as $\tau_K(q)\in\partial T^\circ$ (note that by definition $h_T(\tau_K(q))=1)$, one has $|\iprod{p}{\tau_K(q)}|\leq 1$. Moreover,  for every $p\in \partial T$ and $v \in \partial T^\circ$ one has $\iprod{p}{v}=\pm1$ if and only if $v=\pm n_T(p)$. In particular, if $\iprod{p}{\tau_K(q)}=\pm 1$, then $\iprod{n_K(q)}{n_T(p)}=\iprod{n_K(q)}{\pm\tau_K(q)}=0$, which proves the first assertion. For the second, observe that the map
\begin{equation*}
\partial T_{+,q} \to (-1,1), \ \  {\rm given \ by \ \ }  p \mapsto \iprod{p}{\tau_K(q)},
\end{equation*}
is continuous and monotone, as its derivative alongÊ $\partial T$ is $\iprod{\tau_T(p)}{\tau_K(q)}$, which is positive by the definition of $ \partial T_{+,q}$.
\end{proof}

%In what follows we denote by $A_K = S^1 \times (-1,1)$ the annulus associated with the billiard table $K$.
 Let us note that under the identification $(\partial K \times \partial T)_+ \simeq A_K$ given by \eqref{eq:p_to_s} the $1$-form $s dt$ on $A_K$ pulls back to the Liouville form $-pdq$ on $(\partial K \times \partial T)_+$, and so $ds \wedge dt$ pulls back to the standard symplectic form $dq \wedge dp$. One can naturally pull back the discrete $(K,T)$-billiard map $\Psi^2$ (see Definition~\ref{def-Psi-square}) restricted  to $(\partial K \times \partial T )_{+}$  to a bijective map $\phi_K$ from $A_K$ to itself. Moreover, we lift $\phi_K$ to a map $\widetilde \phi_K$ on $\R \times (-1,1)$ satisfying $\widetilde \phi_K(r+P,s) = \widetilde \phi_K(r,s)+(P,0)$. Note that one can also continuously extend  $\widetilde \phi_K$ to $\R \times [-1,1]$ so that on the boundaries one has $\widetilde \phi_K(r,1) = (r+P,1)$ and $\widetilde \phi_K(r,-1)=(r,-1)$. 
The map $\widetilde \phi_K$ satisfies an important condition: it is a ``monotone twist map"  
  %Note moreover that the map $\widetilde \phi_K$ is a monotone twist map
 (see e.g., Definition 1.1.1 in~\cite{Sib}), with generating function $h$ given by
\begin{equation}\label{eq:generating_function}
h(r,r') = -h_T(\gamma_K(r)-\gamma_K(r')),
\end{equation}
where %$h_T$ is the support function associated with the body $T$, and  
$\gamma_K$ has been lifted to a $P$-periodic function on $\R$. To show this, one computes 
%a straightforward computation gives
$$
\frac{\partial h}{\partial r} = \iprod{-n_K(q-q')}{\tau_K(q)},
$$
where $q=\gamma_K(r)$ and $q'=\gamma_K(r')$, and similarly
$$
                \frac{\partial h}{\partial r'} = \iprod{n_K(q-q')}{\tau_K(q)}.
$$
Using the description of the billiard dynamics in Section \ref{Subsection-background-Minkowski-Billiards-2nd}, one easily sees that
% $\phi_K(r,s)=(r',s')$ is and only if $\frac{\partial h}{\partial r} = s$ and $\frac{\partial h}{\partial r'} = -s'$.
\begin{equation*}
\widetilde{\phi}_K(r,s)=(r',s') \Longleftrightarrow\begin{cases}
\frac{\partial h}{\partial r} = s, \\
\frac{\partial h}{\partial r'} = -s'.
\end{cases}
\end{equation*}
In particular, $\widetilde{\phi}_K$ preserves the area form $dr \wedge ds$. The monotone twist condition for this mapping follows from Lemma \ref{lem:mon-s}. Indeed, the twist condition is equivalent to the fact that for a given $q$ the set $\partial T_{+,q}$ is mapped onto the interval $(-1,1)$ in a monotone manner.

%$h$, which is 1-periodic on both parameters, and for $(x,x') \in \partial K \times \partial K$ it is given by the support function associated with the body $T$, i.e., $h(x,x')=h_T(x-x')$.
%Moreover, it is not hard to check that
%\[  \begin{cases}
%              {\dfrac {\partial h} {\partial x} } (x,x') =  \langle B n_K(x), n_T^{\circ}({x-x'}) \rangle= \sin \angle (n_K(x), n_T^{\circ}({x-x'} )) \\ \\
%                    {\dfrac {\partial h} {\partial x'} } (x,x') = ??
%            \end{cases}, \]
%            where $B(z) = {\rm exp}(i \pi /2)z$ is a counter-clockwise rotation and $n_T{\circ}$ means the 0-homogenuous extension of $n_{\partial {T^{\circ}}}$ to ${\mathbb R}^2$.
%
%Moreover, elementary geometry (?) shows that the map $\widetilde \phi_K$ preserves the 2-form ???

\subsubsection{Minkowski Caustics and Invariant Circles} \label{sec:caustics-invariant-circles}

%In this section we relate convex Minkowski caustics to invariant circles of the monotone twist map associated to the billiard ball map, and
%rephrase caustic duality in these terms.  We then describe a natural ``geometric" point of view on caustic duality.

We recall from the previous section that the discrete $T$-billiard map $\phi_K : A_K \rightarrow A_K$ associated with the table $K \subset {\mathbb R}^2_q$ is a monotone twist map.
An invariant circle of $\phi_K$ is an embedded circle in $A_K$ which is homotopically non-trivial and $\phi_K$-invariant.  A classical result by Birkhoff states that any invariant circle of a monotone twist map is the graph of a Lipschitz function from $S^1_P$ to $(-1,1)$ (see Section 1.3 in~\cite{Sib} and the references therein). We note that under  the identification $(\partial K \times \partial T)_+ \simeq A_K$ \eqref{eq:p_to_s}, a $\phi_K$-invariant circle corresponds to an embedded circle in $(\partial K \times \partial T)_+ \subset \partial K \times \partial T$, which is homotopic to either one of the two disjoint circles forming $(\partial K \times \partial T)_0$, and invariant under $\Psi^2$.

Any convex $T$-caustic $C$ in $K$ gives rise to a $\phi_K$-invariant circle, as follows. For $q \in \partial K$, consider the positive tangent line to $C$ from $q$ with direction $v$. %and its tangent vector $v \in T_qK$ at $q$. 
The set of all such pairs $(q,v) \in \cP(K) \simeq (\partial K \times \partial T)_+ $ forms an embedded circle $\Gamma \subset A_K$, which by construction is a graph over $S^1$ and, in particular, not contractible. Moreover, as $C$ is a caustic, $\Gamma$ is $\phi_K$-invariant. We emphasize 
that the converse is false in general, i.e., not every invariant circle corresponds to a {\it convex} caustic (see e.g.,~\cite{GK, Kni}).
Note that given a $\phi_K$-invariant circle $\Gamma \subset A_K \simeq (\partial K \times \partial T)_+$, its image under $\Psi$ defines a $\phi_T$-invariant circle $\Psi(\Gamma) \subset A_T \simeq (\partial K \times \partial T)_-$. This fact allows us to rephrase Definition \ref{def:caustics-duality} in terms of invariant circles.

\begin{corollary}\label{cor:dual-caustic-psi} Let $(K,T)$ be a symmetric billiard configuration in $ {\mathbb R}^2_q \times {\mathbb R}^2_p$. % two smooth centrally symmetric, strictly convex bodies.
 Consider two convex caustics $C \in {\mathfrak C}(K,T)$ and $C' \in {\mathfrak C}(T,K)$, with corresponding invariant circles $\Gamma \subset A_K$ and $\Gamma' \subset A_T$. Then, $C$ and $C'$ are dual caustics if and only  if $\Psi(\Gamma)=\Gamma'$.
\end{corollary}

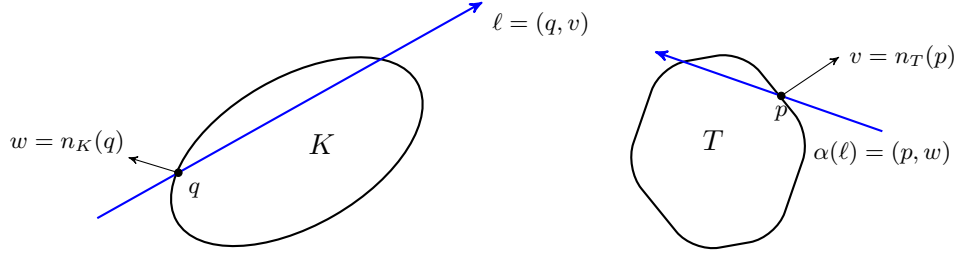
\begin{figure} %[h1]
\begin{center}
\begin{tikzpicture}[scale=0.7]

 \draw[important line][rotate=30] (0,0) ellipse (75pt and 40pt);

 \path coordinate (w1) at (2.1,4*0.75) coordinate (q0) at
 (-4.5*0.5,-2*0.2)  coordinate (q0new) at
 (-4.5*0.5-0.4*3.85,-2*0.2-0.4*2.16)  coordinate (q1) at (1.6,4*0.44) coordinate (q1new) at (1.6+0.5*3.85,4*0.44+0.5*2.16)  coordinate (q2) at
 (1.74,+0.46) coordinate (w2) at (2.6,-1.1) coordinate (w3) at (-3.2,-0.11) coordinate (K) at
 (0,0.15);

%\draw[red] [important line] (q0) -- (q1);
%\draw[->] (q1) -- (w1);
\draw[blue][->] [important line] (q0new) -- (q1new);
%\draw[red] [important line] (q1) -- (q2); %\draw[->] (q2) -- (w2);
\draw[->] (q0) -- (w3);

\filldraw [black]
  (w3) circle (0pt) node[above left=-2.5pt] {{\footnotesize $w = n_K(q)$}}
    %(w1) circle (0pt) node[above] {{\footnotesize $$}}
    (w2) circle (0pt) % node[below ] {{\footnotesize $w_2=\nabla \|q_2\|_K$}}
     (q0) circle (2pt) node[below right] {{\footnotesize $q$}}
%      (q1) circle (2pt) node[below] {{\footnotesize $x$}}
       (q1new) circle (0pt) node[below right] {{\footnotesize ${\ell}=(q,v)$}}
   %    (q2) circle (2pt) node[below=0.5pt] {{\footnotesize $q_0$}}
        (K) circle (0pt) node[right=0.5pt] {${ K}$};

 %      % We start the second graph
       \begin{scope}[xshift=8cm]

 \draw[important line][rounded corners=10pt][rotate=10] (1.8,0) --
 (0.8,1.8)-- (-0.8,1.8)--  (-1.8,0)--  (-0.8,-1.8) -- (0.8,-1.8) --
 cycle;

 \path coordinate (p0new1) at (1.2+.8*2.4 ,2*0.53-.8*0.84) coordinate (p2new1) at (-1.2,1.9) coordinate (p1) at (0.5,-4*0.433) coordinate (p1new) at (0.5-0.3*0.7,-4*0.433-0.3*2.792)  coordinate (np0) at
 (2.3,2*0.9) coordinate (p0) at (1.2,2*0.53) coordinate (p0new) at (1.2+0.6*0.7,2*0.53+0.6*2.792)  coordinate (np1) at
 (0.7,-3) coordinate (p2) at (-1.2,2*0.66) coordinate (D) at
 (0.3,0.22);

 \draw[<-] (np0) node[right] {{\footnotesize $v=n_T(p)$}} -- (p0);
% \draw[<-] (np1) node[right] {{\footnotesize $\nabla \|p_0  \|_{ T}$}} -- (p1);

%\draw[blue][important line] (p0new) -- (p1new);
 \draw[blue][->][important  line] (p0new1) -- (p2new1);

  \filldraw [black]
   %    (p1) circle (2pt) node[below right] {{\footnotesize $p_0$}}
    %     (p2) circle (2pt) node[left] {{\footnotesize $p_2$}}
       (p0new1) circle (0pt) node[below] {{\footnotesize $\alpha({\ell})=(p,w)$}}
         (p0) circle (2pt) node[below] {{\footnotesize $p$}}
          (D) circle (0pt) node[left] {${ T}$};
 \end{scope}
 \end{tikzpicture}

 \caption{\rm The $(K,T)$-duality transform $(q,v) \stackrel{\alpha}{\longrightarrow} (p,w)$.} \label{fig-alpha-transform}
 \end{center}
 \end{figure}

Next we describe a geometric point of view on caustic duality. Let $(K,T)$ as above.
%As before, let $K \subset \R^2_q$ and $T \subset \R^2_p$ be two smooth, centrally symmetric and strictly convex bodies.  We start with describing yet another model for the phase space of $T$-billiard in $K$.
 We consider the space
%of oriented lines in the convex body $K$,
$$
\cL(K) = \{ \ell \subset \R^2 \text{ an oriented line} \,:\, \ell \cap {\rm int } (K) \neq \emptyset \}.
$$
Observe that $\cL(K)$ is naturally identified with the phase space, $\cP(K) \simeq (\partial K \times \partial T)_+$, of the $T$-billiard dynamics in $K$, where the oriented line $\ell \in \cL(K)$ is identified with the pair $(q,v) \in \cP(K)$ such that $v$ is the $h_T$-unit vector in the direction of $\ell$, and $q \in \partial K \cap \ell$ is the first intersection point, that is, $q+\epsilon v \in K$ for small enough $\epsilon > 0$. One has a similar identification $\cL(T) \simeq \cP(T)$. Note that the identification $\cL(K) \simeq \cP(K)$ induces a natural topology on $\cL(K)$.
From these identifications, the Poincar\'{e} map induces a map, which (by a slight abuse of notation) we also denote by $\Psi \colon \cL(K) \to \cL(T)$. Its square, the billiard ball map $\Psi^2$, induces a map $\cL(K) \to \cL(K)$, given by $\ell \mapsto \ell'$, such that the oriented segments $\ell \cap K$ and $\ell' \cap K$ form consecutive segments of a $T$-billiard trajectory in $K$.

Using the identifications $\cL(K) \simeq \cP(K)$ and $\cL(T) \simeq \cP(T)$ implicitly, we define:
\begin{definition} \label{def-the-alpha-transform}
The $(K,T)$-duality map $\alpha = \alpha_{K,T} \colon \cL(K) \to \cL(T)$ is given by
$$
\alpha(q,v) = (p,w),\text{ where $v=n_T(p)$ and $w=n_K(q)$ (see Figure~\ref{fig-alpha-transform})}.
$$
\end{definition}
%$\alpha = \alpha_{K,T} \colon \cL(K) \to \cL(T)$ by $\alpha(q,v) = (p,w)$, where $v=n_T(p)$ and $w=n_K(q)$. 
The duality map $\alpha = \alpha_{T,K} \colon \cL(T) \to \cL(K)$ is defined in a similar manner. %We note that $\alpha^2$ 
The composition of these two maps, which (by a slight abuse of notation) we denote by $\alpha^2$, is the identity map both on $\cL(K) $ and on $\cL(T)$. The next result shows that the maps $\Psi$ and $\alpha$ are essentially the same.

%The billiard dynamics identifies naturally $\cL(K) \simeq (\partial K \times \partial T)_+$ by identifying the oriented line $\ell = (q, v) \in \cL(K)$ with the pair $(q,p) \in (\partial K \times \partial T)_+$, where $v = -n_T(p)$. Similarly, we identify $\cL(T) \simeq (\partial K \times \partial T)_-$ by identifying $\ell= (p,w)$ with $(q,p) \in (\partial K \times \partial T)_-$ where $w = n_K(q)$. In this way, the Poincar\'{e} map $\psi$ determines maps we denote, by a slight abuse of notation, by $\psi: \cL(K) \to \cL(T)$ and $\psi: \cL(T) \to \cL(K)$.
For an oriented line $\ell$ we denote by $-\ell$ the line obtained by reflecting $\ell$ about the origin, and by $\bar{\ell}$ the line $\ell$ with the reversed orientation. We note that these two operations commute, so the notation $-\bar{\ell}$ is unambiguous. \begin{lemma}\label{prop:alpha_psi}
Let $(K,T)$ be a symmetric billiard configuration. Then for $\ell \in \cL(K)$ one has $\alpha(\ell) = \Psi(\bar{\ell})$, and for $\ell \in \cL(T)$ one has $\alpha(\ell) = \Psi(-\bar{\ell})$.
%
%\begin{enumerate}
%\item For $\ell \in \cL(K)$ one has $\alpha(\ell) = \Psi(\bar{\ell})$.
%\item For $\ell \in \cL(T)$ one has $\alpha(\ell) = \Psi(-\bar{\ell})$.
%\end{enumerate}
\end{lemma}

\begin{proof}[{\bf Proof of Lemma~\ref{prop:alpha_psi}}] %$ $
%\begin{enumerate}
%\item 
We shall prove the first assertion, the proof of the second being similar. 
Let $\ell = (q,v) \in \cL(K)$. Consider the composition $\alpha \circ \Psi \colon \cL(K) \to \cL(K)$. By definition, $\Psi(\ell) = (p, w) \in \cL(T)$, for $p,w$ satisfying $v=-n_T(p)$ and $w=n_K(q')$, where $q'$ is the next impact point of the oriented line $\ell$ with $\partial K$. Then $\alpha(\Psi(\ell)) = (q_1, v_1)$, where $q_1$ and $v_1$ satisfy $v_1 = n_T(p)$ and $w = n_K(q_1)$. Hence, $v_1=-v$ and $q_1 =q'$. That is, $\alpha \circ \Psi(\ell) = (q', -v) = \bar{\ell}$. As $\alpha^2$ is the identity, one has $\Psi(\ell) = \alpha(\bar{\ell})$.
%\item The proof is similar, considering the composition $\alpha \circ \Psi \colon \cL(T) \to \cL(T)$.%
%%Let $\ell = (p,w) \in \cL(T)$. Consider this time the composition $\alpha \circ \Psi \colon \cL(T) \to \cL(T)$. By definition, $\Psi(\ell) = (q, v) \in \cL(K)$, where $w=n_K(q)$ and $v=-n_T(p')$. Then $\alpha(\Psi(\ell)) = (p_1, w_1)$ where $v = n_T(p_1)$ and $w_1 = n_K(q)$. Hence,  $w_1 =w$ and $n_T(p_1) = -n_T(p')$, which since $T=-T$ implies $p_1=-p'$ .That is, $\alpha \circ \Psi(\ell) = (-p', w) = -\bar{\ell}$. Again since $\alpha^2$ is the identity, we get $\alpha(-\bar{\ell}) = \Psi({\ell})$.
%\end{enumerate}
\end{proof}
%
%\begin{cor}
%$\alpha$ maps invariant circles to invariant circles.
%\end{cor}
An immediate corollary of Lemma \ref{prop:alpha_psi} which will be used later is the following.
\begin{cor}\label{cor:continuity-of-Psi}%[Continuity of $\Psi$] 
%The map $\Psi$ is continuous in the following sense: 
Suppose that $(K_n,T_n)$ is a sequence of symmetric billiard configurations which converges in the Hausdorff topology to a symmetric billiard configuration $(K,T)$. Suppose moreover that $\ell_n \in \cL(K_n)$ is a sequence of oriented lines which converges to an oriented line $\ell \in\cL(K)$. Then $\Psi_{K_n,T_n}(\ell_n) \to \Psi_{K,T}(\ell)$.
\end{cor}
\begin{proof} [{\bf Proof of Corollary~\ref{cor:continuity-of-Psi}}]
This follows immediately from Proposition \ref{prop:alpha_psi} and the analogous continuity property of the duality map $\alpha$, which holds by the fact that if $D_n$ is a sequence of convex bodies converging (in the Hausdorff topology) to $D$, and $x_n \in \partial D_n$ converge to $x \in \partial D$, then $n_{D_n}(x_n) \to n_D(x)$.
\end{proof}

Using the above identifications ${\mathcal L}(K) \simeq A_K$ and ${\mathcal L}(T) \simeq A_T$, one may consider the $(K,T)$-duality transform $\alpha$ as a map $A_K \to A_T$.  Thus, another immediate corollary of Lemma \ref{prop:alpha_psi} is the following reformulation of Corollary \ref{cor:dual-caustic-psi}:   %we get a characterization of dual caustics in terms of the $(K,T)$-duality map.
\begin{cor} \label{cor:duality-caustic-via-alpha} Let $(K,T) \subset {\mathbb R}^2_q \times {\mathbb R}^2_p$ be a symmetric billiard configuration. %two smooth centrally symmetric, strictly convex bodies.
The $(K,T)$-duality transform $\alpha$ maps invariant circles to invariant circles. Moreover,  given two convex caustics $C \in {\mathfrak C}(K,T)$ and $C' \in {\mathfrak C}(T,K)$ with corresponding invariant circles $\Gamma \subset A_K$ and $\Gamma' \subset A_T$, the caustics $C$ and $C'$ are dual if and only  $\alpha(\Gamma)={ \overline \Gamma}' := \{ \bar \ell' \, : \,  \ell' \in \Gamma' \}$.
\end{cor}

\subsubsection{Parameters of Caustics} \label{Sec:dual-caustics-classical-invariants}

Here we consider some natural parameters associated with caustics, namely, the Lazutkin parameter, the perimeter, and the rotation number. We show that under mild assumptions these parameter coincide for two dual caustics. In fact, these are all marked length spectrum invariants (see e.g., Section 3.2 in~\cite{Sib}), and from the sequel it follows that the monotone twist maps $\phi_K$ and $\phi_T$, for a symmetric billiard configuration $(K,T)$, have identical marked length spectrum.

We start by recalling the notion of rotation number associated with an invariant circle $\Gamma$ of a monotone twist map $\phi \colon A \to A$ (see e.g., Chapter 13 in~\cite{KatHas}, or Chapter 1 in~\cite{Sib}). The restriction $\phi \bigr|_\Gamma :\Gamma
\ \to \Gamma$ defines an orientation preserving homeomorphism of $S^1$, which has a well-defined Poincar\'{e} rotation number $\omega \in [0,1)$. The rotation number of $\Gamma$ is by definition $\omega$. It can be computed as
$$
\omega = \lim_{n \to\infty } \frac{r_n}{n},
$$
where $(r_n, s_n) = \widetilde{\phi}^n (r_0,s_0)\in \R \times (-1,1)$ is the trajectory of an arbitrary point $(r_0, s_0)$ lying on the lift of $\Gamma$.  The rotation number of a convex caustic is defined to be the rotation number of the associated invariant circle. We recall from Section \ref{sec:string-const} that to any convex $T$-caustic $C$ in $K$ one can associate its Lazutkin parameter (see Remark~\ref{rmk:Lazutkin-parameter}). Finally, by the perimeter of the $T$-caustic $C$ we mean the $h_T$-perimeter of $C$.

\begin{figure} %[h1]
\begin{center}
\begin{tikzpicture}[scale=0.90]

 \draw[important line][rounded corners=15pt][rotate=0] (2,0) -- (1.6,1.1) --
 (1,2)--  (0,2.1) -- (-1,2)--  (-1.6, 1.1) -- (-2,0)--  (-1.6,-1.1) -- (-1,-2) -- (0,-2.1) -- (1,-2) -- (1.6, -1.1) --
 cycle;

 \path coordinate (p1) at (-0.85,-2*0.25+0.24) coordinate (np0) at
 (2.3,2*0.9) coordinate (p0) at (1.2-0.55,2*0.53-0.45) coordinate (np1) at
 (0.7,-3) coordinate (p2) at (-1.2,2*0.66+0.38)  coordinate (D) at
 (0.3,0.22) coordinate (E) at (0.3,-1.7);

%\draw[blue][important line] (p2) -- (p1);
% \draw[blue][important  line] (p0) -- (p2);

\draw[blue][important line][dashed] (-1.20,-1.7) -- (0.52,-0.37);
\draw[blue][->][important line] (-1.20,-1.7) -- (-1.2+ 0.3*1.72,-1.7+0.3*1.33);

  \filldraw [black]
   %    (p1) circle (1pt) node[left] {{\footnotesize $e(q)$}}
      %   (p2) circle (1pt) node[above left] {{\footnotesize $q$}}
        % (p0) circle (1pt) node[above=2pt] {{\footnotesize $b(q)$}}
           (-1.20,-1.7) circle (1pt) node[left] {{\footnotesize $\gamma(t)$}}
              (-1.2+ 0.3*1.72,-1.7+0.3*1.33) circle (1pt) node[above] {{\footnotesize $v(t)$}}

             (E) circle (0pt) node[left] {${ K}$}
          (D) circle (0pt) node[left] {${ C}$};

\pgfmathsetmacro{\a}{0.3*3}
      \pgfmathsetmacro{\b}{0.3*2}
      \pgfmathsetmacro{\c}{sqrt(\a^2 - \b^2)}
      \pgfmathsetmacro{\aa}{2.1}
      \pgfmathsetmacro{\bb}{1.1}
      \pgfmathsetmacro{\cc}{sqrt(\aa^2 - \bb^2)}

            \draw[important line][rotate=30] (0,0) ellipse [x radius=\a, y radius=\b];

 \end{tikzpicture}

 \caption{}
 %On the left, the parametrization of the boundary $\partial C$ as $\gamma(t) + \lambda(t) v(t)$. On the right, the tangency points $A_n$ between the consecutive points of the trajectory.} 
 \label{fig-Siburg1}
 \end{center}
 \end{figure}

\begin{proposition}\label{prop:dual-invariants}
Suppose that the symmetric billiard configuration $(K,T)$ admits a pair of dual convex caustics $C \in {\mathfrak C}(K,T)$ and $C' \in {\mathfrak C}(T,K)$. Then the caustics $C$ and $C'$ have the same rotation number. Moreover, if we assume further that $C$ and $C'$ are $C^1$-smooth and strictly convex, then $C$ and $C'$ have equal Lazutkin parameters, and perimeters.
\end{proposition}

To prove Proposition \ref{prop:dual-invariants} we find formulas to compute these parameters. We begin with the following proposition (see Theorem 3.2.10 in~\cite{Sib} for the Euclidean case, cf.~\cite{Am}). 
\begin{proposition}\label{prop:perimieter-sdt}
Let $C \subset K$ be a $C^1$-smooth, strictly convex $T$-caustic, and let $\Gamma \subset A_K$ be the corresponding invariant circle. Then the $h_T$-perimeter of $C$ is given by
$$
\Per_{h_T}(C) = -\int_{\Gamma} sdt,
$$
where $s$ is the function appearing in the graph representation $\Gamma = \{(t,s(t)) \, : \, t \in S^1\}$.
\end{proposition}

\begin{proof}[{\bf Proof of Proposition~\ref{prop:perimieter-sdt}}]
%Since the invariant circle $\Gamma$ is a graph over $S^1$, one may write $\Gamma = \{(t,s(t)) \, : \, t \in S^1\}$. 
Using the identification $A_K \simeq (\partial K \times \partial T)_+$ one has that  $\Gamma = \{(q,p(q)) : q \in \partial K\} $. Then the oriented line corresponding to $\Gamma$ emanating from $q \in \partial K$ has direction $v(q) = n_T(p(q))$. Next, let $\gamma(t)$ be an $h_T$-unit speed parametrization of $\partial K$, i.e., $ \dot \gamma(t) = \tau_K(\gamma(t))$ (where $h_T(\tau_K(q)) = 1$), and abbreviate $v(t)=v(\gamma(t))$. By the definition of a convex caustic, the line $\gamma(t) + {\mathbb R}_{+}v(t)$ is tangent to $C$. Therefore, for some $\lambda(t)>0$, the point $\gamma(t)+\lambda(t)v(t)$ belongs to the boundary of $C$. We thus get a parametrization of $\partial C$ in the form $C(t)=\gamma(t)+\lambda(t)v(t)$ (see Figure~\ref{fig-Siburg1}).
Hence, \begin{equation} \label{eq-C-dot} \dot C(t) = \tau_K(\gamma (t)) + \dot \lambda(t) v(t) + \dot v(t) \lambda(t). \end{equation}
Note that as $\dot C(t) || v(t)$ (since the line $\gamma(t) + {\mathbb R}_{+}v(t)$ is tangent to the caustic $C$), 
\begin{equation} \label{eq-h_T-C-dot} h_T(\dot C(t)) = \iprod{\dot C(t)}{ n_{T^{\circ}}(v(t))}. \end{equation}
Combining relations~\eqref{eq-C-dot} and~\eqref{eq-h_T-C-dot}, and since %recalling that by definition $v(q) = n_T(p(q))$, and thus 
$n_{T^{\circ}}(v(t))= p(t)$, one obtains
\begin{align*} h_T( \dot C(t))  %& = \iprod{\tau_K(t) + \dot \alpha(t) v(t) + \dot v(t) \alpha(t)}{y(t)} \\
& =  \iprod{\tau_K(\gamma (t)) }{p(t)} + \dot \lambda(t) \iprod{ v(t)}{n_{T^{\circ}}(v(t))} + \lambda(t) \iprod{\dot v(t)}{n_{T^{\circ}}(v(t))}. \end{align*}
Note that $v(t) \in \partial T^{\circ}$, and thus $\iprod{v(t)}{n_{T^{\circ}}(v(t))} = h_T(v(t)) = 1$. Moreover, since $\dot v(t)$ is tangent to $\partial T^{\circ}$, one has
$ \iprod {\dot v(t)} {n_{T^{\circ}}(v(t))} = 0$. Finally, by the identification $(\partial K \times \partial T)_{+} \simeq A_K$ given by \eqref{eq:p_to_s}, one has $\iprod{\tau_K(\gamma (t))}{p(t)} = -s(t)$, and hence
$$ h_T(\dot C(t))  = -s(t)+\dot \lambda(t).$$
Integrating, we get the desired equality
$$
\Per_{h_T}(C) = \int_{S^1} h_T(\dot{C}(t)) =\int_{S^1} (-s(t)+\dot \lambda(t))dt = -\int_{\Gamma} sdt,
$$
and the proof of the proposition is thus complete. 
\end{proof}

We next recall the definition of the ``minimal action" $\beta$ associated with an invariant circle of a monotone twist map $\phi$ with generating function $h$ (see e.g., Chapter 1 in~\cite{Sib}):
\begin{equation} \label{eq:minimal-action}
\beta := \lim_{N \to \infty } \frac{1}{2N}\sum_{n=-N}^{N-1} h(r_n, r_{n+1}),
\end{equation}
where $(r_k, s_k) = \widetilde{\phi}^k (r_0,s_0)\in \R \times (-1,1)$ is the trajectory of an arbitrary point $(r_0, s_0)$ lying on the lift of $\Gamma$ to $\R \times (-1,1)$. The minimal action of a convex caustic is defined to be the minimal action of the corresponding invariant circle.

The next proposition relates the minimal action, Lazutkin parameter, perimeter, and rotation number of a convex caustic (see~\cite{Sib} for the Euclidean case, cf.~\cite{Am}).

\begin{proposition}\label{prop:string-action-rotation}
For a symmetric billiard configuration $(K,T)$, let $C \in {\mathfrak C}(K,T)$ with rotation number $\omega$, minimal action $\beta$, and Lazutkin parameter $L$. Then, one has
	\begin{equation} \label{eq-all-parameters} L = - \beta -\omega \cdot \Per_{h_T}(C). \end{equation}
\end{proposition}

\begin{figure} %[h1]
\begin{center}
\begin{tikzpicture}[scale=0.90]

% \draw[important line][rounded corners=15pt][rotate=0] (2,0) -- (1.6,1.1) --
% (1,2)--  (0,2.1) -- (-1,2)--  (-1.6, 1.1) -- (-2,0)--  (-1.6,-1.1) -- (-1,-2) -- (0,-2.1) -- (1,-2) -- (1.6, -1.1) --
% cycle;
%
% \path coordinate (p1) at (-0.85,-2*0.25+0.24) coordinate (np0) at
% (2.3,2*0.9) coordinate (p0) at (1.2-0.55,2*0.53-0.45) coordinate (np1) at
% (0.7,-3) coordinate (p2) at (-1.2,2*0.66+0.38)  coordinate (D) at
% (0.3,0.22) coordinate (E) at (0.3,-1.7);
%
%
%
%%\draw[blue][important line] (p2) -- (p1);
%% \draw[blue][important  line] (p0) -- (p2);
%
%\draw[blue][important line][dashed] (-1.20,-1.7) -- (0.52,-0.37);
%\draw[blue][->][important line] (-1.20,-1.7) -- (-1.2+ 0.3*1.72,-1.7+0.3*1.33);
%
%  \filldraw [black]
%   %    (p1) circle (1pt) node[left] {{\footnotesize $e(q)$}}
%      %   (p2) circle (1pt) node[above left] {{\footnotesize $q$}}
%        % (p0) circle (1pt) node[above=2pt] {{\footnotesize $b(q)$}}
%           (-1.20,-1.7) circle (1pt) node[left] {{\footnotesize $\gamma(t)$}}
%              (-1.2+ 0.3*1.72,-1.7+0.3*1.33) circle (1pt) node[above] {{\footnotesize $v(t)$}}
%
%             (E) circle (0pt) node[left] {${ K}$}
%          (D) circle (0pt) node[left] {${ C}$};
%
%
%
%\pgfmathsetmacro{\a}{0.3*3}
%      \pgfmathsetmacro{\b}{0.3*2}
%      \pgfmathsetmacro{\c}{sqrt(\a^2 - \b^2)}
%      \pgfmathsetmacro{\aa}{2.1}
%      \pgfmathsetmacro{\bb}{1.1}
%      \pgfmathsetmacro{\cc}{sqrt(\aa^2 - \bb^2)}
%
%            \draw[important line][rotate=30] (0,0) ellipse [x radius=\a, y radius=\b];
%
%       \begin{scope}[xshift=8cm]

\pgfmathsetmacro{\a}{3}
      \pgfmathsetmacro{\b}{2}
      \pgfmathsetmacro{\c}{sqrt(\a^2 - \b^2)}
      \pgfmathsetmacro{\aa}{2.1}
      \pgfmathsetmacro{\bb}{1.1}
      \pgfmathsetmacro{\cc}{sqrt(\aa^2 - \bb^2)}
     % \draw (-\c, 0) -- (\c, 0);
  %    \filldraw[draw=black] (-\c, 0) circle (.05cm); %empty focus
   %   \filldraw[important line] (\c, 0) circle (.05cm) node[below] {{\tiny $(f,0)$}}; %focus
    %   \filldraw    (30:{\a} and {\b})  circle (1pt) node[above right=0.1pt] {{\tiny $ A^{-1}x$}};
       %    \filldraw[draw=black] (-\c+1,3) circle (0cm) node[right] {{\tiny $l=(x,v)$}}; %empty focus
  % \filldraw[draw=black] (-\c+0.55,-\c+3.9) circle (0.05cm) node[above left] $3$;
   % \filldraw[draw=black] (-\c+0.12,1.1) circle (0.05cm) node[right] {{\tiny $q_0$}};
     \filldraw[draw=black] (-\c+0.12,-1.4) circle (0.05cm) node[ left] {{\tiny $q_{n-1}$}};
       \filldraw[draw=black] (-\c+0.12 + 0.55*4.86,-1.4+0.55*0.58) circle (0.05cm) node[below] {{\tiny $A_{n-1}$}};
       \filldraw[draw=black] (2.75,-0.82) circle (0.05cm) node[below right] {{\tiny $q_n$}};
        \filldraw[draw=black] (2.75-0.42*1.65,-0.82+0.42*2.67) circle (0.05cm) node[right] {{\tiny $A_{n}$}};
           \filldraw[draw=black] (1.1,1.85) circle (0.05cm) node[above] {{\tiny $q_{n+1}$}};
                     \filldraw[draw=black] (-2.1,2) circle (0.00cm) node[] {{ ${K}$}};
            \draw[important line] (0,0) ellipse [x radius=\a, y radius=\b];
             \draw[important line] (0,0) ellipse [x radius=\aa, y radius=\bb];
    %  \draw [-latex] (\c,0) -- (30:{\a} and {\b});
    %  \draw [] (\c,0) -- (30:{\a} and {\b});
     %  \draw [] (-\c,0) -- (30:{\a} and {\b});
%\draw[red] [important line] (-\c+0.12,1.1) -- (-\c+0.12,-1.4);
\draw[blue] [important line] (-\c+0.12,-1.4) -- (2.75,-0.82);
\draw[blue] [important line] (2.75,-0.82) -- (1.1,1.85);

% \end{scope}
 \end{tikzpicture}

 \caption{}
 %On the left, the parametrization of the boundary $\partial C$ as $\gamma(t) + \lambda(t) v(t)$. On the right, the tangency points $A_n$ between the consecutive points of the trajectory.} 
 \label{fig-Siburg}
 \end{center}
 \end{figure}

\begin{proof}[{\bf Proof of Proposition~\ref{prop:string-action-rotation}}]
Denote by $\Gamma$ the invariant circle associated with the convex $T$-caustic $C$. 
	Let $(r_n, s_n )$, ${n \in {\mathbb Z}}$, be any $\widetilde{\phi}_K$-trajectory lying on (the lift of) $\Gamma$, and let  $q_n = \gamma_K(r_n) \in \partial K$ be the corresponding $T$-billiard trajectory in $K$. Let $A_n \in C$ be the tangency point between $q_n$ and $q_{n+1}$ (see Figure~\ref{fig-Siburg}). By the definition of the Lazuktin parameter,
	\begin{align*}
	 L &=  \Per_{h_T}(\conv(q_n, C))-\Per_{h_T}(C) \\
	 &= h_T(A_{n}-q_n) + h_T(q_n-A_{n-1})  - \length_{h_T}({\rm arc}({A_{n-1}A_n})).
	\end{align*}
	Summing over $-N \leq n \leq N-1$ gives
	$$ 2NL = \sum_{n=-N}^{N-1} h_T(q_{n+1}-q_n) + O(1)  -\sum_{n=-N}^{N-1} \length_{h_T} ({\rm arc} (A_nA_{n+1})).$$
	Thus, using the definition of the generating function \eqref{eq:generating_function}, and \eqref{eq:minimal-action}, one has
	\begin{align*}
	L &= \frac{1}{2N}\sum_{n=-N}^{N-1} h_T(q_{n+1}-q_n) + O\left( \frac{1}{N}\right)  -\frac{1}{2N}\sum_{n=-N}^{N-1} \length_{h_T} (arc (A_nA_{n+1})) \\
	& \xrightarrow[N\to\infty]{}-\beta  -\omega\cdot \Per_{h_T} (C). %=-\beta  -\omega \cdot \Per_{h_T} (C) .
	\end{align*}
	This completes the proof.
\end{proof}

Finally we are in a position to prove that dual caustics have equal parameters. %the main results of this section.

\begin{proof}[{\bf Proof of Proposition~\ref{prop:dual-invariants}}] We first note that the dual caustics $C \subset K$ and $C' \subset T$ have equal rotation numbers. Indeed, this holds as the map $\Psi\bigr|_{\Gamma} \colon \Gamma \to \Gamma'$ induces an orientation preserving diffeomorphism conjugating the circle homeomorphisms $\phi_K\big|_\Gamma$ and $\phi_T \bigr|_{\Gamma'}$, which therefore have the same rotation number (see Chapter 11 in~\cite{KatHas}).

Next, we prove that $C$ and $C'$ have equal minimal actions. To this end, we must show that if $(q_n)_{n\in \Z} \subset \partial K$ is a $T$-billiard trajectory in $K$ which is tangent to $C$, and $(p_n)_{n\in\Z} \subset \partial T$ is its dual trajectory (which is by definition tangent to $C'$), then
\begin{equation}\label{eq:equal-action}
\lim_{N \to \infty} \frac{1}{N} \sum_{n=-N}^{N-1} h_T(q_n- q_{n+1}) = \lim_{N \to \infty} \frac{1}{N} \sum_{n=-N}^{N-1} h_K(p_n- p_{n+1}).
\end{equation}
Indeed, if one (and hence both) of the trajectories is periodic, this holds since  dual periodic trajectories have equal lengths (see Remark \ref{rmk:Role-of-K-T-interchangable}). The general case now follows, as given two non-periodic dual trajectories, for any fixed $N$ one may close the characteristic curve corresponding to the trajectory $\{(q_n, p_n)\}_{n=-N}^{N-1}$ in $\partial (K \times T)$ with a short segment, obtaining $$ \sum_{n=-N}^{N-1} h_T(q_n- q_{n+1}) =   \sum_{n=-N}^{N-1} h_K(p_n- p_{n+1})+ O(1),$$ which proves \eqref{eq:equal-action}.

From this point we assume that $C$ and $C'$ are $C^1$-smooth and strictly convex.
We prove next that $C$ and $C'$ have equal perimeters. . We write, by a slight abuse of notation, $\Psi \colon A_K \to A_T$ for the map induced by $\Psi$ under the identifications $(\partial K \times \partial T)_+ \simeq A_K$ and $(\partial K \times \partial T)_- \simeq A_T$. Since $\Psi$ preserves the standard symplectic form $dq \wedge dp$, and the latter is identified with the form $ds\wedge dt $ on the two cylinders, we deduce that this map is symplectic, i.e., $\Psi^* (ds \wedge dt)=ds \wedge dt$. Note additionally that $\Psi$ extends continuously to a map between the closed cylinders $\Psi \colon \overline{A}_K \to \overline{A}_T$, where for example $\overline{A}_K = S^1 \times [-1,1]$, and this extension maps the boundary circle $\{s=-1\} \subset \overline{A}_K$ to $\{s=-1\} \subset \overline{A}_T$.
Now, let $\Gamma$ and $\Gamma'$ be invariant circles associated with $C$ and $C'$, respectively. As these invariant circles are graphs, we may write
$$
\Gamma = \{(t,s) \in A_K \,:\, s=f(t)\}, \quad \Gamma' = \{(t,s) \in A_T \,:\, s=g(t) \}.
$$
Consider the domains
$$
D = \{(t,s) \in A_K\, :\, -1<s<f(t) \} ,\quad D' = \{(t,s) \in A_T\, :\, -1<s<g(t) \}.
$$
Since $C$ and $C'$ are dual caustics, $\Psi(\Gamma) = \Gamma'$ by Corollary \ref{cor:dual-caustic-psi}. Since, as noted above, $\Psi$ maps the boundary circle $\{s=-1\}$ of $\overline{A}_K$ to that of $\overline{A}_T$, it follows that $\Psi(D) = D'$. Therefore,
\begin{equation}\label{eq:int_D=int_D'}
\int_D ds \wedge dt =\int_D \Psi^*(ds \wedge dt)=  \int_{D'} ds \wedge dt.
\end{equation}
Using Proposition \ref{prop:perimieter-sdt} and Stokes' Theorem, one obtains
\begin{equation}\label{eq:int_D}
\int_D ds \wedge dt = \int_\Gamma sdt - \int_{\{s=-1\}} sdt = \Per_{h_T}(C) + \int_{S^1} dt = \Per_{h_T}(C) + \Per_{h_T}(K).
\end{equation}
By a similar computation,
\begin{equation}\label{eq:int_D'}
\int_{D'} ds \wedge dt = \Per_{h_K}(C') + \Per_{h_K}(T).
\end{equation}
Finally, since $\Per_{h_T}(K) = \Per_{h_K}(T)$ (see formula \eqref{eq:per_moon}), relations \eqref{eq:int_D=int_D'}, \eqref{eq:int_D} and \eqref{eq:int_D'} imply that $\Per_{h_T}(C) = \Per_{h_K}(C')$, as required.

To conclude, we  observe that $C$ and $C'$ have equal Lazutkin parameters. This follows immediately from \eqref{eq-all-parameters}, and the equality of the other parameters appearing there, which has already been established above.
\end{proof}

\begin{remark} {\rm
%We remark that all the invariants mentioned in this section, associated with a $T$-caustic $C$ in $K$, can be computed from the marked length spectrum function associated to the monotone twist map $\phi_K$. 
Proposition \ref{prop:dual-invariants} is a special case of the fact that the monotone twist maps $\phi_K$ and $\phi_T$ have equal marked length spectrum functions (see e.g.,~\cite{Sib} for more details). }
\end{remark}

\section{Existence of Dual Caustics} \label{sec:existence-of-dual-caustic}

In this section we prove Theorem~\ref{thm:main}. We shall start with the special case where the body $K$ is an ellipse. Although this example is quite elementary, it is pivotal in our understanding of Euclidean caustics, since, roughly speaking, every Euclidean string construction is locally an ellipse. This example will play a role both in our proof of 
the theorem under regularity assumptions on the caustic (Section \ref{sect:smooth-caustics}), and in the proof
provided in the Appendix which pertains to the case of polygonal caustics. The case of a general (Euclidean) caustic is then settled by an approximation argument (Section \ref{sect:general-caustics}).

%provides us with a natural mechanism to prove the theorem under regularity assumptions on the caustic (Section \ref{sect:smooth-caustics}), and the general case is then settled by an approximation argument (Section \ref{sect:general-caustics}).
%prove the theorem for a convex polygon (Section~\ref{sec:polygonal-case}), and the general case is then settled by an approximation argument (Section~\ref{sec:proof-of-general-case-main-theorem}).

\subsection{The Ellipse Case}\label{sec:ellipse}
It is well known (see e.g.,~\cite{T}) that confocal ellipses contained in an ellipse $\cE$ are convex caustics for the Euclidean billiard in $\cE$.
We show, by a straightforward computation, that a Euclidean caustic in an ellipse has a dual caustic in the unit disk $B$, which is itself an ellipse. %More precisely,
\begin{proposition}\label{prop:ellipse}
Let $\cE \subset \R^2_q$ be a Euclidean ellipse, given by $A \cE=B$ for some positive definite symmetric $A\in {\rm GL}(2)$, where $B$ is the Euclidean ball. Let $C\subset \cE$ be a confocal ellipse. Then $C':=AC$ is its
 dual convex $\cE$-caustic in  $B \subset \R^2_p$ (see Figure~\ref{fig-dual-ellipse-with-lines}).
\end{proposition}

While it is relatively easy to check that $C'$ is an $\cE$-caustic in $B$, using that a linear transformation maps lines to lines and changes length in a predictable way (in fact, for a convex set $D$, the $h_\cE$-perimeter of $AD$ is the Euclidean perimeter of $D$), Proposition~\ref{prop:ellipse} asserts something slightly stronger -- not only is it a caustic, but it is the caustic \emph{dual to $C$}. The proof is elementary, and amounts to the following computation.

\begin{figure} %\label{fig-ellipse-circle}
\begin{center}
\begin{tikzpicture}[scale=0.7]

\pgfmathsetmacro{\a}{3}
      \pgfmathsetmacro{\b}{2}
      \pgfmathsetmacro{\c}{sqrt(\a^2 - \b^2)}
      \pgfmathsetmacro{\aa}{2.1}
      \pgfmathsetmacro{\bb}{1.1}
      \pgfmathsetmacro{\cc}{sqrt(\aa^2 - \bb^2)}
     % \draw (-\c, 0) -- (\c, 0);
  %    \filldraw[draw=black] (-\c, 0) circle (.05cm); %empty focus
   %   \filldraw[important line] (\c, 0) circle (.05cm) node[below] {{\tiny $(f,0)$}}; %focus
    %   \filldraw    (30:{\a} and {\b})  circle (1pt) node[above right=0.1pt] {{\tiny $ A^{-1}x$}};
       %    \filldraw[draw=black] (-\c+1,3) circle (0cm) node[right] {{\tiny $l=(x,v)$}}; %empty focus
  % \filldraw[draw=black] (-\c+0.55,-\c+3.9) circle (0.05cm) node[above left] $3$;
   % \filldraw[draw=black] (-\c+0.12,1.1) circle (0.05cm) node[right] {{\tiny $q_0$}};
     \filldraw[draw=black] (-\c+0.12,-1.4) circle (0.05cm) node[below left] {{\tiny $q$}};
    %   \filldraw[draw=black] (2.75,-0.82) circle (0.05cm) node[below] {{\tiny $q_2$}};
      %     \filldraw[draw=black] (1.1,1.85) circle (0.05cm) node[above] {{\tiny $q_3$}};
                     \filldraw[draw=black] (-2.1,2) circle (0.00cm) node[] {{ ${\mathcal E}$}};
                        \filldraw[draw=red] (0.4,-1.4) circle (0.00cm) node[right] {{\tiny $\ell$}};
                             \filldraw[draw=black] (0,0) circle (0.00cm) node[] {{ ${C}$}};

            \draw[important line] (0,0) ellipse [x radius=\a, y radius=\b];
             \draw[important line] (0,0) ellipse [x radius=\aa, y radius=\bb];
    %  \draw [-latex] (\c,0) -- (30:{\a} and {\b});
    %  \draw [] (\c,0) -- (30:{\a} and {\b});
     %  \draw [] (-\c,0) -- (30:{\a} and {\b});
%\draw[red] [important line] (-\c+0.12,1.1) -- (-\c+0.12,-1.4);
\draw[blue][->] [important line] (-\c+0.12,-1.4) -- (-\c+0.12+ 0.2*4.87, -1.4+0.2*0.58);
\draw[blue][dashed] [important line] (-\c+0.12,-1.4) -- (2.75,-0.82);
%\draw[red] [important line] (2.75,-0.82) -- (1.1,1.85);
       \begin{scope}[xshift=8cm]

\pgfmathsetmacro{\aa}{2}
      \pgfmathsetmacro{\bb}{2}
      \pgfmathsetmacro{\cc}{sqrt(\aa^2 - \bb^2)}
        \pgfmathsetmacro{\aaa}{1.38}
      \pgfmathsetmacro{\bbb}{0.98}
      \pgfmathsetmacro{\ccc}{sqrt(\aaa^2 - \bbb^2)}

   %   \draw (-0.745, 0) -- (0.745, 0);
   %   \filldraw[draw=black] (-\c, 0) circle (.05cm); %empty focus
     % \filldraw[important line] (\c, 0) circle (.05cm); %focus
      \draw[important line] (0,0) ellipse [x radius=\aa, y radius=\bb];
         \draw[important line] (0,0) ellipse [x radius=\aaa, y radius=\bbb];
    %  \draw [-latex] (\c,0) -- (30:{\a} and {\b});
 %     \draw [] (\c,0) -- (110:{\a} and {\b});
    %   \draw [] (-\c,0) -- (110:{\a} and {\b});
%\filldraw [black]
 % (0.745, 0) circle (1pt) node[below] {{\tiny $({e},0)$}}
%(50:{\aa} and {\bb})  circle (1pt) node[above right=0.1pt] {{\small $x$}}

   %(-0.745, 0) circle (1pt); % node[below] {{\tiny $({-e},0)$}};
 %    \draw [] (0.745,0) -- (50:{\aa} and {\bb});
    %   \draw [] (-0.745,0) -- (50:{\aa} and {\bb});
     %  \draw[blue] [important line] (0,2) -- (-1.95,-0.4);
        %   \draw[blue] [important line] (-1.95,-0.4) -- (1.2,-1.6);
           %           \draw[blue] [important line] (1.2,-1.6) -- (1.5,0.9);
       %    \filldraw[draw=black] (0,2) circle (0.05cm) node[ above] {{\tiny $p_0$}};
          %   \filldraw[draw=black] (-1.95,-0.4) circle (0.05cm) node[ left] {{\tiny $p_1$}};
                 \filldraw[draw=black] (1.95,0.4) circle (0.05cm) node[right] {{\tiny $p$}};
                    \filldraw[draw=blue] (0.4,-1.4) circle (0.00cm) node[right] {{\tiny $\ell'$}};
                   \filldraw[draw=black] (-1.8,1.7) circle (0.00cm) node[] {{ ${B}$}};
                             \filldraw[draw=black] (0,0) circle (0.00cm) node[] {{ ${C'}$}};

             %     \filldraw[draw=black] (1.20,-1.6) circle (0.05cm) node[ below] {{\tiny $p_2$}};
                %        \filldraw[draw=black] (1.5,0.9) circle (0.05cm) node[ left = -2pt] {{\tiny $p_3$}};
%\draw[red] [important line] (2.75-2,-0.82-2) -- (1.1-2,1.85-2);
   %      \filldraw[draw=black] (-\cc+0.25,3) circle (0cm) node[right] {{\tiny $\alpha(l)=l'=(y,w)$}}; %empty focus
     % \filldraw[draw=black] (-\cc-0.1,-\cc+2) circle (0.05cm) node[above left] {{\tiny $y$}};
   
\draw[blue][->] [important line] (1.95,0.4) -- (1.95 - 0.8*0.705,0.4 - 0.8*0.85) ;
\draw[blue][dashed][important line] (1.95,0.4) -- (1.95 - 2.8*0.705,0.4 - 2.8*0.85) ;

 \end{scope}
 \end{tikzpicture}

 \caption{\rm A pair of dual caustics for the Euclidean billiard in an ellipse. 
 %  $({\mathcal E},B)$-billiard. Here ${\mathcal E} = 
  % \{ {\frac {q_1^2} {a^2}} + {\frac {q_2^2} {b^2} } \leq 1 \} \subset {\mathbb R}^2_q$, and $B \subset {\mathbb R}^2_p$ is the Euclidean unit ball. 
 % The dual ${\mathcal E}$-caustic in $B$ of a confocal ellipse $ \{ {\frac {q_1^2} {a^2-\lambda}} + {\frac {q_2^2} {b^2-\lambda} } \leq 1 \}$ (which is a $B$-caustic in ${\mathcal E}$)  is the ellipse $ \{ {\frac {a_1^2p_1^2} {a^2-\lambda}} + {\frac {b_2^2p_2^2} {b^2-\lambda} } \leq 1 \} \subset {\mathbb R}^2_p$ (see ?? below).
}    \label{fig-dual-ellipse-with-lines}
 \end{center} \end{figure}
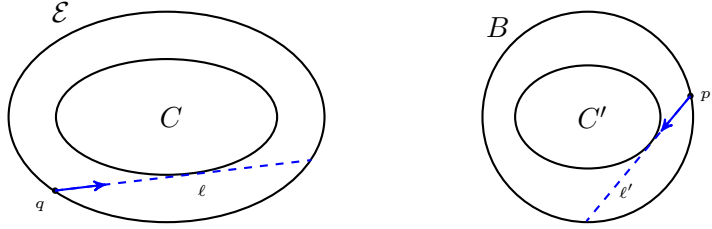

 \begin{proof}[{\bf Proof of Proposition~\ref{prop:ellipse}}]

 Assume without loss of generality that the matrix $A$ is diagonal, i.e., $$A = \begin{pmatrix}
1/a & 0 \\ 0 & 1/b
\end{pmatrix}, \ \text {\ with $0 < b \leq a$.}$$
 In this case
$$
\cE = \left\{ \frac{q_1^2}{a^2} + \frac{q_2^2}{b^2} \leq 1 \right\}.
$$
A confocal ellipse $C \subset \cE$ is then of the form
$$
C = \left\{\frac{q_1^2}{a^2 - \lambda} + \frac{q_2^2}{b^2 - \lambda} \leq 1\right\}, \ {\rm where} \ 0< \lambda < b^2.
$$
Then we must show that the dual $\cE$-caustic in $B$ is
$$
C'= AC = \left\{\frac{a^2p_1^2}{a^2-\lambda} + \frac{b^2p_2^2}{b^2-\lambda} \leq 1 \right\}.
$$
Indeed, let $q = (q_1, q_2) \in \partial \cE$, and consider a line $\ell = q+ \R \cdot p \in \cL(\cE)$ emanating from $q$ with direction $p = (p_1,p_2) \in \partial B$. The line $\ell$ is mapped under the $(\cE,B)$-duality map $\alpha$ (Definition~\ref{def-the-alpha-transform}) to the line $\ell' = p + \R \cdot n_{\cE}(q) \in \cL(B)$ (see Figure~\ref{fig-dual-ellipse-with-lines}). From Corollary~\ref{cor:duality-caustic-via-alpha} it follows that in order to prove our claim, it is enough to check that $\ell$ is tangent to $C$ if and only if $\ell'$ is tangent to $C'$. The intersection  $\ell \cap \partial C$ corresponds to solutions $t\in\R$ of the equation
\begin{equation}\label{eq:intersect1}
\frac{(q_1 + t p_1)^2}{a^2-\lambda} + \frac{(q_2+tp_2)^2}{b^2-\lambda} = 1.
\end{equation}
Expanding, and using the fact that $q \in \partial \cE$, we rewrite \eqref{eq:intersect1} as
\begin{equation}\label{eq:intersection1-simple}
\left(\frac{p_1^2}{a^2-\lambda}+\frac{p_2^2}{b^2-\lambda}\right)t^2+2\left(\frac{q_1 p_1}{a^2-\lambda}+\frac{q_2 p_2}{b^2-\lambda }\right)t+\lambda\left(\frac{q_1^2}{a^2(a^2-\lambda)}+\frac{q_2^2}{b^2(b^2-\lambda)}\right)=0.
\end{equation}
Next, note that the outer normal $n_\cE(q)$ is given, up to rescaling, by the vector
$
n=\left(\frac{q_1}{a^2}, \frac{q_2}{b^2}\right),
$
and so the intersection $\ell' \cap \partial C'$ corresponds to solutions $t\in\R$ of
\begin{equation}\label{eq:intersect2}
\frac{a^2(p_1 + tq_1/a^2)^2}{a^2-\lambda} + \frac{b^2(p_2+tq_2/b^2)^2}{b^2-\lambda} = 1.
\end{equation}
Using the fact that $p \in \partial B$, we rewrite \eqref{eq:intersect2} as
\begin{equation}\label{eq:intersection2-simple}
\left(\frac{q_1^2}{a^2(a^2-\lambda)}+\frac{q_2^2}{b^2(b^2-\lambda)}\right)t^2 + 2\left(\frac{q_1 p_1}{a^2-\lambda}+\frac{q_2 p_2}{b^2-\lambda }\right)t + \lambda \left(\frac{p_1^2}{a^2-\lambda} + \frac{p_2^2}{b^2-\lambda}\right)=0.
\end{equation}
Note that the equations \eqref{eq:intersection1-simple} and \eqref{eq:intersection2-simple} have the same discriminant
$$
\Delta = 4\left(\frac{q_1 p_1}{a^2-\lambda}+\frac{q_2 p_2}{b^2-\lambda }\right)^2-4\lambda \left(\frac{p_2^2}{a^2-\lambda}+\frac{p_2^2}{b^2-\lambda}\right) \left(\frac{q_1^2}{a^2(a^2-\lambda)}+\frac{q_2^2}{b^2(b^2-\lambda)}\right).
$$
In particular, the condition $\Delta=0$ is the same for both equations. We deduce that equation \eqref{eq:intersection1-simple} has a unique solution if and only if the same holds for equation \eqref{eq:intersection2-simple}, or in other words, that $\ell$ is tangent to $C$ if and only if $\ell'$ is tangent to $AC$, proving the claim.
\end{proof}

By a similar computation (or taking the limit $\lambda \nearrow b^2$) we get a dual caustic for the segment between the foci of $\cE$.

\begin{proposition} \label{cor:dual-caustic-to-a-segment}
Let $\cE \subset {\mathbb R}^2_q$ be an ellipse  which is a Euclidean  string construction over the segment $C:= [(-x,0),(x,0)]$, for $x>0$, with string length $L+2x$. Then, the segment $C':= [(-\frac{x}{2L},0),(\frac{x}{2L},0)]$ is the
 dual convex $\cE$-caustic to $C$ in $B \subset \R^2_p$. Moreover, if $\ell \in \cL(\cE)$ passes through $(x,0)$, then 
  the  $({\mathcal E},B)$-dual line  $ \alpha_{{\mathcal E},B}(\ell) \in \cL(B)$ passes through $(\frac{x}{2L},0)$.

   %if $q \in \partial \cE$ and $\ell \in \cL(\cE)$ is the positive tangent from $q$ to $C$, then  the  $({\mathcal E},B)$-dual line  $ \alpha_{{\mathcal E},B}(\ell) \in \cL(B)$ is a negative tangent  to $C'$.
\end{proposition}

%\begin{proof}[{\bf Proof of Corollary~\ref{cor:dual-caustic-to-a-segment}}]
%Without loss of generality $x= \lambda e_1$ and then $\cE = A^{-1}B$ where
%$$A = \begin{pmatrix}
%1/a & \\ & 1/b
%\end{pmatrix},$$
%with $a = L/2$ and $b= \sqrt{L(L/4-\lambda)}$ (the value of $b$ is of no importance in what follows). Therefore $AC = A[-x,x] = [- \frac{x}{2L},\frac{x}{2L}]$ as claimed.
%\end{proof}

Using the rotational symmetry of $B$, we obtain the same result for a rotated ellipse. 

 % (TODO elaborate?)
%
%
%TO REWRITE:
%We shall actually need a slightly stronger claim, which determines the $({\mathcal E},B)$-duality transform applied to portions of (oriented)
%lines in the invariant circle corresponding to the degenerate line caustic in ${\mathcal E}$. Going over the proof of Proposition \ref{prop:ellipse}, one sees that in fact a stronger statement was proved, namely:

\begin{cor} \label{lem:alpha-ellipse-pointwise}\label{cor:dual-caustic-to-a-segment2}
Let $\cE \subset {\mathbb R}^2_q$ be an  ellipse which is a Euclidean string construction over the segment $C:= [e,b]$ with string length $L+|e-b|$, for two distinct points $b, e \in \R^2_q$. Then $C':= [\frac{e-b}{2L},\frac{b-e}{2L}]$ is its
dual convex $\cE$-caustic in $B \subset \R^2_p$. 
Moreover, if $\ell \in \cL(\cE)$ passes through $e$ ($b$), then 
  the  $({\mathcal E},B)$-dual line  $ \alpha_{{\mathcal E},B}(\ell) \in \cL(B)$ passes through $\frac{e-b}{2L}$ ($\frac{b-e}{2L}$).

%Moreover, if $q \in \partial \cE$ and $\ell$ is the positive tangent from $q$ to $C$, then the $({\mathcal E},B)$-dual line  $ \alpha_{{\mathcal E},B}(\ell)$ is a negative tangent  to $C'$.
\end{cor}

This result will be crucial in the proof of Theorem \ref{thm:main} below, and sheds some light on the formula for the dual caustic given in Remark \ref{rmk:dual-caustic-formula}.

\subsection{The case of smooth caustics}\label{sect:smooth-caustics}

Here we prove a special case of Theorem \ref{thm:main} when the caustic $C$ is sufficiently regular.

Let $K$ and $C$ be as in Theorem \ref{thm:main}, and assume moreover that $C$ is 
strictly convex. For $q \in \partial K$, denote by $e(q)$ and $b(q)$ the positive and the negative tangency points to $C$ from $q$ (see Figure~\ref{fig-dual-caustics}), and set $L(q)  =|q-e(q)| + |q-b(q)|$. 
%where $L$ is the string length of $K$, considered as a string construction over $C$ (see Lemma~\ref{lem:Minkowski-string-construction}), and $\ell_{e,b}$ is the (Euclidean) length of the counterclockwise arc of the boundary of $K$ from $e(q)$ to $b(q)$.

\begin{proposition}\label{prop:smooth-proof}
	Let $K \subset {\mathbb R}^2_q$ be a $C^1$-smooth, centrally symmetric and strictly convex body. If $C\subset K$ is a Euclidean caustic which is $C^2$-smooth and has nowhere vanishing curvature, then it admits a dual $K$-caustic $C' \subset B$. Moreover, $C'$ is $C^1$-smooth, strictly convex, and its boundary can be parametrized by
	$$
	w(q) = \frac{e(q)-b(q)}{L(q)}, \quad q\in \partial K.
	$$
%	Here $e(q)$ and $b(q)$ are the positive and negative, respectively, tangency points to $C$ from the point $q\in \partial K$, and $L(q)  = L - \ell_{e,b}$ where $L$ is the string length of $K$ as a string construction over $C$ and $\ell_{e,b}$ is the (Euclidean) length of the counterclockwise arc of the boundary of $K$ from $e(q)$ to $b(q)$.
	\end{proposition}

	\begin{figure} %[h1]
\begin{center}
\begin{tikzpicture}[scale=0.7]

\node[inner sep=0pt] (tangents) at (-1,0)
    {\includegraphics[width=.3\textwidth]{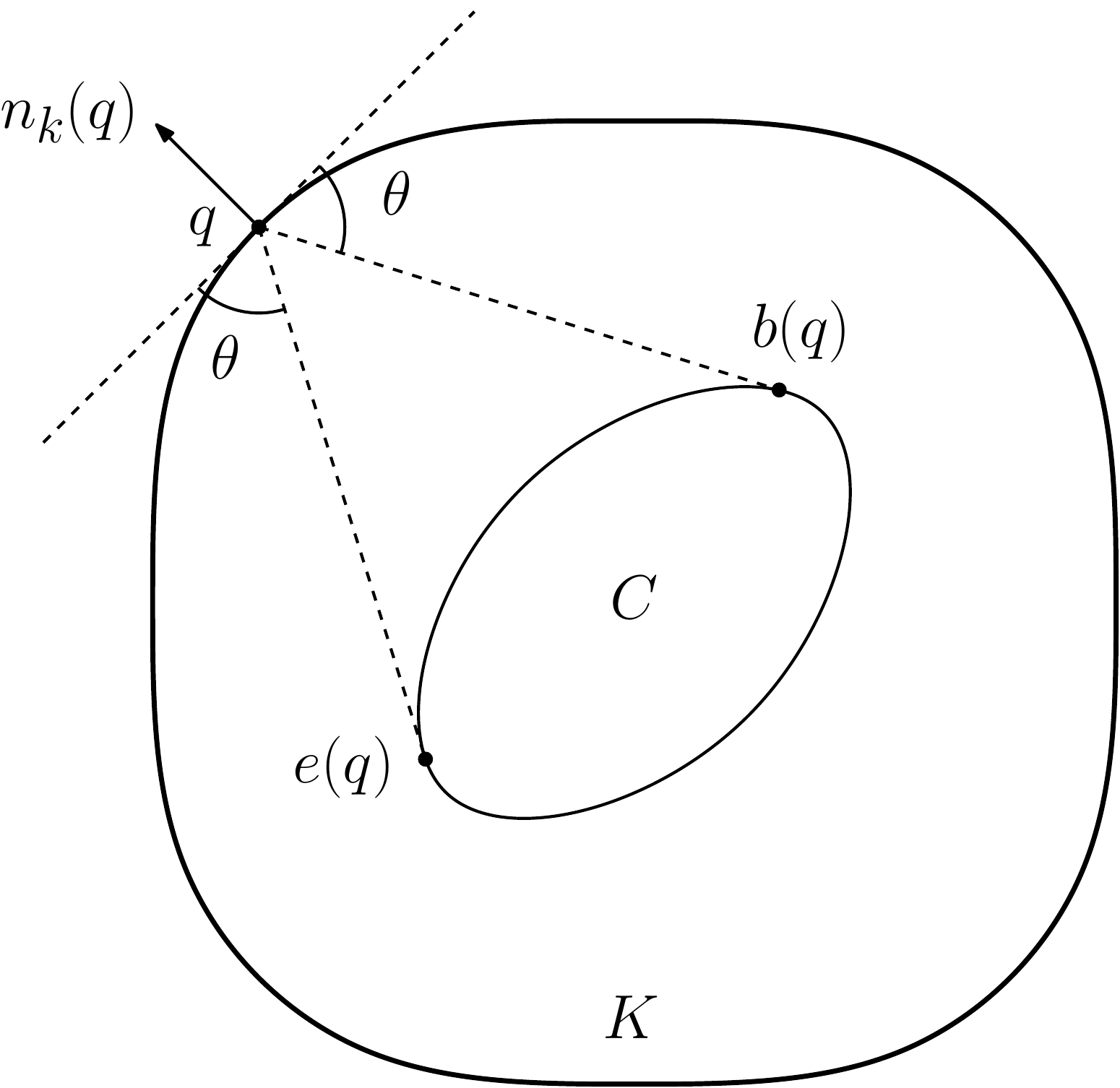}};

       \begin{scope}[xshift=9cm]
       \node[inner sep=0pt] (tangents-1) at (0,0)
 {\includegraphics[width=.27\textwidth]{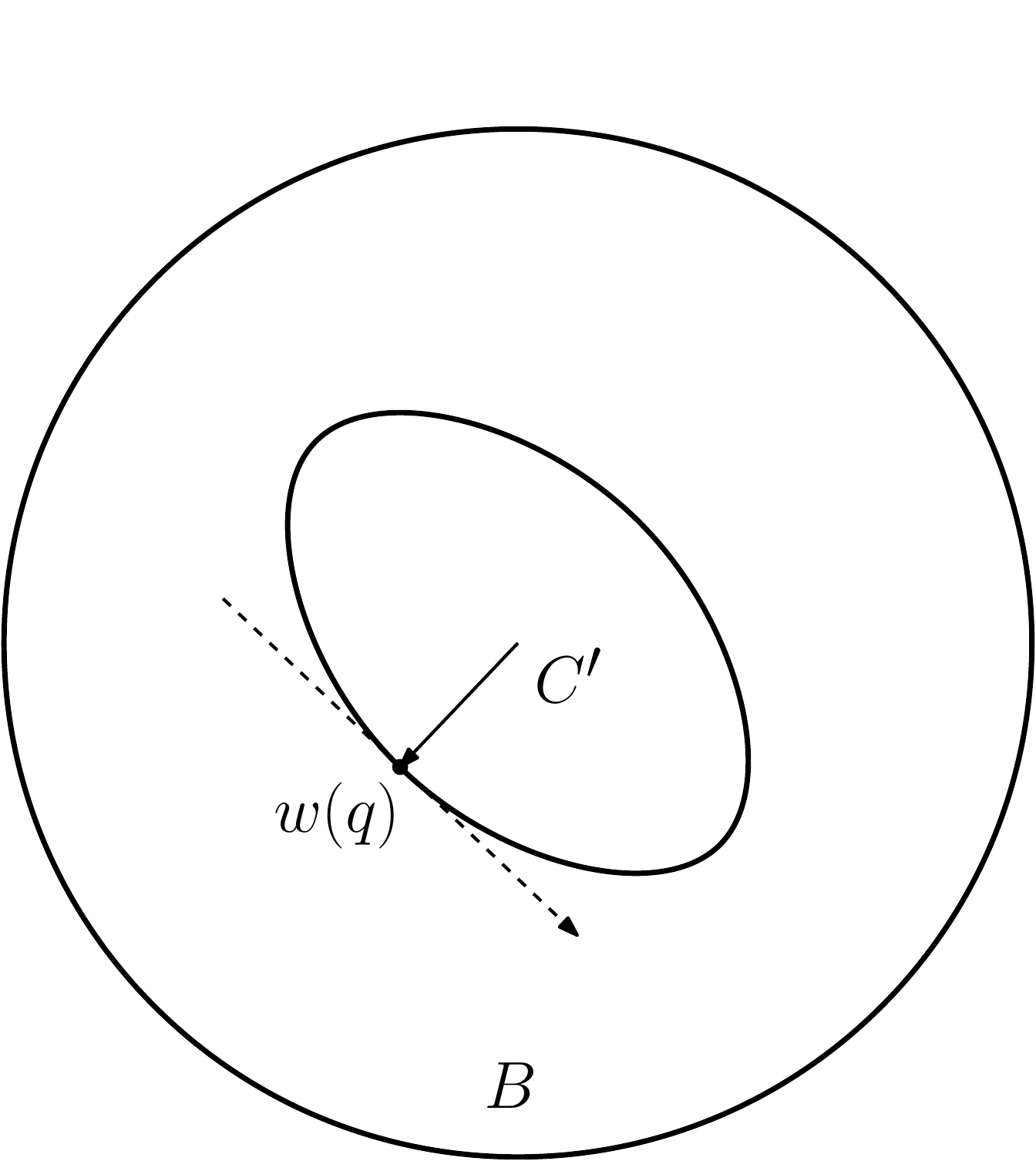}};
 \end{scope}
 \end{tikzpicture}

 \caption{A pair of dual caustics $C \in {\mathfrak{C}}(K,B)$ and  $C' \in {\mathfrak{C}}(B,K)$.} %On the left the tangent vectors $u$ and $v$. On the right the coordinate system $(r,t)$.} 
 \label{fig-dual-caustics} %\label{fig-Minkowksi-billiard}
 \end{center}
 \end{figure}

For the proof of Proposition \ref{prop:smooth-proof} we will require some straightforward computations. In what follows we denote by $\tau_C $ the unit counter-clockwise tangent to $C$ and by $k_C$ the curvature of $C$.
Moreover, we denote by $\theta$ the angle of incidence and reflection at $q$ (see Figure~\ref{fig-dual-caustics}), namely,
		$$
		\cos \theta = \iprod{\tau_K(q)}{\frac{e-q}{|e-q|}}=\iprod{\tau_K(q)}{\frac{q-b}{|q-b|}}.
		$$

	\begin{lemma} \label{lem:computation-of-der}
		With the assumptions of Proposition \ref{prop:smooth-proof}, the functions $e(q)$, $b(q)$, and $L(q)$ are $C^1$-smooth, with derivatives (with respect to the counter-clockwise arc-length parameter along the boundary $\partial K$)  
		%for $q$ traveling with arc-length parametrization counter-clockwise on $\partial K$):
		\begin{align*}
		e'(q) &= \frac{\sin \theta }{k_C(e) |e-q|} \tau_C(e), \quad b'(q) = \frac{\sin \theta }{k_C(b) |b-q|} \tau_C(b), \\
		L'(q) &= \sin\theta \cdot \left[ \frac{1}{k_C(e) |q-e|} - \frac{1}{k_C(b)|q-b|}\right].
		\end{align*}
		
		\end{lemma}
		
		\begin{proof}[{\bf Proof of Lemma~\ref{lem:computation-of-der}}]
		%	We start with $e(q)$ and $b(q)$. 
			Note that the function $e(q)$ is determined implicitly by $q$ via the following equation
			$$
			\iprod{q-e(q)}{n_C(e(q))}=0.
			$$
			In fact, this equation has two solutions, one of which is $e(q)$ and the other $b(q)$.
			Let $\gamma : [0,A] \to \partial C $ be a counter-clockwise parametrization by (Euclidean) arc length. We rewrite the above equation as
			$$
			F(q,t) = \iprod{q-\gamma(t)}{n_C(\gamma(t))}=0.
			$$
			%where $n_\gamma(t)$ is the (Frenet) normal vector to $\gamma$.
			Note that $F:\partial K \times [0,A]\to \R$ is $C^1$-smooth. To see that $e$ depends smoothly on $q$, we must check that $\frac{\partial F}{\partial t} \neq 0$. Indeed,
			\begin{align*}
			\frac{\partial F}{\partial t} (q,t)= \underset{=0}{\underbrace{-\iprod{ {\gamma}'(t)}{n_C(\gamma(t))}}} +  \iprod{q-\gamma(t)}{k_C(\gamma(t))\dot{\gamma}(t)}=%\underset{=0}{\underbrace{-\iprod{ {\gamma}'(t)}{n_C(\gamma(t))}}} + \iprod{q-\gamma(t)}{\frac{d}{dt}({n}_C(\gamma{(t)}))} =
			 k_C(\gamma(t)) \cdot \iprod{q-\gamma(t)}{\dot{\gamma}(t)}.
			\end{align*}
			Note that by the definition of $e$, the vector $q-\gamma(t)$ is negatively tangent to $C$ at $\gamma(t)$,  so that $q-\gamma(t)=-|q-\gamma(t)| \, \dot{\gamma}(t)$. Using the assumption $k_C \neq 0$, we conclude that for this solution $(q,t)$ of the equation $F(q,t) = 0$ one has
			$$
			\frac{\partial F}{\partial t} (q,t) =  - |q-\gamma(t)|\cdot k_C(\gamma(t)) \cdot \iprod{\dot{\gamma}(t)}{\dot{\gamma}(t)}=  -|q-\gamma(t)|\cdot k_C(\gamma(t))\neq 0.
			$$
			This shows that the map $q \mapsto t(q)$ for the above chosen solution of $F(q,t(q)) = 0$ is $C^1$-smooth.			
			We next compute its derivative with respect to $q$. Clearly,
%			When $x \in \R^2$ is not restricted to $\partial K$
%			$$
%			\frac{\del F}{\del x} = n_C(\gamma(t)).
%			$$
%Considering the restriction $x\in \partial K$ we get
			$$
			\frac{\del F}{\del q} (q,t) = \iprod{\tau_K(q)}{n_C(\gamma(t))} = -\iprod{n_K(q)}{\tau_C(\gamma(t))}=-\cos (\theta + \pi/2) =  \sin\theta.
			$$
			 Then, the Implicit Function Theorem gives
			$$
			\frac{dt}{dq} = -\frac{\del F / \del q}{\del F / \del t}=\frac{\sin \theta}{k_C(\gamma(t)) |\gamma(t)-q|}.
			$$
			Letting $e(q) = \gamma(t(q))$, we conclude that $e$ depends $C^1$-smoothly on $q$, and  %from the chain rule that
				$$			
				%\frac{de}{dq}	=		\frac{de}{dt}			\frac{dt}{dq} = \frac{\sin \alpha(q)}{k_C(\gamma(t)) |\gamma(t)-q|}\tau_C(\gamma(t)) =  \frac{\sin \alpha(q)}{k_C(e(q)) |e(q)-q|}\tau_C(e(q))
			e'(q) = \gamma'(t(q)) \cdot 	\frac{dt}{dq} =  \frac{\sin \theta}{k_C(e(q)) |e(q)-q|}\tau_C(e(q)).
			$$
			In computing $b'(q)$,   the only change is that the solution $t$ of $F(q,t)=0$ is now chosen so that $q-\gamma(t)$ is positively tangent to $C$ at $\gamma(t)$, which yields identical computations except for the following signs:
			$$
			q-\gamma(t)=|q-\gamma(t) \, |\dot{\gamma}(t) \,\text{ and }\, \iprod{n_K(q)}{\tau_C(\gamma(t))} = \cos\left({\pi}/{2}-\theta\right)=\sin\theta.
			$$
		Finally, since $e(q)$ and $b(q)$ are $C^1$-smooth, so is $L(q)$. Elementary differentiation gives
			\begin{align*}
			L'(q)  = \frac{ \iprod{e(q)-q}{e'(q)} }{|e(q)-q|} +\frac{\iprod{b(q)-q}{b'(q)} }{|b(q)-q|}= \iprod{e'(q)}{\tau_C(e(q))} - \iprod{b'(q)}{\tau_C(b(q))}.
			\end{align*}
			Substituting the formulas for $e'(q)$ and $b'(q)$ gives the required result.
%			Finally, recall that $L(q) = L - \ell_{e,b}$. Therefore, $L'(q) = -\frac{d}{dq}\ell_{e,b} = \frac{d}{dq}\ell_{b,e}$, where $\ell_{b,e}$ is the counter-clockwise arc of $\partial C$ from $b$ to $e$. Then clearly $L(q)$ is $C^1$-smooth, and
%			\begin{align*}
%			L'(q) & =\frac{d}{dq} \bigl(\gamma^{-1}(e(q))-\gamma^{-1}(b(q)) \bigr) = \iprod{e'(q)}{\tau_C(e(q))} - \iprod{b'(q)}{\tau_C(b(q))}.
%			\end{align*}
%			Substituting the formulas for $e'(q)$ and $b'(q)$ gives the result as required.
			\end{proof}

					\begin{proof}[{\bf Proof of Proposition \ref{prop:smooth-proof}}]
						By the previous lemma, the map $$\partial K \ni q \mapsto w(q) = \frac{e(q)-b(q)}{L(q)}$$ is a closed, $C^1$-smooth parametrized curve  which is clearly contained in $B$. Our goal is to prove that it is a simple closed curve bounding a strictly convex set, which we denote by $C'$, and that $C' \subset B$ is a $K$-caustic which is dual to $C$. To this end, we first prove that $w'(q)$ is negatively proportional to $n_K(q)$ for all $q \in \partial K$. We first fix some notations. We abbreviate $e=e(q)$, $b=b(q)$, and denote
						$$
						B = |b-q|, \,v = \frac{b-q}{B}, \quad E = |e-q|, \,u = \frac{e-q}{E}.
						$$
						Next, by the definition of $e$ and $b$, one has
					that 
$						\tau_C(e) = u$ and $\tau_C(b) = -v$.
											Note that by the (Euclidean) billiard reflection law,
						$
						n_K(q)$ is negatively proportional to the vector $u+v$ (see Figure~\ref{fig-dual-caustics}).
						%						In the above notations,
%						$$
%						L(q) = E + B; \quad e-b =  Eu-Bv.
%						$$
%						We rewrite the previous formulas in these terms,
%						\begin{align*}
%						w(q) &= \frac{Eu-Bv}{E+B},\\
%						e'(q) &= \frac{\sin \alpha }{E\cdot k_C(e)} u,\quad b'(q) = -\frac{\sin \alpha }{B \cdot k_C(b)}v ,\\
%						 L'(q) &= \sin\alpha \left[ \frac{1}{E\cdot k_C(e)} - \frac{1}{B \cdot k_C(b)}\right].
%						\end{align*}
%						First we note that from the first formula it is evident that the image of $w$ is indeed contained inside the unit disc. 
						A direct computation using Lemma~\ref{lem:computation-of-der} gives
						\begin{align*}
						w'(q) 
%						&= \frac{e'(q)-b'(q)}{L(q)} - \frac{e(q)-b(q)}{L(q)^2} L'(q) \\
%						& = \frac{\sin \alpha }{L(q)} \left[  \frac{u}{E\cdot k_C(e)}  + \frac{v}{B \cdot k_C(b)} - \frac{Eu-Bv}{E+B} \left( \frac{1}{E\cdot k_C(e)} - \frac{1}{B \cdot k_C(b)}\right) \right] \\
%						&= \frac{\sin \alpha }{L(q)} \Biggl[ \left( \frac{1}{E\cdot k_C(e)}  - \frac{E}{E+B} \left( \frac{1}{E\cdot k_C(e)} - \frac{1}{B \cdot k_C(b)}\right) \right) u    \\
%						&\qquad \quad +\left( \frac{1}{B\cdot k_C(b)} + \frac{B}{E+B} \left( \frac{1}{E\cdot k_C(e)} - \frac{1}{B \cdot k_C(b)}\right) \right)v 	
%						\Biggr] \\
						&=\frac{\sin \theta }{L(q)^2} \left[\frac{B}{E\cdot k_C(e)}+ \frac{E}{B\cdot k_C(b)}\right] (u+v).
						\end{align*}
						%Since, as noted above, we have $n_K(q) \sim -(u+v)$, 
						This proves that $w'(q)$ is proportional to $-n_K(q)$, as asserted. This implies that the angle of $w'(q)$ is strictly increasing, and increases by $2\pi$ as $q$ traverses $\partial K$ (counter-clockwise) once. Finally, we use  elementary differential geometry to conclude that $w(q)$ is a simple closed curve bounding a strictly convex domain $C'$. Indeed, simplicity follows since otherwise $w$ would self-intersect at some point $p$, in which case either the intersection is transversal, which would imply that points of $w$ lie on both sides of one of the tangents to $w$ at $p$, or the intersection is tangential, in which case $w$ would have a double tangent at $p$. In both cases, this implies the existence of three parallel tangents to $w$, a contradiction.  For the convexity, note that once the curve is simple, if it were not convex, this would mean that we may construct a line which intersects $w$ at least four times, partitioning it into four arcs, and such that $w$ admits a tangent parallel to this line in each of the four regions, again contradicting the fact that the angle increases by $2\pi$ during the journey along $w$. 
						
						Next let us prove that $C'$ is the dual caustic to $C$. Indeed, fix $q \in \partial K$ and let $\cE$ be the ellipse formed by the Euclidean string construction over the segment $[e(q),b(q)]$ with string length $L(q)+|e(q)-b(q)|$. Note that $\cE$ contains $q$ on its boundary, and moreover $n_\cE(q) = n_K(q)$ (since the broken line $\overline{bqe)}$ obeys the Euclidean billiard reflection law in both). In particular, the duality transforms $\alpha_{K,B}$ and $\alpha_{\cE,B}$ agree on oriented lines emanating from $q$. By Corollary \ref{cor:dual-caustic-to-a-segment2}, the dual  of the oriented line $\ell$ emanating from $q$ and tangent to $C$ at $e(q)$, $\alpha_{\cE,B}(\ell)$, passes through $w(q)$, and has direction $n_\cE(q)=n_K(q)$. Since  $w'(q)$ is negatively proportional to $n_K(q)$, it follows that the line $\alpha_{K,B}(\ell)=\alpha_{\cE,B}(\ell)$ is tangent to $C'$. In other words, denoting by $\Gamma$ the invariant circle determined by the caustic $C \subset K$, we have shown that $C'$ is tangent to every line in  $\alpha_{K,B}(\Gamma)$. By Corollary \ref{cor:duality-caustic-via-alpha}, this means that $C'$ is the dual convex caustic to $C$. This completes the proof.
						\end{proof}

\subsection{Proof of Theorem \ref{thm:main}}\label{sect:general-caustics}

Here we complete the proof of Theorem~\ref{thm:main}, by using an approximation argument to reduce it to Proposition \ref{prop:smooth-proof}. To this end we shall need the following lemma.

\begin{lemma} \label{lem:approximation-without-monotonicity-lemma}
Suppose that $K_n \subset {\mathbb R}^2_q$ and $T_n  \subset {\mathbb R}^2_p$ are a sequence of symmetric billiard configurations, and let $C_n$ be a sequence
of convex $T_n$-caustics in $K_n$.  Assume further that the following Hausdorff limits exist:
\begin{equation*}
K_n \to K, \ T_n \to T, \ C_n \to C,
\end{equation*}
where $(K,T)$ is a symmetric billiard configuration. Then, $C$ is a (convex) $T$-caustics in $K$. Moreover, if in addition for each $n \in {\mathbb N}$ there exists a  convex $K_n$-caustic $C'_n $ dual to $C_n$, and $C'_n \to C'$ in the Hausdorff metric, then $C$ and $C'$ are dual caustics.
%
%
%, for each $n$, possess a dual pair of caustics $C_n \subset K_n$ and $D_n \subset T_n$. Assume further that there exist the following Hausdorff limits:
%\begin{equation*}
%K_n \to K, \, T_n \to T, \, C_n \to C, \, D_n \to D.
%\end{equation*}
%Then $C$ is a a $T$-caustic in $K$, $D$ is a $K$-caustic in $T$, and the two are dual.
\end{lemma}

\begin{proof}[{\bf Proof of Lemma~\ref{lem:approximation-without-monotonicity-lemma}}]
The fact that $C \subseteq K$ follows from the properties of Hausdorff convergence. To show that $C$ is a $T$-caustic in $K$, note that
any positive tangent line $\ell$ to $C$ is a limit of positive tangent lines $\ell_n$ to $C_n$.
Indeed, let $x \in \partial K$. From the fact that $K_n \to K$ it follows that there is a sequence $x_n \in \partial K_n$ such that $x_n \to x$.
For each $x_n$ let $\ell_n$ be the positive tangent to $C_n$ emanating from $x_n$.
By passing to a subsequence, we may assume without loss of generality that $\ell_n$ converges to some oriented line $\widetilde \ell$ emanating from the point $x$.
It is easily seen that the line $\widetilde \ell$ is tangent to $C$. This follows from the fact that the tangency of $\ell_n$ to $C_n$ can be described by the equality $\max_{C_n} \phi_n = \phi_n(x_n)$, where $\phi_n$ is a linear functional (say, of norm one)  which admits $\ell_n$ as a level set. Since the line emanating from $x$ and positively tangent to $C$ is unique, one has $\widetilde \ell = \ell$ i.e., $\ell_n \to \ell$.
To show that $C$ is indeed a $T$-caustic, consider for each $n$ the next reflection point $x_n' \in \partial K_n$ of the line $\ell_n$, and the line $\ell_n'$ emanating from $x_n'$ and positively tangent to $C_n$.  In other words,  $\Psi^2_{K_n,T_n}(\ell_n) = \ell_n'$. % (note that the map $\Psi^2$ is a different mapping for each $n$).
Repeating the argument above, we obtain that $\ell_n' \to \ell'$ which is positively tangent to $C$.
From Corollary \ref{cor:continuity-of-Psi} it follows that $\Psi^2_{K,T}(\ell)=  \ell'$, which shows that $C$ is indeed a $T$-caustic in $K$.
For the second part of the lemma, let $C'_n$ be a convex $K_n$-caustic in $T_n$ such that $C'_n \to C'$. As before, $C'$ is a $K$-caustic in $T$. Moreover, since $\Psi_{K_n,T_n}(\ell_n)$ is tangent to $C'_n$, passing to the limit one sees that $\Psi_{K,T}(\ell)$ is tangent to $C'$, and so by Corollary \ref{cor:dual-caustic-psi}, $C$ and $C'$ are dual caustics.
\end{proof}

\begin{proof}[{\bf Proof of Theorem \ref{thm:main}}]
Let $K$ be a centrally symmetric, smooth, and strictly convex body, and let $C$ be a Euclidean caustic in $K$. 
By Lemma~\ref{lem:Minkowski-string-construction} the body $K$ is a (Euclidean) string construction over $C$, and we denote its string length by $L$. 
%of $K$ as a Euclidean string construction over $C$.  
Let $C_n$, $n \in \N $, be a sequence of $C^2$-smooth convex bodies with nowhere vanishing curvature which converges to $C$ in the Hausdorff metric. For each $n$, let $K_n$ be the convex body formed by a Euclidean string construction over $C_n$ with string length $L$. Then, denoting $f_n(x) = \Per (\conv(x,C_n))$ and $f(x) = \Per (\conv(x,C))$, one has $K_n = \{f_n \leq L\}$ and $K = \{f \leq L\}$. Since $C_n \to C$, one easily verifies that $(f_n)$ converges to $f$ uniformly, which implies that $K_n$ converge to $K$ in the Hausdorff metric. By Proposition \ref{prop:smooth-proof}, each $C_n$ admits a dual $K_n$-caustic $C'_n \subset B$. By passing to a subsequence, we may assume that $C'_n$ converge in the Hausdorff metric to a convex body $C'$. Now by Lemma \ref{lem:approximation-without-monotonicity-lemma}, $C'$ is a $K$-caustic in $B$, dual to $C$. By Proposition \ref{prop:dual-invariants}, the caustics $C$ and $C'$ have the same rotation numbers, and that the dual caustics $C_n$ and $C'_n$ have equal  perimeters and string lengths. 
Finally, by continuity, $C$ and $C'$ have equal perimeters and equal string lengths. This completes the proof.
\end{proof}

%Then as $n\to \infty$, $\ell_n(x) \to \ell(x)$ which is the positive tangent to $C$ emanating from $x$ (indeed to self: consider the functional $J_n$ which equals $1$ on $\ell_n(x)$ - including orientation - and is $\leq 1$ on $C_n$. Then since weak inequality is preserved in the limit, $J = \lim_n J_n$ is $1$ on $\ell(x)$ and $\leq 1$ on $C$). This shows that the circle determined by $C$ is the limit of those determined by $C_n$, and hence by continuity of the billiard map this circle is invariant. Hence, $C$ is a caustic. Similarly, $D \subset T$ is a $K$-caustic. Finally, by continuity of $\Psi$, the invariant circles corresponding to $C$ and $D$ are dual.
%\end{proof}

\section{Non-existence of a dual caustic} \label{sec:non-existence-of-dual-caustic}

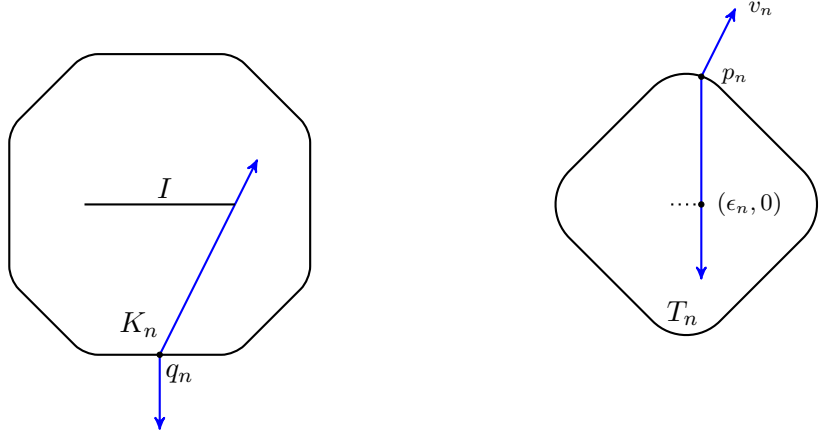
\begin{figure} %[h1]
\begin{center}
\begin{tikzpicture}[scale=1]

% \draw[important line][rounded corners=5pt][rotate=22.5] (2,0) -- (1.414,1.414) --
%  (0,2) -- (-1.414,1.414)--  (-2,0)--  (-1.414,-1.414) -- (0,-2) -- (1.414,-1.414) --
% cycle;

 \draw[important line][rounded corners=5pt] (2,1) -- (1,2) --
  (-1,2) -- (-2,1)--  (-2,-1)--  (-1,-2) -- (1,-2) -- (2,-1) --
 cycle;

 \path coordinate (p1) at (-0.85,-2*0.25+0.24) coordinate (np0) at
 (2.3,2*0.9) coordinate (p0) at (1.2-0.55,2*0.53-0.45) coordinate (np1) at
 (0.7,-3) coordinate (p2) at (-1.2,2*0.66+0.38)  coordinate (D) at
 (0.3,0.22) coordinate (E) at (0.3,-1.6);

 \draw[<-][blue][important line] (1+0.3,0+0.6)  -- (0,-2);
  \draw[<-][blue][important line] (0,-3)  -- (0,-2);
    \draw[important line] (-1,0)  -- (1,0);
 %node[right] {{\footnotesize $\nabla \|p_0  \|_{ T}$}}

%\draw[blue][important line] (p2) -- (p1);
% \draw[blue][important  line] (p0) -- (p2);

  \filldraw [black]
    %   (p1) circle (1pt) node[left] {{\footnotesize $e(q)$}}
%         (p2) circle (1pt) node[above left] {{\footnotesize $q$}}
%         (p0) circle (1pt) node[above=2pt] {{\footnotesize $b(q)$}}
  (0,-2) circle (1pt) node[below] {{ }}
  (0.2,-2) circle (0pt) node[below] {{ $q_n$}}
             (E) circle (0pt) node[left=5pt] {${ K_n}$}
          (D) circle (0pt) node[left] {${ I}$};

\pgfmathsetmacro{\a}{0.3*3}
      \pgfmathsetmacro{\b}{0}
      \pgfmathsetmacro{\c}{sqrt(\a^2 - \b^2)}
      \pgfmathsetmacro{\aa}{2.1}
      \pgfmathsetmacro{\bb}{1.1}
      \pgfmathsetmacro{\cc}{sqrt(\aa^2 - \bb^2)}

       %     \draw[important line][rotate=0] (0,0) ellipse [x radius=\a, y radius=\b];

       \begin{scope}[xshift=7cm]

 \draw[important line][rounded corners=18pt,  tension=2][rotate=0] (2,0) --  (0,2) --
 (-2,0) --  (0,-2)   -- cycle;

 %\draw[important line] (0.2,0) --  (0,0.2) -- (0.2,0);

%\pgfmathsetmacro{\aaa}{0.2}
%      \pgfmathsetmacro{\bbb}{0}
%
%   \draw[important line][rotate=0] (0,0) ellipse [x radius=\aaa, y radius=\bbb];

 \draw[<-][blue][important line] (1+0.3+0.2-0.65*1.3,0+0.6+3.7-1.3*1.3)  -- (0+0.2,-2+3.7);
  \draw[<-][blue][important line] (0+0.2,-3+2)  -- (0+0.2,-2+3.7);
    \draw[important line, dotted] (-0.2,0)  -- (0.2,0);

\path coordinate (D) at (0.3,0.22);
\path coordinate (E) at (0.3,-1.45);

  \filldraw [black]
  (1+0.3+0.2-0.65*1.3,0+0.6+3.7-1.3*1.3) circle (0pt) node[right=1pt] {{\footnotesize $v_n$}}
 (0+0.2,-2+3.7) circle (1pt) node[right=4pt] {{\footnotesize $p_n$}}
  (0+0.2,-2+2) circle (1pt) node[right=2pt] {{\footnotesize $(\epsilon_n,0)$}}
%       (p1) circle (2pt) node[below right] {{\footnotesize $p_0$}}
%         (p2) circle (2pt) node[left] {{\footnotesize $p_2$}}
%         (p0) circle (2pt) node[above=2pt] {{\footnotesize $p_1$}}
    (E) circle (0pt) node[left] {${ T_n}$}
       ;
 \end{scope}
 \end{tikzpicture}

 \caption{A convex caustic which has no dual convex caustic.} \label{fig-counterexample}
 \end{center}
 \end{figure}

In this section we prove Theorem~\ref{thm-counter} and show that the caustic-to-caustic duality established in Theorem~\ref{thm:main} fails for arbitrary Minkowski billiards.

\begin{proof}[{\bf Proof of Theorem~\ref{thm-counter}}]
We start with a construction of a ``degenerate" convex caustic which has no convex dual caustic.
Let $T \subset \R_p^2$ be the unit ball of the $\ell_1$-norm, i.e.,  $T = \{ (p_1,p_2) \, : \, |p_1| + |p_2| \leq 1 \}$.
Let $T_n$ be a family of smooth, unconditional, and strictly convex bodies in $\R_p^2$ which converges to $T$ in the Hausdorff metric. Assume further that for each $n$, the boundary of $T_n$ contains the point $(1,0)$, and that the (Euclidean) outer normal  $n_{T_n}$ to $T_n$ on the open first quadrant
converges locally uniformly to the vector $(1/\sqrt{2},1/\sqrt{2})$ (and hence by unconditionality $n_{T_n}\to (\pm 1/\sqrt{2}, \pm 1/\sqrt{2})$ locally uniformly on the other open quadrants). Fix the interval $I = [-1,1] \subset \R^2_q$, and a string length $L=6$, and let $K_n \subset \R^2_q$ be the result of the $h_{T_n}$-string construction on $I$ with string length $L$ (see Figure~\ref{fig-counterexample}). By Lemma~\ref{lem:Minkowski-string-construction}, $I=[-1,1]$ is a convex $T_n$-caustic in $K_n$ for all $n$. We claim that for  large enough $n$, the caustic $I$ admits no dual convex caustic in $T_n$.

Indeed, assume towards a contradiction that for infinitely many $n$ there exists a convex $K_n$-caustic $C_n$ in $T_n$, dual to $I$. We first note that the caustic $I \subset K_n$ admits a $2$-periodic tangent $T_n$-billiard trajectory along the $q_1$-axis, with bouncing points $(\pm 2 , 0)$. Since the vector $n_{K_n}(\pm 2, 0) = (\pm 1, 0)$, the dual caustic $C_n$ also admits a $2$-periodic tangent $K_n$-billiard trajectory along the $p_1$-axis, and hence $C_n$ must be ``flat", that is $C_n \subset \{p_2 = 0\}$. From symmetry, we conclude that $C_n = [-\epsilon_n, \epsilon_n] \times \{ 0 \} \subset \R^2_p$, for some $\epsilon_n>0$.
Next, let $q_n = (0, -y_n) \in \partial K_n$ with $y_n>0$. Note that, by symmetry, $n_{K_n}(q_n)$ is in direction $(0, -1)$. Let $\ell_n$ be the positive tangent from $q_n$ to $I$, whose direction we denote by $v_n$ (normalized so that $v_n \in \partial T_n$). Then, by Corollary~\ref{cor:duality-caustic-via-alpha}, the line $\alpha_{K_n,T_n}(\ell_n)$ is (negatively) tangent to $C_n$. By definition, $\alpha_{K_n,T_n}(\ell_n) = (p_n, (0,-1))$, where $p_n \in \partial T_n$ is the unique point such that $n_{T_n}(p_n)=v_n$ (see Figure \ref{fig-counterexample}). Therefore, the first coordinate of $p_n$ is $\epsilon_n$. Note that the sequence $v_n$ converges to some $v \notin \{(t, \pm t) \, : \, t \in {\mathbb R} \}$. % (in fact $v$ is in fact  direction of the vector $(1,2)$
Indeed, $K_n$ converges to $K$, which is the body obtained by the $h_T$-string construction on the interval $I$, and so $v_n \to v$, which is the vector in the direction of the positive tangent to $I$ from $q=(0, -2) \in \partial K$, i.e., $v$ is in direction $(1,2)$. %since $n_{T_n}(p_n) = v_n \to v \notin \{(\pm 1, \pm 1)\}$, 
Since $v_n=n_{T_n}(p_n) \to v$, by the construction of $T_n$ one has that $p_n = n_{T_n^{\circ}}(v_n) \to n_{T^{\circ}}(v) = (0,1)$. This shows that $\epsilon_n \to 0$.

Using Lemma~\ref{lem:Minkowski-string-construction}, as $C_n=[-\epsilon_n, \epsilon_n] \times \{ 0 \}$ is a $K_n$-caustic in $T_n$, 
it is an $h_{K_n}$-string construction on $[-\epsilon_n, \epsilon_n]\times\{0\}$ of some string length $L_n$. In fact, since $\partial T_n$ contains the point $(1,0)$, one has that $L_n = 4(1+ \epsilon_n)$ (as $h_{K_n}(1,0)=2$).
%({\bf we don't know that $L_n=L$ right now, since they are not smooth}).
We claim that from this it follows that the Hausdorff distance $d_{H}(T_n, (L_n/2) \cdot K_n^\circ) \to 0$ as $n\to \infty$. Indeed, by item $(i)$ in Lemma~\ref{monotonicity-of-caustics-lemma}, it follows that $T_n$ is contained in
 %the result of the $h_{K_n}$-string construction of string length $L_n$ over the point $\{(0,0)\}$, which is simply 
 $(L_n/2)\cdot K_n^\circ.$ On the other hand, let us verify that $T_n$ contains the result of an $h_{K_n}$-string construction of string length $L_n-8\epsilon_n$ (which is positive as $\epsilon_n < 1$) on the point $\{(0,0)\}$, which is $(L_n/2-4\epsilon_n) \cdot K_n^\circ$.
Indeed, let $p$ belong to $(L_n/2-4\epsilon_n) \cdot K_n^\circ$, so that $h_{K_n}(p) \leq L_n/2-4\epsilon_n$. By the triangle inequality, and as $h_{K_n} (\epsilon_n,0) = 2 \epsilon_n$, one has $$h_{K_n}(p-(\epsilon_n,0)) \leq h_{K_n}(p) + 2\epsilon_n \leq L_n/2- 2 \epsilon_n.$$  Similarly, $h_{K_n}(p+(\epsilon_n,0)) \leq L_n/2- 2 \epsilon_n$.
It follows that the $h_{K_n}$-perimeter of ${\rm Conv}(p,C_n)$ is at most $L_n$, so $p \in T_n$. 
Since $L_n = 4(1+ \epsilon_n)$, we verified that
$$
\left(2 - 2 \epsilon_n\right) K_n^\circ \subset T_n \subset\left(2 + 2\epsilon_n \right)K_n^\circ.
$$
Therefore, the Hausdorff distance $d_{H}(T_n, 2 K_n^\circ) \to 0$ as $n\to \infty$. By construction, $T_n \to T$ as $n\to \infty$. Therefore, $K_n \to 2T^{\circ}$, which is a square of side length $4$.
%, where $K$ is the results of the $h_{T}$-string construction of length $L$ on $I$. Thus, $T_n \to 2K^{\circ}$.
On the other hand, by the continuity of the string construction, $K_n$ converge to the $h_{T}$-string construction of length $L=6$ on $I$, which is clearly not a square (it is an octagon, see Figure~\ref{fig-counterexample}).
We have thus reached a contradiction, meaning that for large enough $n$, the $T_n$-caustic $I$ in $K_n$ has no dual convex caustic.

To produce a counterexample  with a smooth caustic we use a standard approximation argument.
Let $(K,T)$ be a symmetric billiard configuration admitting an interval $I \subset K$ as a $T$-caustic with string length $L$ which has no dual convex $K$-caustic (as constructed above). 
%as above with $I$ as a $T$-caustic in $K$ (with string length $L$) with no dual convex $K$-caustic in $T$.  
Let $\cE_j$ be a sequence of ellipses such that $\cE_j \to I$, and let $K_j \subset \R_q^2$ be the result of the $h_{T}$-string construction on $\cE_j$ with string length $L$. We claim that for $j$ large enough, $T$ contains no convex $K_j$-caustic dual to $\cE_j$. Indeed, if a dual caustic exists for infinitely many $j$'s, then from Lemma~\ref{lem:approximation-without-monotonicity-lemma} it follows that in the limit (after passing to a subsequence if necessary) $I$ admits a convex $K$-dual caustic in $T$ which is a contradiction. %This completes the proof.
\end{proof}

\section{Appendix: dual caustics to polygonal caustics} \label{sec:polygonal-case}

To complete the picture we present a different proof for Theorem~\ref{thm:main}, using an approximation of a general $B$-caustic in $K$ by polygons. The approximation argument is similar to the one in the proof of Theorem~\ref{thm:main}, using Lemma~\ref{lem:approximation-without-monotonicity-lemma}. The main ingredient is the following theorem for polygonal caustics.

\begin{figure}[ht]
	\centering
	\includegraphics[scale=0.4]{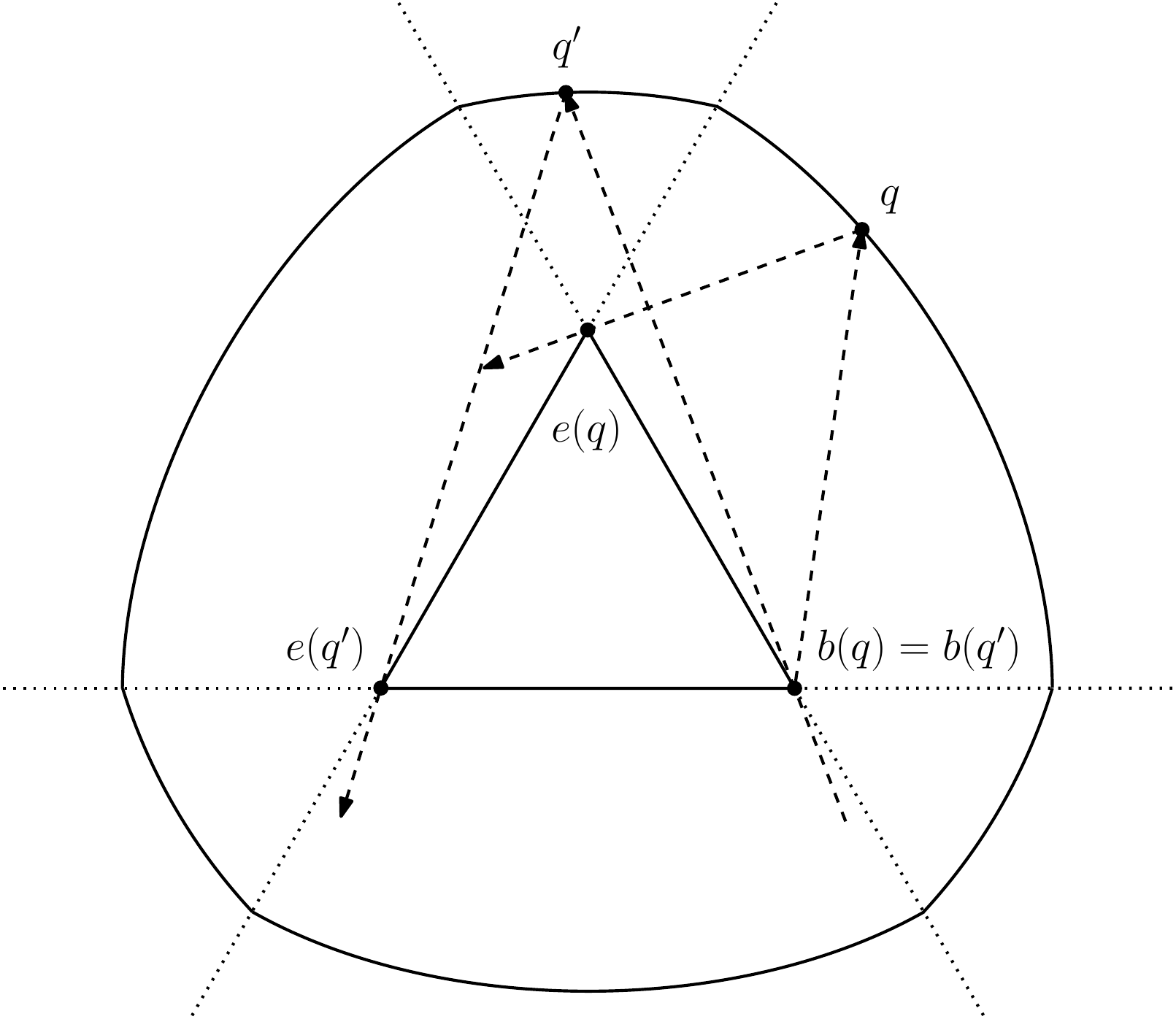}
	\caption{String construction over a triangle.}
	\label{fig:ellipsogon}
\end{figure}

\begin{theorem}\label{thm-poly-case}
Let $C$ be a convex polygon which is a caustic for the Euclidean billiard in a $C^1$-smooth, 
centrally symmetric strictly 
convex body $K \subset \R^2_q$. Then $C$ has a convex dual polygonal caustic $C' \subset B^2 \subset {\mathbb R}^2_p$.
\end{theorem}

\begin{proof}[{\bf Proof of Theorem \ref{thm-poly-case}. }]
First note that since $K$ is a Euclidean string construction over $C$, then by Lemma~\ref{lem:conveixty-of-Mink-string-construction} it follows that $K$ is necessarily $C^1$-smooth and strictly convex. Next, for $q \in \partial K$, we denote by $e(q)$ and $b(q)$ the positive and negative tangency points to $C$ from $q$, and let  $L(q)  = |q-e(q)|+|q-b(q)|$. These notation are as in the proof of Proposition \ref{prop:smooth-proof}, but the main difference is that in the polygonal case the functions $e(q)$ and $b(q)$ are no longer continuous, and are piecewise constant (and initially multi-valued at a finite number of points). In fact, the boundary $\partial K$ is a finite union of arcs, on (the interior of) each of which the functions $e(q)$ and $b(q)$ are constant, and those arcs are therefore arcs of ellipses with foci $e(q)$ and $b(q)$ and string length given by $L(q)+|e(q)-b(q)|$ (see Figure~\ref{fig:ellipsogon})\footnote{Note that a string construction over a non-centrally-symmetric polygon can never be centrally symmetric. However, the figure suffices to illustrate the parameters specified.}.

Let us explain this in more detail. Note that, as  $q$ traverses  $\partial K$ counter-clockwise,
the changes in $b(q)$ and $e(q)$ occur only when $q$ crosses  the lines passing through the edges of $C$, and at such an intersection either one or both of $e(q)$ and $b(q)$ change, by moving from one vertex of $C$ to a successive vertex (and for a generic string length they do not change simultaneously).  Indeed, for any point $z \in {\mathbb R}^n$ one may consider the edges of $C$ which are ``illuminated" by $z$ and those which are ``dark". Here an edge is called  ``illuminated by $z$" if the interval $[z,q] \cap C = \{q\}$ for any $q$ on that edge. It is not hard to check that for any fixed $z \in {\mathbb R}^2 \setminus C$, the edges of $C$ are separated into two disjoint sets (illuminated and dark). Moreover, with respect to the cyclic order on the edges, the illuminated edges form an order interval, as do the dark edges. For $q \in \partial K$, the point $b(q)$ is the initial vertex of the illuminated polygonal line, and $e(q)$ is its terminal vertex, and we use this to extend the definitions of $e(q)$ and $b(q)$ to those points where they were previously multi-valued. As $q$ crosses the line on which a certain edge lies, this edge changes from illuminated to dark, or vice versa. In the former case, $b(q)$ changes from the initial vertex of this edge to its terminal vertex, and in the latter case, $e(q)$ changes from the initial vertex of this edge to its terminal vertex (again, see Figure~\ref{fig:ellipsogon}). This shows that, as asserted, the boundary of $K$ is a finite union $\partial K = \bigcup_{i=1}^M \cE_i$, disjoint except for common boundary points, where for each $i$ the functions $e(q)$ and $b(q)$ are constant on the interior of $\cE_i$, and $\cE_i$ is therefore an arc of an ellipse with foci $b(q)$ and $e(q)$. We index $\cE_i$ according to the counter-clockwise order on $\partial K$, and denote by $q_i$ the joint boundary point of $\cE_{i-1}$ and $\cE_i$, and by $e_i$, $b_i$ and $L_i$ the constant values of $e(q)$, $b(q)$, and $L(q)$ on the interior of $\cE_i$. %Note that $e(q_i)$ and $b(q_i)$ are not uniquely defined, and in fact can be chosen to take any value on the edges $[e_{i-1}, e_i]$ and $[b_{i-1}, b_i]$, respectively, of $\partial C$.

%\begin{figure}[ht]
%	\centering
%	\includegraphics[scale=0.5]{ellipsogon.pdf}
%	\caption{String construction over a triangle.}
%	\label{fig:ellipsogon}
%\end{figure}

 \begin{figure} %[h1]
\begin{center}
\begin{tikzpicture}[scale=0.7]

\node[inner sep=0pt] (tangents) at (-2,0)
    {\includegraphics[width=.3\textwidth]{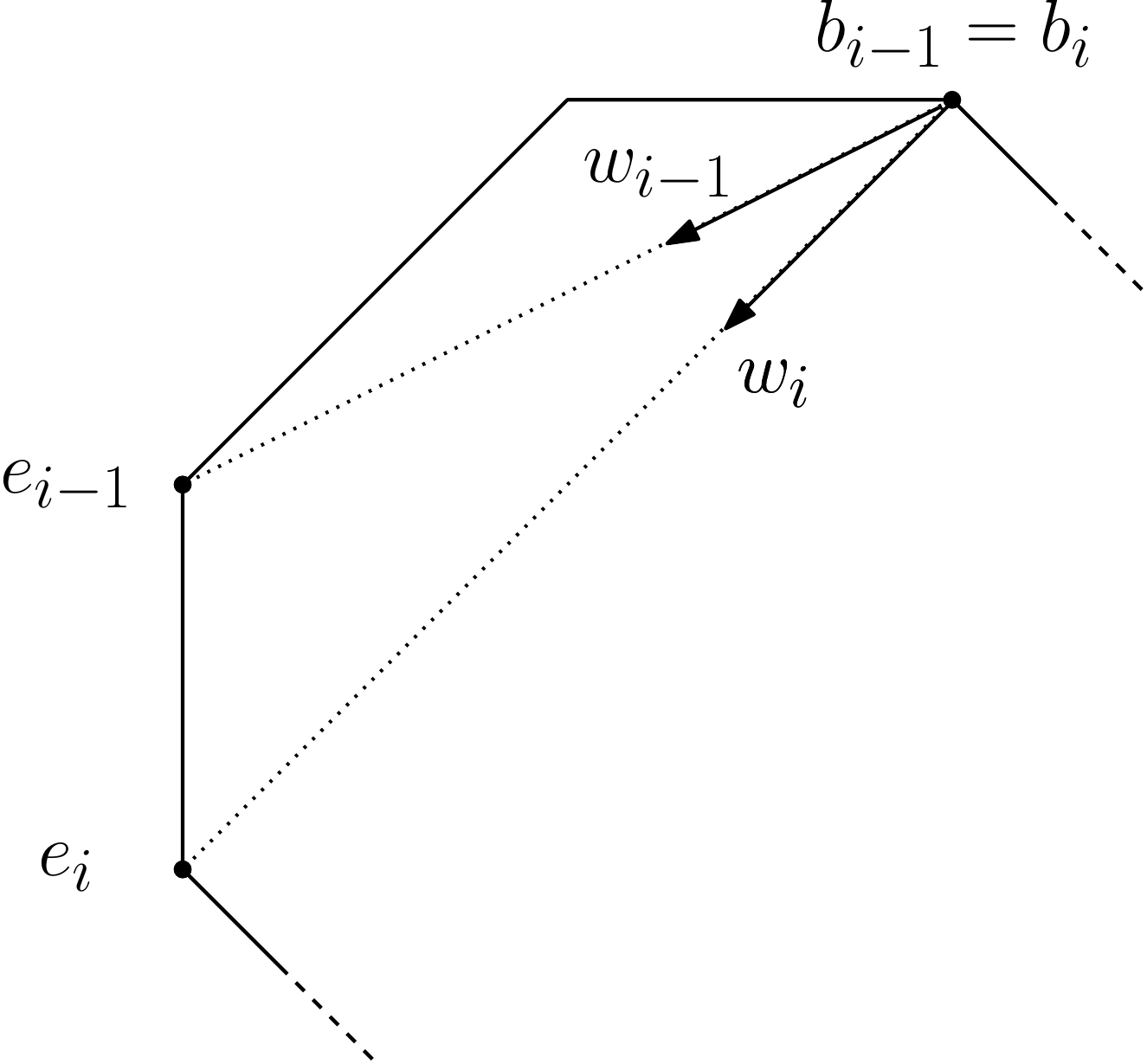}};

       \begin{scope}[xshift=9cm]
       \node[inner sep=0pt] (tangents-1) at (1,0)
 {\includegraphics[width=.27\textwidth]{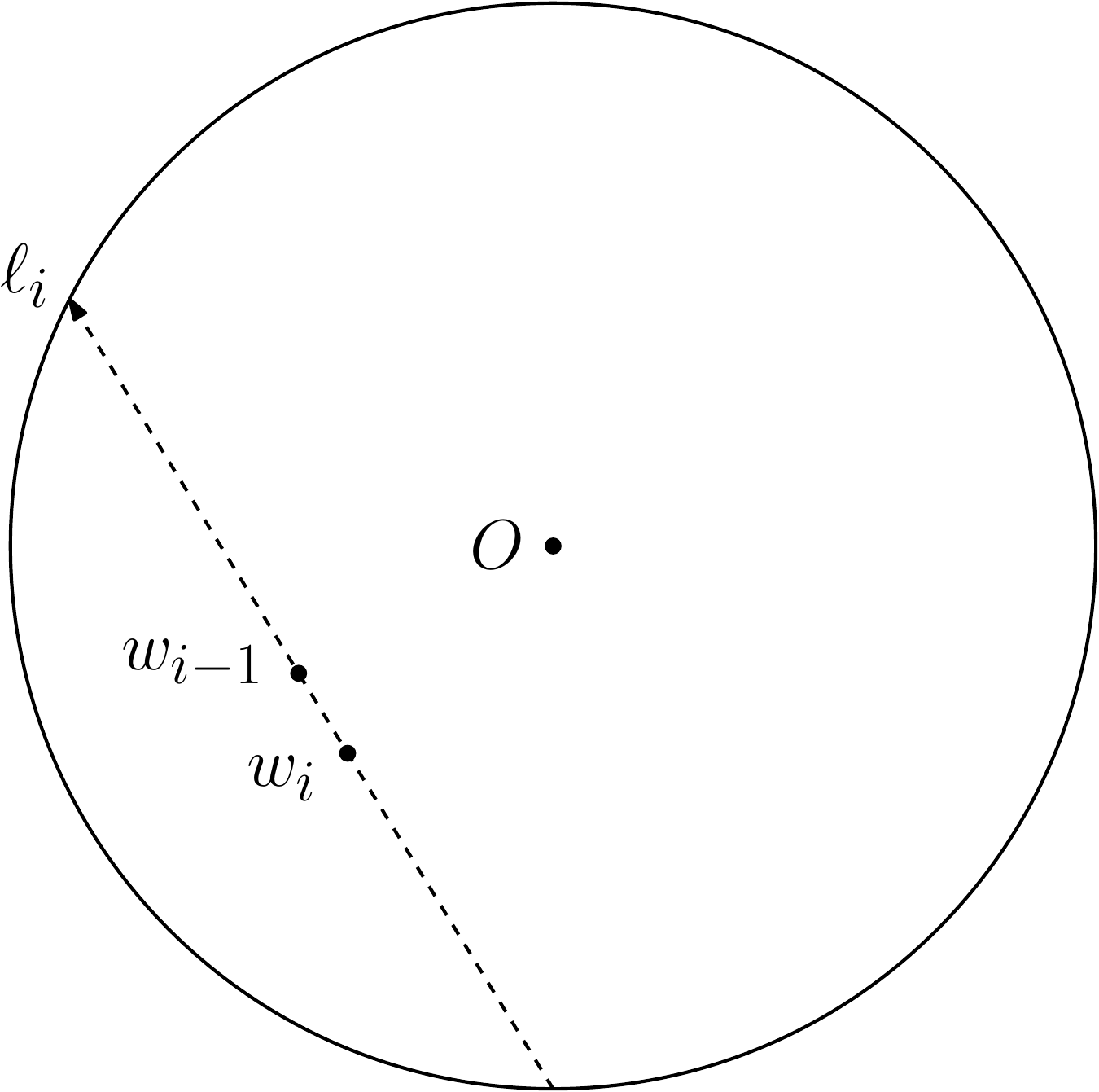}};

 \end{scope}
 \end{tikzpicture}

 \caption{}  \label{fig-local-mono} %\label{fig-Minkowksi-billiard}
 \end{center}
 \end{figure}

Denote, for $1\leq i \leq M$, $w_i = \frac{e_i-b_i}{L_i} \in B$, and consider the closed polygonal line $P$ with vertices $(w_i)_{i=1}^M$, in that order. We will prove that $P$ is a simple closed polygonal line, bounding a convex polygon $C'$, and that $C'$ is in fact a convex $K$-caustic in $B$ dual to $C$. Indeed, consider the invariant circle $\Gamma$ corresponding to $C$. The partition $\partial K = \bigcup_i \cE_i$ induces a partition $\Gamma = \bigcup_i \Gamma_i$, where for each $i$, $\Gamma_i$ corresponds to those oriented lines of $\Gamma$ emanating from points of $\cE_i$. We have observed that $\cE_i$ is an arc of an ellipse with foci $b_i$ and $e_i$ and string length $L_i + |e_i - b_i|$, and $\Gamma_i$ contains those lines emanating from points of $\cE_i$ and 
%positively tangent to $[e_i, b_i]$.
 passing through $e_i$. Therefore, by Corollary \ref{cor:dual-caustic-to-a-segment2}, the image of $\Gamma_i$ under the $(K,B)$-duality map $\alpha$ contains lines emanating from some arc of $\partial B$, which we denote by $B_i$, and passing through $w_i$.

 For each $i$ we consider the point $p_i = B_i \cap B_{i-1} \in \partial B$. Denote by $\ell_i \in \Gamma$ the oriented line emanating from $q_i$, and $\ell_i' = \alpha (\ell_i) \in \alpha(\Gamma)$, which emanates from $p_i$.
 %in the dual invariant circle $\alpha(\Gamma)$ emanating from $p_i$. 
 Note that $\ell_i'$ is in direction $n_K(q_i)$, and that it belongs to both $\alpha(\Gamma_{i-1})$ and $\alpha(\Gamma_i)$, hence it passes through both $w_{i-1}$ and $w_i$, so 
 %. By definition of the $(K,B)$-duality map, it follows that
  $w_{i}-w_{i-1} $ is parallel to the outer normal $n_K(q_i)$. We claim that it is in fact negatively proportional to this normal. Indeed, by Corollary \ref{cor:dual-caustic-to-a-segment2}, the line $\ell_i'$ is negatively tangent to the segment whose end-point is $w_i$, and hence the origin $O$ lies to the right of $\ell_i'$. Additionally, one easily verifies that $(w_{i-1}, w_i)$ forms a positive basis of $\R^2$, since by construction $w_{i-1}$ and $w_{i}$ either share a common initial point $b_i$, in which case  $e_{i-1}$ and $e_i$ are consecutive vertices of $C$ and $b_i$ lies in $C$,  or they share the end point $e_i$ and $b_{i-1}$ and $b_i$ are consecutive points in $C$ (see Figure \ref{fig-local-mono}). It follows easily that $w_{i-1}-w_{i}$ agrees with the orientation of $\ell_i'$, which means that $w_i-w_{i-1}$ is negatively proportional to $n_K(q_i)$ (see Figure \ref{fig-local-mono} again). It follows that the angle of the edges $w_i-w_{i-1}$ is strictly increasing, and increases by $2\pi$, as $i$ runs from $1$ to $M$. Similarly to the smooth case, this implies that the polygonal line $P$ is a simple closed curve, which bounds a convex polygon $C'$. Indeed, simplicity follows as before, since at a self-intersection there must occur one of the following two: (1) at least one of the (inbound or outbound) edges, when elongated to a line, is such that $P$ lies on both of its sides, or (2) the intersection is a segment along an edge; both cases lead to a contradiction with the fact that the angle of edges increases only by $2\pi$ during the journey along $P$. For the convexity one repeats the argument in the smooth case verbatim. 
 
We are left with showing that $C'$ is the dual caustic to $C$. Indeed, by construction the dual invariant circle $\alpha(\Gamma)$ contains lines which pass through the vertices of $C'$. By the convexity of $C'$, this means that $C'$ is everywhere tangent to $\alpha(\Gamma)$. By Corollary \ref{cor:duality-caustic-via-alpha}, this means that $C'$ is a convex caustic dual to $C$, as required.
\end{proof}  

\smallskip \noindent
Shiri Artstein-Avidan \\
School of Mathematical Sciences \\
Tel Aviv University, Tel Aviv 69978, Israel  \\
{\it e-mail}: artstein@post.tau.ac.il \\

\smallskip \noindent
Dan Florentin \\
Department of Mathematical Sciences  \\
Kent State University, Kent, OH, 44242 USA  \\
{\it e-mail}: danflorentin@gmail.com  \\

\smallskip \noindent
Yaron Ostrover \\
School of Mathematical Sciences \\
Tel Aviv University, Tel Aviv 69978, Israel  \\
{\it e-mail}: ostrover@post.tau.ac.il \\

\smallskip \noindent
Daniel Rosen \\
School of Mathematical Sciences \\
Tel Aviv University, Tel Aviv 69978, Israel  \\
{\it e-mail}: danielr6@post.tau.ac.il \\

\end{document}